\documentclass[reqno]{amsart}
\usepackage{amsfonts}
\usepackage{a4wide}
\usepackage{}
\usepackage{amsmath,amssymb,amsthm,amsfonts}
\usepackage{mathrsfs}
\usepackage{bbm}
\usepackage{hyperref}

\usepackage{esint}
\usepackage{calrsfs}

\usepackage{color}

\newtheorem{lemma}{Lemma}[section]
\newtheorem{theorem}{Theorem}[section]
\newtheorem{definition}{Definition}[section]
\newtheorem{proposition}{Proposition}[section]
\newtheorem{remark}{Remark}[section]
\newtheorem{corollary}{Corollary}[section]

\numberwithin{equation}{section}

\newcommand{\be}{\begin{equation}}
\newcommand{\bm}{\begin{multline}}
\newcommand{\ee}{\end{equation}}

\newcommand{\dis}{\displaystyle}

\newcommand{\R}{\mathbb{R}}

\newcommand{\N}{\mathbb{N}}
\renewcommand{\S}{\mathbb{S}}

\newcommand{\FX}{\mathbf{X}}

\newcommand{\CA}{\mathcal{A}}
\newcommand{\CB}{\mathcal{B}}
\newcommand{\CE}{\mathcal{E}}

\newcommand{\CG}{\mathcal{G}}

\newcommand{\CR}{\mathcal{R}}

\newcommand{\bu}{\mathbbm{u}}
\newcommand{\bv}{\mathbbm{v}}

\newcommand{\na}{\nabla}

\newcommand{\al}{\alpha}

\newcommand{\ga}{\gamma}
\newcommand{\om}{\omega}
\newcommand{\Om}{\Omega}
\newcommand{\la}{\lambda}
\newcommand{\de}{\delta}
\newcommand{\si}{\sigma}
\newcommand{\pa}{\partial}
\newcommand{\ka}{\kappa}
\newcommand{\eps}{\epsilon}

\newcommand{\ta}{\theta}

\newcommand{\eqdef}{\overset{\mbox{\tiny{def}}}{=}}

\begin{document}
\title[Compressible Navier-Stokes approximation for the Boltzmann equation]{Compressible Navier-Stokes approximation for the Boltzmann equation in bounded domains}

\author[R.-J. Duan]{Renjun Duan}
%\thanks{}
\address[RJD]{Department of Mathematics, The Chinese University of Hong Kong,
Shatin, Hong Kong, P.R.~China}
\email{rjduan@math.cuhk.edu.hk}

\author[S.-Q. Liu]{Shuangqian Liu}
\address[SQL]{School of Mathematics and Statistics, Central China Normal University, Wuhan 430079, and Department of Mathematics, Jinan University, Guangzhou 510632, P.R.~China}
\email{tsqliu@jnu.edu.cn}

%\thanks{{\sl 2010 MR Subject Classification.} primary 35Q35, 35Q20; secondary 76P05}
%\thanks{{\sl Keywords and phrases.} Conormal derivatives, compressible Navier-Stokes approximation, Chapman-Enskog expansion, diffusive boundary condition}
%\date{May 10, 2014}

\begin{abstract}
It is well known that the full compressible Navier-Stokes equations can be deduced via the Chapman-Enskog expansion from the Boltzmann equation as the first-order correction to the Euler equations with viscosity  and heat-conductivity coefficients of order of the Knudsen number $\eps>0$. In the paper, we carry out the rigorous mathematical analysis  of the compressible Navier-Stokes approximation for the Boltzmann equation regarding the  initial-boundary value problems in general bounded domains. The main goal is to measure the uniform-in-time deviation of the Boltzmann solution with diffusive reflection boundary condition from a local Maxwellian with its fluid quantities given by the solutions to the corresponding compressible Navier-Stokes equations with consistent non-slip boundary conditions whenever $\eps>0$ is small enough. Specifically, it is shown that for well chosen initial data around constant equilibrium states, the deviation weighted by a velocity function is $O(\eps^{1/2})$ in $L^\infty_{x,v}$ and $O(\eps^{3/2})$ in $L^2_{x,v}$ globally in time. The proof is based on the uniform estimates for the remainder in different functional spaces without any spatial regularity. One key step is to obtain the global-in-time existence as well as uniform-in-$\eps$ estimates for regular solutions to the full compressible Navier-Stokes equations in bounded domains
%, which can be viewed as a refinement of the classical result in \cite{Matsumura-Nishida-1983} by Matsumura and Nishida for the case
when the parameter $\eps>0$ is involved in the analysis.
%The result provides a rigorous justification of the compressible Navier-Stokes approximation for the Boltzmann equation in general bounded domains.
%In the paper we study the initial-boundary value problem on the Boltzmann equation in bounded domains with diffusive reflection boundary condition. The main goal is to measure the deviation of the solution from a local Maxwellian with its fluid quantities given by the solutions to the  initial-boundary value problem on the corresponding compressible Navier-Stokes equations with consistent non-slip boundary conditions whenever the Knudsen number $\eps>0$ is small enough. Specifically, we prove that for well chosen initial data, the deviation is $O(\eps^{3/2})$ in $L^\infty$ and $O(\eps^{1/2})$ in $L^2$ uniformly for all time. The result provides a rigorous justification of the compressible Navier-Stokes approximation for the Boltzmann equation in general bounded domains.
\end{abstract}
\maketitle

\thispagestyle{empty}
\setcounter{tocdepth}{1}
\tableofcontents

%\newpage
\section{Introduction}%\label{In}

\subsection{Boltzmann equation}
The problem of ``the limiting processes, there merely indicated, which lead from the atomistic view to the laws of motion of continua" was approached by many groups of mathematicians in kinetic theory, cf.~\cite{Cer-88,CC-90,Gr-63,Ko,So,SBGS}. In this paper, we study the compressible Navier-Stokes  (CNS) hydrodynamic approximation for the Boltzmann equation in bounded domains. The Boltzmann equation in the Euler scaling reads:
\begin{equation}\label{BE}
\pa_t F+v\cdot \na_x F=\frac{1}{\eps}Q(F,F).
\end{equation}
Here, $F=F(t,x,v)$ denotes the density distribution function of the particle gas at time $t\geq0$, position $x\in \Omega$ and velocity $v\in \R^3$, $\Omega$ is a bounded domain in $\R^3$, which can be given by $\{x|\xi(x)<0\}$ with $\xi$ being a smooth function. $\eps>0$ is the Knudsen number defined as the ratio of the molecular mean free path length to a representative physical length scale.
Moreover, let $(v_\ast,v)$ and $(v_\ast', v')$ be the velocities of the particles before and after the collision respectively, which satisfy
\begin{eqnarray*}%\label{v.re}
\left\{\begin{array}{lll}
\begin{split}
&v'=v+[(v_\ast-v)\cdot\omega]\om,\ \ v_\ast'=v_\ast-[(v_\ast-v)\cdot\omega]\om,\ \om\in\S^2,\\
&|v_\ast|^2+|v|^2=|v_\ast'|^2+|v'|^2.\end{split}\end{array}\right.
\end{eqnarray*}
The Boltzmann collision operator $Q(\cdot,\cdot)$ is then given as the following non-symmetric form for hard sphere model
\begin{equation*}%\label{n.op}
\begin{split}
Q(F_1,F_2)=&\int_{\R^3\times\S^2}|(v_\ast-v)\cdot\omega|[F_1(v_\ast')F_2(v')-F_1(v_\ast)F_2(v)]dv_\ast d\omega\\
=&Q_{\textrm{gain}}(F_1,F_2)-Q_{\textrm{loss}}(F_1,F_2).
\end{split}
\end{equation*}
The equation \eqref{BE} is supplemented with the initial data
\begin{equation}\label{ID}
F(0,x,v)=F_0(x,v),\ x\in\Omega,\ v\in\R^3,
\end{equation}
and the diffuse reflection boundary condition
\begin{equation}\label{dbd}
F(t,x,v)|_{n(x)\cdot v<0}=M^w\int_{n(x)\cdot v'>0}F(t,x,v')(n(x)\cdot v')dv',\ x\in\pa\Omega,\ t\geq0,
\end{equation}
where $n(x)=\frac{\na_x\xi}{|\na_x\xi|}$ is the unit outer normal vector and
$M^w$ defined by
\begin{equation}\label{wM}
M^w=\frac{1}{2\pi\ta^2_w}e^{-\frac{|v-u_w|^2}{2\ta_w}}
\end{equation}
is a Maxwellian distribution having the wall temperature $\ta_w$ and wall velocity $u_w$. Through the paper we assume $u_w=0$ and $\ta_w=1$ at the wall. % and the particle density $\rho_w$ on the wall.

\subsection{%Champman-Enskog expansion and full
Compressible Navier-Stokes equations}

In this subsection, we make use of the classical Chapman-Enskog expansion to formally derive the full compressible Navier-Stokes equations.
For this, we first denote a local Maxwellian $M$ by
$$
M(v)=M_{[\overline{\rho},\overline{u},\overline{\ta}]}(v)=\frac{\overline{\rho}(t,x)}{(2\pi\overline{\ta}(t,x))^{3/2}}
e^{-\frac{|v-\overline{u}(t,x)|^2}{2\overline{\ta}(t,x)}},
$$
where $\overline{\rho}$, $\overline{u}$ and $\overline{\ta}$ stand for the arbitrary fluid density, velocity and temperature respectively. With this Maxwellian, one can define the following linear Boltzmann operator:
$$
-L_Mg=Q(M,g)+Q(g,M).
$$
Note that $L_M$ can be split into
$$
L_M=\nu_M-K_M(v,v_\ast),
$$
where $\nu_M$ is a multiplier, given by
$$
\nu_{M}(v)=\int_{\R^3\times \S^2}
|(v-v_\ast)\cdot\omega|M\,dv_\ast d\omega,
$$
and
$K(v,v_\ast)$ is a self-adjoint $L^2$ compact operator, defined by
$$
K_Mg=Q_{\textrm{gain}}(M,g)-Q_{\textrm{loss}}(g,M)+Q_{\textrm{gain}}(g,M).
$$
Clearly, $\nu_M\sim \langle v\rangle$ provided that $\overline{\rho}$ and $\overline{u}$ are bounded, and $\overline{\ta}$ has positive lower and upper bounds.
Moreover, $K(v,v_\ast)$ can be also presented (cf.~\cite{Gla}) as
\begin{eqnarray*}%\label{}
\begin{split}
\left\{\begin{array}{rll}
\begin{split}
K_Mg=&\sqrt{M(v)}
k_{M}\left(\left(\frac{g}{\sqrt{M}}\right)(v)\right),\
k_{M}=k_{2M}-k_{1M},\\
k_{1M}g=&\int_{\R^3\times \S^2}
|(v-v_\ast)\cdot\omega|\sqrt{M(v)}\sqrt{M(v_\ast)}g(v_\ast)\,dv_\ast d\omega,\\
k_{2M}g=&\int_{\R^3\times \S^2}
|(v-v_\ast)\cdot\omega|\sqrt{M(v_\ast)}\left\{\sqrt{M(v')}g(v'_\ast)+\sqrt{M(v_\ast')}g(v')\right\}\,dv_\ast d\omega.
\end{split}
\end{array}\right.
\end{split}
\end{eqnarray*}

We now introduce a correction term
\begin{equation}\label{def.G}
\begin{split}
%G=
G(\overline{\rho},\overline{u},\overline{\ta})=-L^{-1}_MM\left\{\frac{1}{2}A(V):\si(\overline{u})
+B(V)\cdot\frac{\na_x\overline{\ta}}{\sqrt{\overline{\ta}}}\right\},
\end{split}
\end{equation}
with $$\si(\overline{u})=\na_x \overline{u}+(\na_x \overline{u})^T-\frac{2}{3}\na_x\cdot \overline{u}I,$$
%$$
%A(V)=V\otimes V-\frac{1}{3}|V|^2I, \ \ B(V)=\frac{|V|^2-5}{2}V,\ \ V=\frac{v-u}{\sqrt{\ta}}.
%$$
and a higher order dissipation term
$$
%H=
H[\overline{\rho},\overline{u},\overline{\ta}]=MV\cdot\frac{\na_x\cdot[\mu(\overline{\ta})\sigma(\overline{u})]}{\overline{\rho}\sqrt{\overline{\ta}}}
+M\left(\frac{1}{3}|V|^2-1\right)\frac{\frac{1}{2}\mu(\overline{\ta})\sigma(\overline{u}):\sigma(\overline{u})
+\na_x\cdot[\ka(\overline{\ta})\na_x\overline{\ta}]}{\overline{\rho}\overline{\ta}},
$$
where
\begin{eqnarray*}%\label{Bur.fun.}
\left\{
\begin{array}{rllll}
\begin{split}
A(V)&=V\otimes V-\frac{1}{3}|V|^2I,\ \ B(V)=\frac{|V|^2-5}{2}V,\ \ V=\frac{v-\overline{u}}{\sqrt{\overline{\ta}}},\\
\mu(\overline{\ta})&=\frac{\overline{\ta}}{10}\int_{{\R}^3}A(V):L^{-1}_{M_{[1,\overline{u},\overline{\ta}]}}
A(V)M_{[1,\overline{u},\overline{\ta}]}>0,\\
\ka(\overline{\ta})&=\frac{\overline{\ta}}{3}\int_{{\R}^3}B(V)\cdot L^{-1}_{M_{[1,\overline{u},\overline{\ta}]}}
B(V)M_{[1,\overline{u},\overline{\ta}]}>0.
\end{split}
\end{array}
\right.
\end{eqnarray*}
Let us now take $M=M_{[\rho,u,\ta]}$ where $[\rho,u,\ta]$ satisfy
\begin{eqnarray}\label{NS}
\left\{\begin{array}{rlll}
\begin{split}
&\pa_t\rho+\na_x\cdot(\rho u)=0,\\
&\rho(\pa_tu+u\cdot \na_xu)+\na_x P=\eps\na_x\cdot\left[\mu(\ta)\sigma(u)\right],\ P=\rho\ta,\\
&\frac{3}{2}\rho(\pa_t\ta+u\cdot\na_x\ta)+P\na_x\cdot u
=\eps\na_x\cdot\left[\ka(\ta)\na_x\ta\right]
+\frac{\eps}{2}\mu(\ta)\sigma(u):\sigma(u).
\end{split}
\end{array}\right.
\end{eqnarray}
We then set
\begin{equation}\label{epn}
F=M_{[\rho,u,\ta]}+\eps G(\rho,u,\ta)+\eps^{3/2}R,
\end{equation}
and plug this into \eqref{BE} to obtain the equation for $M$:
\begin{equation}\label{M}
\pa_t M+v\cdot\na_x M+L_MG=\eps H,
\end{equation}
and the remainder equation for $R$:
\begin{equation}\label{R}
\begin{split}
\pa_tR+&v\cdot\na_xR+\frac{1}{\eps}L_MR\\
=&\eps^{1/2}Q(R,R)+Q(R,G)+Q(G,R)+\eps^{-1/2}Q(G,G)-\eps^{-1/2}(\pa_tG+v\cdot \na_xG+H).
\end{split}
\end{equation}
It is straightforward to check that \eqref{M} is equivalent to the full CNS %compressible Navier-Stokes
equations \eqref{NS}.

To solve \eqref{NS}, we impose the initial data
\begin{equation}\label{NSid}
[\rho,u,\ta](0,x)=[\rho_0,u_0,\ta_0](x),
\end{equation}
and the Dirichlet boundary condition
\begin{equation}\label{NSbd}
{[u,\ta]\big|_{\pa\Omega}=[0,1]}.
\end{equation}
%Note that \eqref{NSbd} is matched with the velocity and temperature of the wall cf.\eqref{wM}, which retrains from the appearance of the Knudsen layer.

\subsection{Main results}
Define $$
M_-=M_{[1,0,1]}=\frac{1}{(2\pi)^{3/2}}e^{-\frac{|v|^2}{2}}.
$$
We now state our main results as follows:

\begin{theorem}\label{mr}
Let $w_\ell=\langle v\rangle^{\ell}=(1+|v|^2)^{\ell/2}$ with $\ell>3/2$, assume $F_0=M_{[\rho_0,u_0,\ta_0]}$,
there exists a constant $\ka_0>0$, such that
if
\begin{equation}
\label{mr.conid}
\|[\rho_0-1,u_0,\ta_0-1]\|_{\FX_\eps}\leq \ka_0\eps,
\end{equation}
with the norm $\|\cdot\|_{\FX_\eps}$ given by \eqref{Xe},
then
the initial boundary value problem \eqref{BE}, \eqref{ID} and \eqref{dbd} on the Boltzmann equation admits a unique global solution
$$
F(t,x,v)=M_{[\rho,u,\ta]}+\eps G+\eps^{3/2} R\geq0, \quad (t,x,v)\in [0,\infty)\times \Omega\times \R^3,
$$
where $[\rho,u,\ta]$ is %determined by
is a unique global solution obtained in Theorem \ref{NSsol} for the initial boundary value problem \eqref{NS}, \eqref{NSid} and \eqref{NSbd} on the full compressible Navier-Stokes equations, $G=G(\rho,u,\ta)$ with $G(\cdot,\cdot,\cdot)$ defined in \eqref{def.G},  and the remainder $R$ satisfies
\eqref{R}. Moreover,
it holds that
\begin{equation}\label{diff}
\begin{split}
\sup_{t\geq 0}\left\|\frac{F-M_{[\rho,u,\ta]}}{\sqrt{M_-}}\right\|_2\lesssim \eps^{3/2},\quad \sup_{t\geq 0}\left\|\frac{w_\ell\left(F-M_{[\rho,u,\ta]}\right)}{\sqrt{M_-}}\right\|_\infty\lesssim \eps^{1/2} .
\end{split}
\end{equation}
\end{theorem}

\begin{remark}
It is a common view  that the temporal derivatives of the initial datum such as $\pa_t R_0$ and $\pa_t[\rho_0,u_0,\ta_0]$ are determined by the time
evolution equations of $R$ and $[\rho,u,\ta]$ respectively, for instance, $\pa_t R_0$ is understood by the equation \eqref{R}.
%\begin{equation*}
%\begin{split}
%\pa_t R_0=&-v\cdot\na_xR_0-\frac{1}{\eps}L_MR_0
%+\eps^{1/2}Q(R_0,R_0)+Q(R_0,G_0)+Q(G_0,R_0)\\&+\eps^{-1/2}Q(G_0,G_0)-\eps^{-1/2}(\pa_tG_0+v\cdot \na_xG_0+H_0),
%\end{split}
%\end{equation*}
%where $G_0=G(\rho_0,u_0,\ta_0)$ and $H_0=H(\rho_0,u_0,\ta_0)$.

\end{remark}

\begin{remark}\label{mr.rm2}
Note that \eqref{NSbd} is matched with \eqref{wM} for the velocity $u_w=0$ and temperature $\ta_w=1$ at the wall, and also the initial data $F_0$ is chosen to be a local Maxwellian $M_{[\rho_0,u_0,\ta_0]}$ whose fluid quantities $[\rho_0,u_0,\ta_0]$  is consistent with initial data for the CNS equations and close to the constant state $[1,0,1]$ up to order of $\eps$ in terms of the norm $\FX_\eps$. Therefore, these restrictions retrain from the appearance of the Knudsen layer.
\end{remark}

%\Red{To ADD: Literature}

The result of Theorem \ref{mr} clarifies how close to each other in $L^2$ and $L^\infty$ settings with respect to the small Knudsen number $\eps>0$ the Boltzmann solution and the CNS solution are.
Although it is well known that the Boltzmann equation can be approximated by the CNS equations via the Chapman-Enskog expansion up to first order (cf.~\cite{Cer-88,CC-90}), it seems that  Theorem \ref{mr} provides the first rigorous justification to this global-in-time approximation regarding the IBVPs of kinetic and compressible fluid equations in general bounded domains. In the mean time, we have also obtained in Theorem \ref{NSsol} the global-in-time existence as well as uniform-in-$\eps$ estimates for regular solutions to the CNS equations in bounded domains, and this result has its own interest even in the context of the compressible fluid dynamical equations.

\subsection{Literature}
The hydrodynamic limit or approximation of the Boltzmann equation is a fundamental topic in kinetic theory. There have been extensive mathematical studies on this subject. We may refer readers to the Cercignani's book \cite[e.g.~Chapter V]{Cer-88} for the history of the problem, to the survey book \cite{SR} by Saint-Raymond for the detailed presentation of limits of the Boltzmann equation to either the incompressible Navier-Stokes and Euler equations or the compressible Euler equations, and to two recent works \cite{EGKM-15,LYZ} as well as references therein for more up-to-date results on the topic. In what follows we review some existing works in the literature that are most related to the topic of this paper in connection with the CNS approximation to the Boltzmann equation.

First of all, for the IBVP on the Boltzmann equation \eqref{BE} in bounded domains, the unique global-in-time solution around global Maxwellians in the $L^\infty$ setting has been constructed by Guo \cite{Guo-2010} for different types of boundary conditions including the diffusive-reflection condition \eqref{dbd} and for angular cutoff hard potentials. As for the IBVP \eqref{NS} and \eqref{NSbd} on the corresponding CNS in an unbounded domain which is either the half-space or exterior to a bounded domain, the global existence of smooth solutions around constant states was established by Matsumura and Nishida \cite{Matsumura-Nishida-1983} through the energy method developed in their previous pioneering  work \cite{MN-80} for the whole space. The case of the bounded domain has been studied in \cite{MN-82}; see also \cite{VZ} for the problem on existence and stability of periodic and stationary solutions.

For the CNS approximation to the Boltzmann equation, there have been a few mathematical results in different situations, for instance, \cite{DE-96,ELM-94,ELM-95,KMN, La-92,LYZ}; see also \cite{BLM} and references therein for the numeric method solving the Boltzmann equation by the CNS approximation. Specifically, Kawashima, Matsumura and Nishida \cite{KMN} proved that the Boltzmann solution is asymptotically equivalent in large time to the CNS solution for the Cauchy problems in the whole space when initial data for both problems are sufficiently close to constant equilibrium states in smooth Sobolev spaces. In spirit of a previous work \cite{La-87}, Lachowicz \cite{La-92}  initiated a study of the CNS approximation including the initial layer to the Boltzmann equation in spatially periodic domains. Without considering the initial layer, the result of \cite{La-92} has been recently extended by the second author of the paper together with Yang and Zhao \cite{LYZ}  to the more general situations. The main differences between \cite{LYZ} and \cite{La-92}  have been pointed out in Remark 1.1 in the latter paper. Particularly, the error estimates obtained in the former is uniform in all time $t>0$ and also for small enough $\eps>0$, and the CNS solution is proved to exist and satisfy some properties required for the analysis of the remainder equation. One of the main goals in this paper is to make use of the techniques developed in \cite{Masmoudi-Rousset-2017,Wa-16}  in order to extend the result in \cite{LYZ} related to the CNS in periodic domains to the case of general bounded domains, see Theorem \ref{NSsol} in Section \ref{sec2} as mentioned before.

In the one-dimensional case, Esposito, Lebowitz and Marra \cite{ELM-94} studied the steady solution of the Boltzmann equation in a slab with a constant external force of order $\eps$ parallel to the boundary and with complete accommodation at the walls, and they showed that for any small enough force, the Boltzmann solution can be approximated by the steady CNS solution with non-slip boundary condition up to $3/2$ order of $\eps$ in $L^\infty$ norm. The result of \cite{ELM-94} was extended later by the same authors \cite{ELM-95} to allow the temperature difference on two ends of the slab and also by Di Meo and Esposito \cite{DE-96} to the case of general hard potentials. It could be interesting to further extend those results \cite{DE-96,ELM-94,ELM-95} to the steady problem on the Boltzmann equation in general bounded domains. Note that in the time-dependent situation under consideration of this paper, the large-time behavior of the Boltzmann solution is trivial and hence it is hard to see if the developed approach in the paper can be directly adopted to treat the CNS approximation to the stationary Boltzmann equation with non-trivial solutions in some physical situations such as with a variable wall temperature \cite{EGKM-13} or a small external force \cite{ELM-94,EGKM-15}.

Another important issue for justifying the more accurate CNS approximation in bounded domains is concerned with the suitable choice of the boundary conditions for the CNS equations \eqref{NS}. In the current work, the non-slip boundary condition \eqref{NSbd} has been used. However, to provide the correct overall solution to the Boltzmann equation, it is generally common to supplement the CNS equations with the slip boundary conditions, cf.~\cite{Cer-88,Ko}. Derivation of slip boundary conditions for the CNS equations from the Boltzmann equation at the kinetic level was first made by Coron \cite{Co} with the explicit computations of slip coefficients. This problem has been very recently revisited by Aoki et al.~\cite{ABHK} in a systematic way. Particularly, those slip boundary conditions are essentially the consequence of the analysis of the Knudsen layer. Interested readers may also refer to \cite{AKFG} for the application of their approach to a specific problem. It should be an interesting and challenging problem to extend the current result to the case with such slip boundary conditions for the CNS equations by introducing the extra Knudsen layer correction around the boundary, cf.~\cite{Ca,SBGS}, for instance.

Lastly, as related to the current work, we again mention \cite{EGKM-15} for the hydrodynamic limit to the incompressible Navier-Stokes equations for the stationary Boltzmann equation in bounded domains in the presence of a small external force field and a small boundary temperature variation for the diffusive reflection boundary condition. In this paper we shall adopt some techniques developed in \cite{EGKM-15} (see also \cite{Gu16}) to treat the uniform estimates on the remainder.

%\newpage
\subsection{Strategy of the proof}

For convenience, we sketch the proof of Theorem \ref{mr} in the following two parts basing on solving two initial boundary value problems on the CNS equations and the remainder equation.

\medskip
\noindent{\it Part I. Global existence and uniform regularity of the CNS solution $[\rho,u,\ta]$}. Regarding global existence of classical solutions to the full CNS equations around constant states, we have mentioned the known results \cite{MN-80,MN-82,Matsumura-Nishida-1983,VZ} in the case of $\R^3$, half-space, exterior domain or bounded domain.
% has been investigated in \cite{Matsumura-Nishida-1983}, so far there are few results on existence of strong or classical solutions to the full CNS equations around constant states in bounded domains. Particularly,
However, for the problem in bounded domains, when the viscosity and heat-conductivity coefficients depend on the parameter $\eps$, it is non-trivial to obtain the uniform-in-$\eps$ high-order regular solution globally in time, due to the supplemented non-slip boundary condition as well as the weak hyperbolic-parabolic dissipation structure of the system. To overcome such difficulty, we shall employ the techniques in \cite{Masmoudi-Rousset-2017} and \cite{Wa-16} by introducing the energy norm $\FX_\eps$ with the conormal derivatives with respect to time and space variables. Compared to the incompressible setting where the regularity of solutions can be gained via the elliptic estimates (cf.~\cite{Masmoudi-Rousset-2017}), we have to additionally treat the uniform estimates on the hyperbolic component $\rho$ in the compressible situation. Moreover, different from the isentropic case, appearance of the temperature equation in the full CNS system induces an extra difficulty, cf.~\cite{Wa-16}.

Precisely, we are able to derive the following uniform estimates
\begin{equation*}%\label{apes}
\begin{split}
\sup\limits_{0\leq s\leq t}&
\sum\limits_{\al_0\leq 3}\left\|\pa_t^{\al_0}[\rho-1,u,\ta-1](s)\right\|_2^2
+\sup\limits_{0\leq s\leq t}\sum\limits_{\al_0\leq 2}\left\|\pa_t^{\al_0}\na[\rho,u,\ta](s)\right\|_2^2
\\&+\sup\limits_{0\leq s\leq t}\sum\limits_{\al_0\leq 2}\eps^2\left\|\pa_t^{\al_0}\na^2\rho(s)\right\|_2^2
+\sup\limits_{0\leq s\leq t}\sum\limits_{\al_0\leq 1}\eps^2\left\|\pa_t^{\al_0}\na^2[u,\ta](s)\right\|_2^2
\\&+\sup\limits_{0\leq s\leq t}\sum\limits_{\al_0\leq 3,|\al|=2}\eps\left\|\pa_t^{\al_0}Z^{\al}\left[\widetilde{\rho},u,\widetilde{\ta}\right](s)\right\|_2^2
+\sup\limits_{0\leq s\leq t}\sum\limits_{\al_0\leq 2}\eps^2\left\|\pa_t^{\al_0}\na_x[\rho,u,\ta](s)\right\|_{H_{co}^{2}}^2
\\&+\sup\limits_{0\leq s\leq t}\sum\limits_{\al_0\leq 2}\eps^4\left\|\pa_t^{\al_0}\na^2\rho\right\|_{H_{co}^{2}}^2
+\sup\limits_{0\leq s\leq t}\sum\limits_{\al_0\leq 1}\eps^4\left\|\pa_t^{\al_0}\na^2[u,\ta]\right\|_{H_{co}^{2}}^2,
\end{split}
\end{equation*}
under the a priori assumption $N(\rho,u,\theta)\lesssim \eps^2$ with
\begin{equation*}
\begin{split}
N(\rho,u,\ta)(t)\eqdef &\|[\rho-1,u,\ta-1](t)\|^2_{\FX_{\eps}}
 +\eps\sum\limits_{\al_0\leq3}\int_0^t\|\pa_t^{\al_0}\na_x[u,\ta]\|^2_{2}ds
 +\eps\sum\limits_{\al_0\leq2}\int_0^t\|\pa_t^{\al_0}\na_x \rho\|^2_{2}ds
 \\&+\eps\sum\limits_{1\leq\al_0\leq3}\int_0^t\|\pa_t^{\al_0}\rho\|^2_{2}ds
 +\eps\sum\limits_{\al_0\leq2}\int_0^t\|\pa_t^{\al_0}\na^2_x[\rho,u,\ta]\|^2_{2}ds
+\eps^2\sum\limits_{1\leq\al_0\leq3}\int_0^t\|\pa_t^{\al_0}\rho\|^2_{H^2_{co}}ds
  \\& +\eps^2\sum\limits_{\al_0\leq 2}\int_0^t\left\|\pa_t^{\al_0}\na_x\widetilde{\rho}(s)\right\|_{H^2_{co}}^2ds
 +\eps^2\sum\limits_{\al_0\leq 3}\int_0^t\left\|\pa_t^{\al_0}\na_x\left[u,\widetilde{\ta}\right](s)\right\|_{H^2_{co}}^2ds
 \\&+\eps^3\sum\limits_{\al_0\leq 2}\int_0^t\left\|\pa_t^{\al_0}\na^2_x\left[\widetilde{\rho},u,\widetilde{\ta}\right](s)\right\|_{H^2_{co}}^2ds
+\eps^3\sum\limits_{\al_0\leq1}\int_0^t\|\pa_t^{\al_0}\na^3_x[u,\ta]\|_{H^2_{co}}^2ds.
\end{split}
\end{equation*}
Therefore, we have to require the condition \eqref{mr.conid}.  As mentioned in Remark \ref{mr.rm2}, to remove such restriction could make it necessary to include the analysis of the Knudsen layer as well as the possible switch of the non-slip boundary condition to the appropriate slip boundary condition as suggested in \cite{ABHK}.  %In those norms above, we require $m_0\geq 6$ which is necessary to derive the estimates on source terms for the remainder equation.

To deduce the above uniform estimates, we first obtain the zero-order energy estimate \eqref{0engp2} which also involves the pure time derivatives up to the third order. The energy estimate \eqref{1drho} for the first-order spatial derivatives of $[\rho,u,\theta]$ are subtle to obtain, and they are treated in a special way by the Galerkin method together with the Helmholtz decomposition (see \eqref{def.helmholtz} as well as Lemma \ref{Hdec}) and elliptic estimates. Similarly, we also control the energy norm $\|\nabla_x^2\rho\|_2$ as in \eqref{2drhop2} with the coefficient $\eps^2$. In the last step, we perform the same energy estimates by first acting the conormal differentiation $Z^\alpha$ $(|\al|\leq 2)$ and conclude the proof.

\medskip
\noindent{\it Part II. Global existence and uniform estimates of the remainder $R$}. We extend the main ideas of \cite{EGKM-15,Gu-06,Gu16} for treating the incompressible hydrodynamic limit to the setting of the compressible fluid approximation under consideration. Recall \eqref{epn} for the expansion of the Boltzmann solution $F$. As the background profile is a local Maxwellian $M_{[\rho,u,\ta]}$ with $[\rho,u,\ta]$ chosen as the CNS solution, the analysis for obtaining the uniform estimates on the remainder $R$ is hard to carry out, for instance, the linearized Boltzmann operator is around $M_{[\rho,u,\ta]}$, no longer a global Maxwellian. We may adopt ideas in \cite{LYY,Liu-Yang-Yu-Zhao-2006} with the macro-micro decomposition to treat the energy estimates. Moreover, the technique for $L^\infty$ estimates in terms of the characteristic approach as in \cite{Guo-2010} has to be modified to consider the effect of the local Maxwellian, see also \cite{GJ}.

Before going to the details of the proof, we explain a little the reason of choosing the expansion \eqref{epn}, particularly $3/2$-order of $\eps$ as a coefficient of the remainder $R$. In fact, the first two terms $M_{[\rho,u,\ta]}+G$ is natural, corresponding to the Chapman-Enskop solution for the CNS system \eqref{NS}. A general form for the remainder should be taken as $\eps^\beta R$ with $\beta>1$ to be determined so that the Boltzmann solution can be approximated by the CNS solution up to $O(\eps^\beta)$ in $L^\infty$ setting. Once we plug the ansatz into the Boltzmann equation \eqref{BE}, we obtain
\begin{equation*}%\label{R}
\begin{split}
\pa_tR+&v\cdot\na_xR+\frac{1}{\eps}L_MR\\
=&\eps^{\beta-1}Q(R,R)+Q(R,G)+Q(G,R)+\eps^{1-\beta}Q(G,G)-\eps^{1-\beta}(\pa_tG+v\cdot \na_xG+H),
\end{split}
\end{equation*}
with the boundary condition $R_-=P_\ga R+\eps^{1-\beta} [P_\ga G-G]$. Hence, the nonlinear term $\eps^{\beta-1}Q(R,R)$ is singular for small $\beta$ while the inhomogeneous source term $\eps^{1-\beta}(\pa_tG+v\cdot \na_xG+H)$ is singular for large $\beta$. The interplay between different function spaces to bound $R$ determines that $\beta=3/2$ is a suitable choice for obtaining the uniform-in-$\eps$ estimates. This in turn yields that the boundary source term $\eps^{1-\beta} [P_\ga G-G]$ must be singular in $\eps$. Fortunately, from the proof the boundary singularity turns out to be controlled.

Now we shall sketch the proof of Theorem \ref{mr}, particularly the validity of the uniform error estimate \eqref{diff} implying that the norm $\|R/{M}_-^{1/2}\|_{L^\infty}$ should be singular of order $\eps^{-1}$. Keep in mind from the natural energy inequality that
$$
\int_0^t \left\|\frac{P_1^M R}{\sqrt{M}}\right\|_2^2ds\sim O(\eps),
$$
and further in terms of the formal elliptic estimates that
\begin{equation}
\label{ad.mawd}
\int_0^t \left\|\frac{P_0^M R}{\sqrt{M}}\right\|_2^2ds\sim O(\eps^{-1}),
\end{equation}
namely, the time integration of the micro energy dissipation is of order $\eps$ while the time integration of the macro energy dissipation is singular of order $\eps^{-1}$. Starting from estimates on the pointwise bound of $R(t,x,v)$ as in \cite{EGKM-15}, we can make use of the technique of the $L^2$-$L^\infty$ interplay. However, even at the linear level, if the pure $L^2$ norm was used as an upper bound, such $L^2$ bound involves a strong singularity in $\eps$, see \eqref{upl6-l3}. Instead, we may decompose $R=P_0^MR+P_1^MR$ and rearrange the upper bound as the combination of $L^6$ norm of the macro component $P_0^MR$ and $L^2$ norm of the micro component $P_1^MR$ distributed by the different orders of $\eps$ in the form of
$$
\eps \sup_{0\leq s\leq t} \left\|\frac{R(s)}{\sqrt{M}}\right\|_\infty\lesssim \eps^{1/2} \sup_{0\leq s\leq t} \left\|\frac{P_0^M R(s)}{\sqrt{M_-}}\right\|_6+\eps^{-1/2} \sup_{0\leq s\leq t} \left\|\frac{P_1^M R(s)}{\sqrt{M_-}}\right\|_2+\cdots,
$$
see \eqref{upl2} where we have multiplied the inequality by $\eps$. Here, the second term on the right can be estimated as
$$
\eps^{-1} \sup_{0\leq s\leq t} \left\|\frac{P_1^M R(s)}{\sqrt{M_-}}\right\|_2^2\leq \eps^{-1}  \left\|\frac{P_1^M R(0)}{\sqrt{M_-}}\right\|_2^2+2\eps^{-1}\int_0^t\left\|\frac{P_1^M R(s)}{\sqrt{M_-}}\right\|_2^2\left\|\frac{\pa_tP_1^M R(s)}{\sqrt{M_-}}\right\|_2^2 ds,
$$
provided that the right-hand time integral is suitably controlled by $O(1)$. To deal with the $L^6$ norm of $P_0^MR$, we use the dual argument as in \cite{Gu16}. Formally, by the Sobolev's inequality $W^{2,6/5}\hookrightarrow L^2$ and elliptic estimates, we can show that
\begin{equation}
\label{ad.L6}
 \eps^{1/2} \sup_{0\leq s\leq t} \left\|\frac{P_0^M R(s)}{\sqrt{M_-}}\right\|_6\lesssim \eps^{-1/2} \sup_{0\leq s\leq t} \left\|\frac{P_1^M R(s)}{\sqrt{M_-}}\right\|_2+\eta \eps  \sup_{0\leq s\leq t}  \left\|\frac{R(s)}{\sqrt{M_-}}\right\|_\infty+\cdots,
\end{equation}
see \eqref{P1-P6}. It remains to bound the energy norm of $R$ at the nonlinear level. To do so, the most key point is to treat the trilinear term
$$
\eps^{1/2} \int_0^t \left(Q(R,R),\frac{P_1^M R}{\sqrt{M_-}}\right)ds.
$$
To bound the above term from the macro contribution, instead of the usual way by the product of $L^2$ and $L^\infty$ norms which actually fails due to the strong $\eps$-singularity as in \eqref{ad.mawd}, we use the $L^6$-$L^3$ estimates to bound it by
$$
\eta \eps^{-1}\int_0^t \left\|\frac{P_1^M R}{\sqrt{M_-}}\right\|_2^2ds+C_\eta \eps^2\int \left\|\frac{P_0^M R}{\sqrt{M_-}}\right\|_6^2 \left\|\frac{P_0^M R}{\sqrt{M_-}}\right\|_3^2ds.
$$
Recall by \eqref{ad.L6} that the weighted $L^6$ norm of $P_0^M R$ is of order $\eps^{-1/2}$. Thus, one has to verify that
$$
\int_0^t \left\|\frac{P_0^M R}{\sqrt{M_-}}\right\|_3^2ds\sim O(\eps^{-1}).
$$
The rigorous proof for the above estimate is based on the velocity average lemma. In the case of the whole space of three dimensions, $L^3$ bound is a  consequence of the critical Sobolev imbedding $H^{1/2}\hookrightarrow L^3$. However, we have to treat additional difficulties for the bounded domain.

%\Red{To ADD: Idea of the proof}

%\subsection{Notations %and Norms
%}

\subsection{Organization of the paper}

The rest of this paper is organized as follows. In Section \ref{sec2} we study the initial-boundary value problem on the CNS system \eqref{NS}, \eqref{NSid} and \eqref{NSbd}. In Sections 3, 4, and 5 we make series of estimates on the remainder $R$. Section 6 is devoted to proving Theorem \ref{mr}. Some basic estimates used in those sections are collected in the appendix in Section 7.

\medskip
\noindent{\it Notations.}
We now list some notations  used in the paper.
 \begin{itemize}
 \item
 Throughout this paper,  $C$ denotes some generic positive (generally large) constant and $\la$ denotes some generic positive (generally small) constants, where $C$ and $\la$  may take different values in different places. $D\lesssim E$ means that  there is a generic constant $C>0$
such that $D\leq CE$. $D\sim E$
means $D\lesssim E$ and $E\lesssim D$.
\item Let $1\leq p\leq \infty$, we denote $\Vert \,\cdot \,\Vert _{p }$ the $L^{p }(\Omega
\times \R^{3})-$norm or the $L^{p }(\Omega )-$norm or $L^{p }(\Omega\cup\ga )-$norm,
sometimes, we use $|\,\cdot \,|_{\infty }$ to denote either the $L^{\infty }(\partial \Omega
\times \R^{3})-$norm or the $L^{\infty }(\partial \Omega )-$norm at
the boundary. Moreover, %we denote $\|\cdot \|_{\nu}\equiv \|\nu^{1/2}\cdot\|_2$, and
$(\cdot,\cdot)$ denotes the $L^{2}$ inner product in
$\Omega\times {\R}^{3}$  with
the $L^{2}$ norm $\|\cdot\|_2$ and $\langle\cdot\rangle$ denotes the $L^{2}$ inner product in $\R^3_v$.
%and $\|f\|_{H^k}= \|f\|_{2}+ \sum_{i=1}^{k}\|\nabla_x^i f \|_{2}$.

\item As to
the phase boundary integration, we denote $d\gamma = |n(x)\cdot v|dS_xdv$,
where $dS_x$ is the surface element and for $1\leq p<+\infty$, we define $|f|_p^p = \int_{\gamma}
|f(x,v)|^p d\gamma \equiv\int_{\gamma} |f(x,v)|^p $ and the corresponding
space as $L^p(\partial\Omega\times\R^3;d\gamma)=L^p(\partial\Omega%
\times\R^3)$. Furthermore $|f|_{p,\pm}= |f \mathbf{1}_{\gamma_{\pm}}|_p$
and $|f|_{\infty,\pm}= |f \mathbf{1}_{\gamma_{\pm}}|_{\infty}$. %For simplicity,  we use $%
%|f|_p^p = \int_{\partial\Omega} |f(x)|^p dS_x\equiv\int_{\partial\Omega}
%|f(x)|^p $.
We also denote $f_{\pm}=f_{\gamma_{\pm}}=f\mathbf{1}_{\gamma_{\pm}}$ and $f_\ga=f\mathbf{1}_{\gamma}$.

%\item Finally, we define
%\begin{equation*}
%P_{\gamma }f(x,v)=\sqrt{\mu (v)}\int_{n(x)\cdot v^{\prime }>0}f(x,v^{\prime
%})\sqrt{\mu (v^{\prime })}(n(x)\cdot v^{\prime })dv^{\prime }, \ \ x\in\pa\Om.
%\label{pgamma}
%\end{equation*}%
%Thanks to \eqref{mu.n}, $P_{\gamma }f$ defined on $\pa\Om\times\R^3$, %viewed as function on $\{v\in
%\R^{3}: \ v\cdot n(x)>0\}$ for any fixed $x\in \partial \Omega $,
%is an $L_{v}^{2}$-projection with respect to the measure $|n(x)\cdot v|$ for
%any boundary function $f$ defined on $\gamma _{+}$. We also denote $\{I-P_\ga\}f=f-P_\ga f$.
\end{itemize}

\section{Solutions of the compressible Navier-Stokes equations}\label{sec2}

This section is devoted to obtaining the existence and most importantly to deducing the higher regularity of the solutions of the compressible Navier-Stokes system \eqref{NS}, \eqref{NSid} and \eqref{NSbd}.

It should be pointed out that it is extremely difficult to obtain the uniform higher regularity of the solutions of the system \eqref{NS}, \eqref{NSid} and \eqref{NSbd} due to the weak dissipation on the right hand side and the non-slip boundary condition, which is quite different from the incompressible case, where the standard elliptic estimates can be directly adopted to gain the regularity of the solutions, cf. \cite{Masmoudi-Rousset-2017}. To settle this problem, it is convenient to introduce the
so-called conormal derivatives.

Since $\pa\Omega$ is compact, one can find finitely many points $x_i^0\in\pa\Om$, radii $r_i>0$, corresponding sets $\Om_i=\Om\cap B^0(x_i^0,r_i)$ and smooth functions $\phi_i\in C^k(\overline{\Om}_i)$ $(i=1,2,\cdots,m, k\geq6)$ such that $\pa\Om\subset \cup_i^m B^0(x_i^0,r_i)$ and
$$
\Om_i=\{x\in B^0(x_i^0,r_i)|x_3>\phi_i(x_1,x_2)\}, \ m\geq i\geq1.
$$
In what follows, we omit the subscript $i$ of $\phi_i$ for notational simplicity. Using this, we now change coordinates so as to flatten out the boundary. To be more specific, we define
\begin{equation*}
	\Phi:~(y,z)\longmapsto (y,\phi(y)+z)=x.
\end{equation*}
Denote $e_{y^1}=(1,0,\partial_1\phi)^T$, $e_{y^2}=(0,1,\partial_2\phi)^T$ and $e_{z}=(0,0,-1)^T$, one sees that
 $(e_{y^1},e_{y^2},e_z)$ is a local basis around the boundary. We emphasize that $e_{y^1}$ and $e_{y^2}$ on the boundary are tangent to $\partial\Omega$, and in general, $e_z$ is not a normal vector field. We now define
\begin{equation*}
	 Z_i=\partial_{y^i}=\partial_i-\partial_i\phi\partial_z,~i=1,2, ~~Z_3=\varphi(z)\partial_z,
\end{equation*}
where $\varphi(z)=\frac {z}{1+z}$ is smooth, supported in $\mathbb{R}_+$ with the property $\varphi(0)=0$,
$\varphi'(0)>0$, $\varphi(z)>0$ for $z>0$. It is easy to check that
\begin{equation*}%\label{3.7}
	Z_kZ_j=Z_jZ_k,~~j,~k=1,2,3
\end{equation*}
and
\begin{equation*}%\label{3.8}
	 \partial_zZ_i=Z_i\partial_z,~i=1,2,~~\mbox{and}~~\partial_zZ_3\neq Z_3\partial_z.
\end{equation*}
Now the unit outward normal $n(x)$ can be equivalently given by
\begin{equation*}%\label{3.10}
	n(x)\equiv n(\Phi(y,z))=\frac{1}{\sqrt{1+|\nabla\phi(y)|^2}}\left(\begin{array}{cccc} &\partial_1\phi(y)\\&\partial_2\phi(y)\\&-1\end{array}\right)\eqdef \frac{-N(y)}{\sqrt{1+|\nabla\phi(y)|^2}}.
\end{equation*}

%One can further take $\Om_0\subset\Om$ such that $\Om\subset \cup_0^m \Om_i$.
We now define the following Sobolev conormal derivatives
\begin{equation*}%\label{2.1}
	 Z^\al=Z_1^{\al_{1}}Z_2^{\al_{2}}Z_3^{\al_{3}}.
\end{equation*}
where $\al$ is a  multi-index with $\al=(a_{1},\al_{2},\al_{3})$,
and the corresponding Sobolev conormal norm:
\begin{equation*}%\label{2.3}
	\|f(t)\|^2_{H_{co}^m}=\sum_{|\al|\leq m}\|Z^\al f(t)\|^2_{L^2_x},~~
	\|f(t)\|_{H_{co}^{k,\infty}}=\sum_{|\al|\leq k}\|Z^\al f(t)\|_{L^\infty_x},
\end{equation*}
for smooth function $f(t,x)$. Note that we also use $H^k$ to denote the usual Sobolev space $W^{k,2}(\Om).$

Then the solution of \eqref{NS}, \eqref{NSid} and \eqref{NSbd} is sought in the set of the
functions
\begin{equation*}
\begin{split}
\FX_\eps(t)=&\left\{[\rho,u,\ta]\Big|\|[\rho-1,u,\ta-1](t)\|^2_{\FX_{\eps}}\leq c_0\eps^2,\ c_0>0\right\}
\end{split}
\end{equation*}
where
\begin{equation}\label{Xe}
\begin{split}
\|[\rho-1,u,\ta-1](t)\|^2_{\FX_{\eps}}
=&\sup\limits_{0\leq s\leq t}
\sum\limits_{\al_0\leq 3}\left\|\pa_t^{\al_0}[\rho-1,u,\ta-1](s)\right\|_2^2
+\sup\limits_{0\leq s\leq t}\sum\limits_{\al_0\leq 2}\left\|\pa_t^{\al_0}\na_x[\rho,u,\ta](s)\right\|_2^2
\\&+\sup\limits_{0\leq s\leq t}\sum\limits_{\al_0\leq 2}\eps^2\left\|\pa_t^{\al_0}\na_x^2\rho(s)\right\|_2^2
+\sup\limits_{0\leq s\leq t}\sum\limits_{\al_0\leq 1}\eps^2\left\|\pa_t^{\al_0}\na_x^2[u,\ta](s)\right\|_2^2
\\&+\sup\limits_{0\leq s\leq t}\sum\limits_{\al_0\leq 3,|\al|=2}\eps\left\|\pa_t^{\al_0}Z^{\al}\left[\widetilde{\rho},u,\widetilde{\ta}\right](s)\right\|_2^2
+\sup\limits_{0\leq s\leq t}\sum\limits_{\al_0\leq 2}\eps^2\left\|\pa_t^{\al_0}\na_x[\rho,u,\ta](s)\right\|_{H_{co}^{2}}^2
\\&+\sup\limits_{0\leq s\leq t}\sum\limits_{\al_0\leq 2}\eps^4\left\|\pa_t^{\al_0}\na_x^2\rho\right\|_{H_{co}^{2}}^2
+\sup\limits_{0\leq s\leq t}\sum\limits_{\al_0\leq 1}\eps^4\left\|\pa_t^{\al_0}\na_x^2[u,\ta]\right\|_{H_{co}^{2}}^2.
\end{split}
\end{equation}

\begin{theorem}\label{NSsol}
Let $\ka_0>0$. If $$\|[\rho_0-1,u_0,\ta_0-1]\|_{\FX_\eps}\leq \ka_0\eps,$$
then there exists a unique global smooth solution $[\rho,u,\ta](t,x)$ to \eqref{NS}, \eqref{NSid} and \eqref{NSbd} satisfying
\begin{equation} \label{apes}
\begin{split}
\|[\rho-1,u,\ta-1](t)\|^2_{\FX_{\eps}}&
 +\eps\sum\limits_{\al_0\leq3}\int_0^t\|\pa_t^{\al_0}\na_x[u,\ta]\|^2_{2}ds
 +\eps\sum\limits_{\al_0\leq2}\int_0^t\|\pa_t^{\al_0}\na_x \rho\|^2_{2}ds
 \\&+\eps\sum\limits_{1\leq\al_0\leq3}\int_0^t\|\pa_t^{\al_0}\rho\|^2_{2}ds
 +\eps\sum\limits_{\al_0\leq2}\int_0^t\|\pa_t^{\al_0}\na^2_x[\rho,u,\ta]\|^2_{2}ds
+\eps^2\sum\limits_{1\leq\al_0\leq3}\int_0^t\|\pa_t^{\al_0}\rho\|^2_{H^2_{co}}ds
  \\& +\eps^2\sum\limits_{\al_0\leq 2}\int_0^t\left\|\pa_t^{\al_0}\na_x\widetilde{\rho}(s)\right\|_{H^2_{co}}^2ds
 +\eps^2\sum\limits_{\al_0\leq 3}\int_0^t\left\|\pa_t^{\al_0}\na_x\left[u,\widetilde{\ta}\right](s)\right\|_{H^2_{co}}^2ds
 \\&+\eps^3\sum\limits_{\al_0\leq 2}\int_0^t\left\|\pa_t^{\al_0}\na^2_x\left[\widetilde{\rho},u,\widetilde{\ta}\right](s)\right\|_{H^2_{co}}^2ds
+\eps^3\sum\limits_{\al_0\leq1}\int_0^t\|\pa_t^{\al_0}\na^3_x[u,\ta]\|_{H^2_{co}}^2ds
\\& \leq C_0\|[\rho_0-1,u_0,\ta_0-1]\|^2_{\FX_\eps},
\end{split}
\end{equation}
for $C_0>0$.
\end{theorem}
%The proof of Proposition \ref{NSsol} is left to the Appendix \ref{}.
\begin{proof}
The local existence of \eqref{NS}, \eqref{NSid} and \eqref{NSbd} follows from a standard iteration method, we only prove the {\it a priori} estimate
\eqref{apes} under the {\it a priori} assumption
\begin{equation}\label{aps}
N(t)\leq \ka^2_0\eps^2,
\end{equation}
where $N(t)$ is given by
\begin{equation*}
\begin{split}
N(t)=&N(\rho,u,\ta)(t)\\=&\|[\rho-1,u,\ta-1](t)\|^2_{\FX_{\eps}}
 +\eps\sum\limits_{\al_0\leq3}\int_0^t\|\pa_t^{\al_0}\na_x[u,\ta]\|^2_{2}ds
 +\eps\sum\limits_{\al_0\leq2}\int_0^t\|\pa_t^{\al_0}\na_x \rho\|^2_{2}ds
 \\&+\eps\sum\limits_{1\leq\al_0\leq3}\int_0^t\|\pa_t^{\al_0}\rho\|^2_{2}ds
 +\eps\sum\limits_{\al_0\leq2}\int_0^t\|\pa_t^{\al_0}\na^2_x[\rho,u,\ta]\|^2_{2}ds
+\eps^2\sum\limits_{1\leq\al_0\leq3}\int_0^t\|\pa_t^{\al_0}\rho\|^2_{H^2_{co}}ds
  \\& +\eps^2\sum\limits_{\al_0\leq 2}\int_0^t\left\|\pa_t^{\al_0}\na_x\widetilde{\rho}(s)\right\|_{H^2_{co}}^2ds
 +\eps^2\sum\limits_{\al_0\leq 3}\int_0^t\left\|\pa_t^{\al_0}\na_x\left[u,\widetilde{\ta}\right](s)\right\|_{H^2_{co}}^2ds
 \\&+\eps^3\sum\limits_{\al_0\leq 2}\int_0^t\left\|\pa_t^{\al_0}\na^2_x\left[\widetilde{\rho},u,\widetilde{\ta}\right](s)\right\|_{H^2_{co}}^2ds
+\eps^3\sum\limits_{\al_0\leq1}\int_0^t\|\pa_t^{\al_0}\na^3_x[u,\ta]\|_{H^2_{co}}^2ds.
\end{split}
\end{equation*}
The proof is then divided into following four steps.

{\it Step 1. The zeroth order energy estimate.} Denote $[\widetilde{\rho},\widetilde{\ta}]=[\rho-1,\ta-1]$, take the inner product of $\eqref{NS}_1$,
$\eqref{NS}_2$ and $\eqref{NS}_3$ with $\widetilde{\rho}, u$ and $\frac{\widetilde{\ta}}{\ta}$, respectively, to obtain
\begin{equation}\label{rhoz}
\frac{1}{2}\frac{d}{dt}\|\widetilde{\rho}\|_2^2+(\na_x\widetilde{\rho}u,\widetilde{\rho})
+(\widetilde{\rho}\na_x\cdot u,\widetilde{\rho})+(\na_x\cdot u,\widetilde{\rho})=0,
\end{equation}
\begin{equation}\label{uz}
\begin{split}
\frac{1}{2}\frac{d}{dt}\int_\Om\rho u^2dx&+(\na_x \widetilde{\rho},u)+(\na_x \widetilde{\ta},u)+(\widetilde{\ta}\na_x \widetilde{\rho},u)+(\widetilde{\rho}\na_x \widetilde{\ta},u)\\=&-\eps\left(\left\|\sqrt{\mu(\ta)}\na_xu\right\|_2^2+\frac{1}{3}\left\|\sqrt{\mu(\ta)}\na_x\cdot u\right\|_2^2\right),
\end{split}
\end{equation}
\begin{equation}\label{taz}
\begin{split}
\frac{3}{4}\frac{d}{dt}\int_\Om\frac{\rho}{\ta} \widetilde{\ta}^2dx&+\frac{3}{4}\int_\Om\frac{\rho}{\ta^2} \widetilde{\ta}^2\pa_t\ta dx
+\frac{3}{4}\int_\Om\frac{\rho}{\ta^2} \widetilde{\ta}^2 u\cdot\na_x\ta dx
+(\na_x\cdot u,\widetilde{\ta})+(\widetilde{\rho}\na_x\cdot u,\widetilde{\ta})\\
=&-\eps\left\|\sqrt{\frac{\ka(\ta)}{\ta}}\na_x\widetilde{\ta}\right\|_2^2
+\frac{\eps}{2}(\mu(\ta)\sigma(u):\sigma(u),\widetilde{\ta}/\ta)+\left(\ka(\ta)\na_x\ta,\frac{\widetilde{\ta}\na_x\ta}{\ta^2}\right).
\end{split}
\end{equation}
Taking the summation of \eqref{rhoz}, \eqref{uz} and \eqref{taz}, applying Lemmas \ref{sob.ine.lem} and \ref{sob.ine} and the {\it a priori} assumption \eqref{aps}, we then have for some $\la>0$
\begin{equation}\label{0eng}
\begin{split}
\left\|\left[\widetilde{\rho},u,\widetilde{\ta}\right](t)\right\|_2^2+\la\eps\int_0^t\left\|\na_x[u,\widetilde{\ta}](s)\right\|_2^2ds
\leq C\left\|\left[\widetilde{\rho},u,\widetilde{\ta}\right](0,x)\right\|_2^2
+\ka_0\eps\int_0^t\left\|\na_x\left[\rho,u,\widetilde{\ta}\right](s)\right\|_2^2ds.
\end{split}
\end{equation}
To obtain the dissipation of $\na_x\widetilde{\rho}$, we next get from the inner product of $\eps(\eqref{NS}_1,\na_x\cdot u)$ and
$\eps(\eqref{NS}_2,\na_x\widetilde{\rho}/\rho)$ that for any $\eta>0$
\begin{equation}\label{rhodis}
\begin{split}
-\eps(u,&\na_x\widetilde{\rho})(t)+\la\eps\int_0^t\left\|\na_x\widetilde{\rho}\right\|_2^2ds\\
\leq& C\eps|(u,\na_x\widetilde{\rho})(0)|
+C_\eta\eps^3\int_0^t\|\na_x^2u\|_2^2ds+C(\eps+\eps^2)\int_0^t\left\|\na_x\left[u,\widetilde{\ta}\right](s)\right\|_2^2ds
\\&+C(\ka_0+\eta)\eps\int_0^t\left\|\na_x\widetilde{\rho}\right\|_2^2ds,
\end{split}
\end{equation}
where we used the fact $(\widetilde{\rho},\na_x\cdot \pa_tu)+(\na_x\rho,\pa_tu)=0$.

Let $\ka_0$ and $\eps$ be suitably small, then \eqref{0eng} and \eqref{rhodis} give rise to
\begin{equation*}%\label{0engp1}
\begin{split}
\left\|\left[\widetilde{\rho},u,\widetilde{\ta}\right](t)\right\|_2^2
&-\eps|(u,\na_x\widetilde{\rho})|+\la\eps\int_0^t\left\|\na_x[\widetilde{\rho},u,\widetilde{\ta}](s)\right\|_2^2ds\\
\leq& C\left\|\left[\widetilde{\rho}_0,u_0,\widetilde{\ta}_0\right]\right\|_2^2+
C\eps|(u_0,\na_x\widetilde{\rho}_0)|
+C\eps^3\int_0^t\|\na_x^2u\|_2^2ds
\leq CN(0)+C\eps N(t).
\end{split}
\end{equation*}
Similarly, by acting $\pa_t^{\al_0}$ to \eqref{NS}, one also has
\begin{equation}\label{0engp2}
\begin{split}
\sum\limits_{\al_0\leq3}&\left\|\pa_t^{\al_0}\left[\widetilde{\rho},u,\widetilde{\ta}\right](t)\right\|_2^2
-\eps\sum\limits_{\al_0\leq2}|(\pa_t^{\al_0}u,\pa_t^{\al_0}\na_x\widetilde{\rho})|
+\la\sum\limits_{\al_0\leq2}\eps\int_0^t\left\|\na_x\pa_t^{\al_0}\widetilde{\rho}(s)\right\|_2^2ds
\\&+\la\sum\limits_{1\leq\al_0\leq3}\eps\int_0^t\left\|\pa_t^{\al_0}\widetilde{\rho}(s)\right\|_2^2ds
+\la\eps\sum\limits_{\al_0\leq3}\int_0^t\left\|\na_x\pa_t^{\al_0}[u,\widetilde{\ta}](s)\right\|_2^2ds\\
\leq& C\sum\limits_{\al_0\leq3}\left\|\pa_t^{\al_0}\left[\widetilde{\rho}_0,u_0,\widetilde{\ta}_0\right]\right\|_2^2+
C\eps\sum\limits_{\al_0\leq2}|(\pa_t^{\al_0}u_0,\na_x\pa_t^{\al_0}\widetilde{\rho}_0)|
\\&+C\eps^3\sum\limits_{\al_0\leq2}\int_0^t\|\na_x^2\pa_t^{\al_0}u\|_2^2ds
+C(\eps+\ka_0)N(t)\\
\leq& CN(0)+C(\eps+\ka_0)N(t).
\end{split}
\end{equation}
Moreover, it follows from Lemmas \ref{sob.ine.lem}  and \ref{sob.ine} and the {\it a priori} assumption \eqref{aps} that
\begin{align}\label{pat3rho}
\eps\int_0^t\|\pa_t^3\rho\|_2^2ds\lesssim& \eps\int_0^t\|\pa_t^2(\na_x\rho\cdot u)\|_2^2ds
+\eps\int_0^t\|\pa_t^2(\tilde{\rho}\na_x\cdot u)\|_2^2ds
+\eps\int_0^t\|\pa_t^2\cdot u\|_2^2ds\notag\\
\leq& C(\eps+\ka_0)N(t)+\eps\int_0^t\|\pa_t^2\na_x u\|_2^2ds.
\end{align}

{\it Step 2. The first order energy estimate.} The energy estimates for $\na_x\left[\widetilde{\rho},u,\widetilde{\ta}\right]$ are subtle since we know nothing about the derivatives of these quantities on the boundary and the dissipation of \eqref{NS} is very weak. Our strategy to take care of these difficulties is the Helmholtz decomposition, elliptic estimates and the Galerkin method. To see this, in terms of Lemma \ref{Hdec}, we first decompose $u$
as
\begin{equation}
\label{def.helmholtz}
u=u^1+u^2,\quad u^1=\na_x \bu, \quad u^2=\na_x\times \bv, \quad n\cdot
u^2|_{\pa\Om}=0.
\end{equation}

Moreover, we set $\ta_m(t,x)-1=\sum\limits_{k=1}^md_k(t)\mathbbm{w}_k(x)$ with $\mathbbm{w}_k(x)\in H^1_0(\Om) (k=1,2,\cdots)$ being the eigenvalues of the operator $-\Delta_x$, i.e.
\begin{eqnarray*}
\left\{\begin{array}{rll}
&-\Delta_x \mathbbm{w}_k=\la_k\mathbbm{w}_k,\ \ x\in\Om,\\[2mm]
&\mathbbm{w}_k=0,\ x\in\pa\Om,
\end{array}\right.
\end{eqnarray*}
where $0<\la_1\leq\la_2\leq\cdots.$ The key point here is that we get an approximation sequence $\ta_m$ such that $\Delta \ta_m|_{\pa\Om}=0.$

We now approximate \eqref{NS} as
\begin{eqnarray}\label{NS2}
\left\{\begin{array}{rlll}
\begin{split}
&\pa_t\rho+\na_x\cdot(\rho u)=0,\\
&\rho(\pa_tu+u\cdot \na_xu)+\na_x (\rho\ta_m)=\eps\na_x\cdot\left[\mu(\ta_m)\sigma(u)\right],\\
&\frac{3}{2}\rho(\pa_t\ta_m+u\cdot\na_x\ta_m)+\rho\ta_m\na_x\cdot u
=\eps\na_x\cdot\left[\ka(\ta_m)\na_x\ta_m\right]
+\frac{\eps}{2}\mu(\ta_m)\sigma(u):\sigma(u),\\
&[\rho,u,\ta_m](0,x)=[\rho_0,u_0,\ta_0](x),\\
&{[u,\ta_m]\big|_{\pa\Omega}=[0,1]}.
\end{split}
\end{array}\right.
\end{eqnarray}
Note that here $[\rho,u]\eqdef[\rho_m,u_m]$ also depend on $m$, we drop the subscript $m$ for brevity.

As before, by Lemma \ref{Hdec}, for any function function $f\in H^k(\Om,\R^3)$, we have the Helmholtz decomposition $f=P_1f+P_2f$, where $P_1$ is the curl free projection operator and $P_2$ is the divergence free operator. Acting $P_2$ to $\eqref{NS2}_2$, one has
\begin{equation}\label{P2u}
\begin{split}
\pa_tu^2&+P_2\{(\rho-1)\pa_tu\}+P_2\{\rho u\cdot \na_xu\}+\eps\mu(1)\na_x\times\na_x\times u\\=&\frac{4\eps}{3}P_2\{(\mu(\ta_m)-\mu(1))\na_x\na_x\cdot u\}
+\eps P_2\{\na_x\mu(\ta_m)\cdot \si(u)\}
\\&-P_2\{(\mu(\ta_m)-\mu(1))\na_x\times\na_x\times u\}.
\end{split}
\end{equation}
Taking the inner product of $\eqref{P2u}$ with $\pa_t u$ and integrating the resulting equation with respect to $t$, we further have by applying Lemma \ref{Hdec} again that
\begin{equation}\label{p2u.p1}
\begin{split}
\eps(\na_x &\times u^2,\na_x \times u^2)+\la\int_0^t\|\pa_tu^2\|_2^2~ds\\
\leq& C\eps|(\na_x\times u_0^2,\na_x\times u_0^2)|+
C\int_0^t\|\widetilde{\rho}\pa_tu\|_2^2ds
+C\int_0^t\|\rho u\cdot \na_xu\|_2^2ds\\
&+C\eps^2\int_0^t\|\na_x\mu(\ta_m)\cdot\sigma(u)\|_2^2ds+C\eps^2\int_0^t\|(\mu(\ta_m)-\mu(1))\na_x\na_x\cdot u\|_2^2ds
\\
&+C\eps^2\int_0^t\|(\mu(\ta_m)-\mu(1))\na_x\times\na_x\times u\|_2ds\\
\leq &C\eps|(\na_x\times u_0^2,\na_x\times u_0^2)|+C\ka^2_0\eps^2\int_0^t\left\|\na_x u(s)\right\|_{H^1}^2ds
\leq C\eps N(0)+C\ka^2_0\eps N(t),
\end{split}
\end{equation}
where we have used the following identities that
$$\na_x\cdot(\mu(\ta_m)\si(u))=\na_x\mu(\ta_m)\cdot \si(u)+\mu(\ta_m)(\frac{4}{3}\na_x\na_x\cdot u-\na_x\times\na_x\times u),$$
as well as
$$
(\na_x\times\na_x\times u,\pa_t u)=(\na_x\times u,\pa_t \na_x\times u)=(\na_x\times u^2,\na_x\times \pa_t u^2),
$$
due to the boundary condition that $u|_{\pa\Omega}=0$ and the Helmholtz decomposition \eqref{def.helmholtz}.
%and
%$$(\na_x(\rho\ta),\pa_t u_2)=(\na_x(\rho\ta),\pa_t \na\times\bv)=0.$$
Likewise, it follows that for $\al_0\leq 2$
\begin{equation}\label{patu2p1}
\begin{split}
\eps(\na_x& \times \pa_t^{\al_0}u^2,\na_x \times \pa_t^{\al_0}u^2)+\la\int_0^t\|\pa_t^{\al_0+1}u^2\|_2^2~ds\\
\leq &C\eps|(\na_x\times \pa_t^{\al_0}u_0^2,\na_x\times \pa_t^{\al_0}u_0^2)|+C\ka^2_0\eps^2\sum\limits_{\al_0\leq2}\int_0^t\left\|\na_x\pa_t^{\al_0}u(s)\right\|_{H^1}^2ds
\leq C\eps N(0)+C\ka^2_0\eps N(t).
\end{split}
\end{equation}

Next, we rewrite $\na_x\cdot(\mu(\ta_m)\si(u))=\na_x\mu(\ta_m)\cdot \si(u)+\mu(\ta_m)(\Delta_x u+1/3\na_x\na_x\cdot u)$ and consider the following elliptic problems:
\begin{eqnarray*}%\label{epeq}
\left\{\begin{array}{rll}
&\eps\na_x\cdot u=-\eps\pa_t\rho-\eps\widetilde{\rho}\na_x\cdot u-\eps\na_x\rho \cdot u\eqdef h_1,\\[2mm]
&-\eps\mu(1)\Delta u+\na_x\overline{P}=-\pa_tu^2-\widetilde{\rho}\pa_t u-\rho u\cdot\na_xu+\eps(\mu(\ta_m)-\mu(1))\Delta u
\\[2mm]&\qquad\qquad+\eps\na_x\mu(\ta_m)\cdot \si(u)+\frac{\eps}{3}(\mu(\ta_m)-\mu(1))\na_x\na_x\cdot u\eqdef h_2,\\
&u|_{\pa\Om}=0,
\end{array}\right.
\end{eqnarray*}
where
$$
\overline{P}=\pa_t\bu+\rho\ta_m-\frac{\eps}{3}\mu(1)\na_x\cdot u.
$$
In view of Lemma 4.3 in \cite[pp.451]{Matsumura-Nishida-1983}, one has for any $\al_0\geq0$
\begin{align}
\eps^2\|\na_x^2 \pa_t^{\al_0}u\|_2^2\lesssim& \|\pa_t^{\al_0}h_1\|_{H^1}^2+\| \pa_t^{\al_0}h_2\|_2^2\notag\\
\lesssim& \eps^2\|\pa_t^{\al_0+1}\widetilde{\rho}\|^2_{H^1}
+\ka^2_0\eps^3\sum\limits_{\al_0'\leq\al_0}\left\|\na_x\pa_t^{\al'_0}[\widetilde{\rho},u]\right\|_2^2
+\ka_0^2\eps^3\sum\limits_{\al_0'\leq\al_0}\left\|\na^2_x\pa_t^{\al_0}[\widetilde{\rho},u]\right\|_2^2\notag\\
&+\|\pa_t\pa_t^{\al_0}u^2\|_2^2+\ka_0^2\eps^2\sum\limits_{\al_0'\leq\al_0}\left\|\pa_t^{\al_0}[\na_x\rho,\na_xu]\right\|_2^2\label{na2u.p1}\\
\lesssim& \eps^2\|\na_x\cdot \pa_t^{\al_0}u\|^2_{H^1}+\ka^2_0\eps^3\sum\limits_{\al_0'\leq\al_0}\left\|\na_x\pa_t^{\al_0}[\widetilde{\rho},u]\right\|_2^2
+\ka_0^2\eps^3\sum\limits_{\al_0'\leq\al_0}\left\|\na^2_x\pa_t^{\al_0}[\widetilde{\rho},u]\right\|_2^2\notag\\
&+\|\pa_t\pa_t^{\al_0}u^2\|_2^2+\ka_0^2\eps^2\sum\limits_{\al_0'\leq\al_0}\left\|\pa_t^{\al'_0}[\na_x\rho,\na_xu]\right\|_2^2,\label{na2u.p2}
\end{align}
and
\begin{equation}\label{na2u.curl}
\begin{split}
\eps^2\|\na^2_x \pa_t^{\al_0}u^2\|_2^2
\lesssim& \ka^2_0\eps^3\sum\limits_{\al_0'\leq\al_0}\left\|\na_x\pa_t^{\al_0}[\widetilde{\rho},u]\right\|_2^2
+\ka_0^2\eps^3\sum\limits_{\al_0'\leq\al_0}\left\|\na^2_x\pa_t^{\al_0}[\widetilde{\rho},u]\right\|_2^2\\
&+\|\pa_t\pa_t^{\al_0}u^2\|_2^2+\ka_0^2\eps^2\sum\limits_{\al_0'\leq\al_0}\left\|\pa_t^{\al_0}[\na_x\rho,\na_xu]\right\|_2^2,
\end{split}
\end{equation}
since $\Delta u=\na_x\na_x\cdot u-\na_x \times\na_x\times u$, $\Delta u^2=-\na_x \times\na_x\times u^2,$ and $\na_x\cdot u^2=0.$

%Combing \eqref{0engp2}, \eqref{patu2p1} and \eqref{na2u}, we arrive at
%\begin{equation}\label{na2udiss}
%\begin{split}
%(\na_x \times \pa_t^{\al_0}u^2,&\na_x \pa_t^{\al_0}\times u^2)+\frac{\la}{\eps}\int_0^t\|\pa_t\pa_t^{\al_0}u^2\|_2^2~ds
%+\la\eps\int_0^t\|\na_x^2\pa_t^{\al_0} u\|_2^2ds\\
%\lesssim& N(0)+(\ka^2_0\eps+\eps^2+\eps) N(t)+\ka_0^2\eps^2\int_0^t\|\na_x^2 \pa_t^{\al_0}\rho\|_2^2ds+\eps\int_0^t\|\na_x \pa_t^{\al_0}u\|_2^2ds,
%\end{split}
%\end{equation}
%where $\al_0\leq1.$

Moreover, by using
\begin{equation*}
\begin{split}
&-\eps\mu(1)\Delta u-\frac{\eps}{3}\mu(1)\na_x\na_x\cdot u=-\rho\pa_tu-\rho u\cdot\na_xu-\na_xp+\eps(\mu(\ta_m)-\mu(1))\Delta u
\\&\qquad\qquad+\eps\na_x\mu(\ta_m)\cdot \si(u)+\frac{\eps}{3}(\mu(\ta_m)-\mu(1))\na_x\na_x\cdot u,\\
&-\eps\ka(1)\Delta\ta_m=
-\frac{3}{2}\rho(\pa_t\ta_m+u\cdot\na_x\ta_m)-\rho\ta_m\na_x\cdot u+
\eps\na_x\ka(\ta_m)\cdot\na_x\ta_m\\&\qquad\qquad+\eps(\ka(\ta_m)-\ka(1))\Delta\ta_m
+\frac{\eps}{2}\mu(\ta_m)\sigma(u):\sigma(u),
\end{split}
\end{equation*}
we get from standard elliptic estimates that for any $\al_0\geq0$
\begin{equation}\label{na2u.ep}
\begin{split}
\eps^2\|\na_x^2\pa_t^{\al_0}u\|_2^2
\lesssim& \|\pa_t\pa_t^{\al_0}u\|_2^2
+\sum\limits_{\al_0'\leq\al_0}\left\|\na_x\pa_t^{\al'_0}[\widetilde{\rho},\ta_m]\right\|_2^2
+\ka^2_0\eps\sum\limits_{\al_0'\leq\al_0}\left\|\na_x\pa_t^{\al'_0}u\right\|^2
\\&+\ka_0^2\eps^2\sum\limits_{\al_0'\leq\al_0}\left\|\na^2_x\pa_t^{\al'_0}u\right\|_2^2,
\end{split}
\end{equation}
\begin{equation}\label{na3u}
\begin{split}
\eps^4\|\na_x^3 \pa_t^{\al_0}u\|_2^2
\lesssim& \eps^2\|\pa_t\pa_t^{\al_0}u\|^2_{H^1}
+\eps^2\sum\limits_{\al_0'\leq\al_0}\left\|\na_x\pa_t^{\al'_0}[\widetilde{\rho},\ta_m]\right\|_{H^1}^2
+\ka^2_0\eps^3\sum\limits_{\al_0'\leq\al_0}\left\|\na_x\pa_t^{\al'_0}u\right\|_{H^1}^2
\\&+\ka_0^2\eps^4\sum\limits_{\al_0'\leq\al_0}\left\|\na^3_x\pa_t^{\al'_0}u\right\|_2^2,
\end{split}
\end{equation}
and
\begin{equation}\label{na2ta}
\begin{split}
\eps^2\|\na_x^2 \pa_t^{\al_0}\ta_m\|_2^2
\lesssim& \|[\pa_t\pa_t^{\al_0}\ta_m,\na_x\pa_t^{\al_0}u]\|^2_2
+\ka^2_0\eps^2\sum\limits_{\al_0'\leq\al_0}\left\|\na_x\pa_t^{\al_0}[u,\ta_m]\right\|_2^2
+\ka_0^2\eps^3\sum\limits_{\al_0'\leq\al_0}\left\|\na^2_x\pa_t^{\al'_0}\ta_m\right\|_2^2,
\end{split}
\end{equation}
\begin{equation}\label{na3ta}
\begin{split}
\eps^4\|\na_x^3 \pa_t^{\al_0}\ta_m\|_2^2
\lesssim& \eps^2\|[\pa_t\pa_t^{\al_0}\ta_m,\na_x\pa_t^{\al_0}u]\|^2_{H^1}
+\ka^2_0\eps^4\sum\limits_{\al_0'\leq\al_0}\left\|\na_x\pa_t^{\al_0}[u,\ta_m]\right\|_2^2
\\&+\ka_0^2\eps^4\sum\limits_{\al_0'\leq\al_0}\left\|\na^3_x\pa_t^{\al'_0}[\ta_m,u](s)\right\|_2^2.
\end{split}
\end{equation}
%\begin{equation}
%\begin{split}
%\eps^4\|\na_x^3 u\|_2^2
%\lesssim& \eps^2\|\pa_t\na^2_x\widetilde{\rho}\|^2_2+\ka^2_0\eps^3\left\|\na^2_x[\widetilde{\rho},u]\right\|_2^2
%+\ka_0^2\eps^4\left\|\na^3_x[\widetilde{\rho},u](s)\right\|_2^2\\
%&+\|\pa_t\na_xu^2\|_2^2+\ka_0^2\eps^2\left\|[\na_x\widetilde{\rho},,\na_xu,\pa_t\na_xu,\na^2_xu]\right\|_2^2,
%\end{split}
%\end{equation}
Next, $(\na_x\eqref{NS2}_1,\na_x\rho)-(\eqref{NS2}_2,\na_x\na_x\cdot u)+(\na_x\eqref{NS2}_3,\na_x\ta_m)$ yields
\begin{equation}\label{1drut}
\begin{split}
\frac{1}{2}\frac{d}{dt}&\|\na_x\rho\|^2_2+\frac{1}{2}\frac{d}{dt}\|\sqrt{\rho}\na_x\cdot u\|^2_2
+\frac{3}{4}\frac{d}{dt}\|\sqrt{\rho}\na_x\ta_m\|^2_2
\\&
+(\na_x(\na_x\rho \cdot u),\na_x\rho)+(\na_x(\widetilde{\rho}\na_x\cdot u),\na_x\rho)
+(\na_x\rho\cdot\pa_tu,\na_x\cdot u)\\
&+(\na_x\cdot(\rho u\cdot \na_xu),\na_x\cdot u)
-(\widetilde{\ta}_m\na_x \rho,\na_x\na_x\cdot u)-(\widetilde{\rho}\na_x \ta_m,\na_x\na_x\cdot u) \\&+\eps(\na_x\mu(\ta_m)\cdot \si(u),\na_x\na_x\cdot u)+
\eps\left(\mu(\ta_m)(\frac{4}{3}\na_x\na_x\cdot u-\na_x\times\na_x\times u),\na_x\na_x\cdot u\right)\\
&+\frac{3}{2}(\na_x\rho\pa_t\ta_m,\na_x\ta_m)+(\na_x(\rho u\cdot\na_x\ta_m),\na_x\ta_m)+(\na_x(\rho\ta_m)\na_x\cdot u,\na_x\ta_m)
\\&+((\rho\ta_m-1)\na_x\na_x\cdot u,\na_x\ta_m)
-\eps(\na_x(\na_x\ka(\ta_m)\cdot\na_x\ta_m),\na_x\ta_m)
-\eps(\na_x\ka(\ta_m)\Delta_x\ta_m,\na_x\ta_m)
\\&-\eps(\ka(\ta_m)\na_x\Delta_x\ta_m,\na_x\ta_m)
-\frac{\eps}{2}(\na_x(\mu(\ta_m)\sigma(u):\sigma(u)),\na_x\ta_m)=0,\\
\end{split}
\end{equation}
which further implies
\begin{equation}\label{1drho}
\begin{split}
\|\na_x\rho\|^2_2&+\|\na_x\cdot u\|^2_2
+\|\na_x\ta_m\|^2_2
+\la\eps\int_0^t\|\na_x\na_x\cdot u\|^2_2+\|\na^2_x\ta_m\|_2^2~ds\\
\lesssim &N(0)+(\eps+\ka_0) N(t)+\eps\int_0^t\|\na_x\times\na_x\times u\|^2_2~ds,
\end{split}
\end{equation}
where we also used the fact $\|\na^2_x\ta_m\|_2^2\leq C\|\Delta_x\ta_m\|_2^2$.
Similarly, it also holds
\begin{equation}\label{1drhop1}
\begin{split}
\|\na_x\pa_t^{\al_0}\rho\|^2_2&+\|\na_x\cdot \pa_t^{\al_0}u\|^2_2
+\|\na_x\pa_t^{\al_0}\ta_m\|^2_2
+\la\eps\int_0^t\|\na_x\na_x\cdot \pa_t^{\al_0}u\|^2_2+\|\na^2_x\pa_t^{\al_0}\ta_m\|_2^2~ds\\
\lesssim &N(0)+(\eps+\ka_0) N(t)+\eps\int_0^t\|\na_x\times\na_x\times\pa_t^{\al_0}u\|^2_2~ds,
\end{split}
\end{equation}
for $\al_0\leq 2.$ Hence, combing \eqref{patu2p1}, \eqref{na2u.p2}, \eqref{na2u.curl}, \eqref{na2u.ep} and \eqref{1drhop1} together, one has
\begin{equation}\label{2diss.uta}
\begin{split}
\sum\limits_{\al_0\leq2}&\|\na_x\pa_t^{\al_0}\rho\|^2_2+\sum\limits_{\al_0\leq2}\|\na_x \pa_t^{\al_0}u\|^2_2
+\sum\limits_{\al_0\leq2}\|\na_x\pa_t^{\al_0}\ta_m\|^2_2
+\la\eps\sum\limits_{\al_0\leq2}\int_0^t\|\na_x^2 \pa_t^{\al_0}[u,\ta_m]\|_2^2~ds\\
\lesssim &N(0)+\ka_0 N(t),
\end{split}
\end{equation}

{\it Step 3. The estimates for $\na_x^2\rho$.} We now act $P_1$ to $\eqref{NS2}_2$ to obtain
\begin{equation}\label{P1u}
\begin{split}
P_1\{\rho(\pa_tu+u\cdot \na_xu)\}+\na_x (\rho\ta_m)=&\frac{4\eps}{3}P_1\{(\mu(\ta_m)-\mu(1))\na_x\na_x\cdot u\}
+\frac{4\eps}{3}\mu(1)\na_x\na_x\cdot u
\\&+\eps P_1\{\na_x\mu(\ta_m)\cdot \si(u)\}
-\frac{4\eps}{3}P_1\{(\mu(\ta_m)-\mu(1))\na_x\times\na_x\times u\}.
\end{split}
\end{equation}
Then $\frac{4\mu(1)\eps^2}{3}(\na^2_x\eqref{NS2}_1,\na_x^2\rho)+(\na_x\eqref{P1u},\eps\na_x^2\rho)$ gives rise to
\begin{equation}\label{p1u.p1}
\begin{split}
\frac{2\mu(1)\eps^2}{3}&\frac{d}{dt}\|\na_x^2\rho\|^2_2+\frac{4\mu(1)\eps^2}{3}(\na^2_x(\na_x\rho\cdot u),\na_x^2\rho)+\frac{4\mu(1)\eps^2}{3}(\na^2_x(\widetilde{\rho}\na_x\cdot u),\na_x^2\rho)\\
&+\eps(\na_xP_1\{\rho(\pa_tu+u\cdot \na_xu)\},\na_x^2\rho)+\eps(\ta_m\na_x^2\rho,\na_x^2\rho)
+\eps(\na_x^2\ta_m\rho,\na_x^2\rho)\\&+2\eps(\na_x\ta_m\na_x\rho,\na_x^2\rho)
-\frac{4\eps^2}{3}(P_1\{(\mu(\ta_m)-\mu(1))\na_x\na_x\cdot u\},\na_x^2\rho)\\&-
\eps^2(\na_xP_1\{\na_x\mu(\ta_m)\cdot \si(u)\},\na_x^2\rho)+\frac{4\eps^2}{3}(\na_xP_1\{(\mu(\ta_m)-\mu(1))\na_x\times\na_x\times u\},\na_x^2\rho).
\end{split}
\end{equation}
Consequently, by employing Lemmas \ref{sob.ine} and \ref{Hdec} and the {\it a priori} assumption \ref{aps} one has
\begin{equation*}%\label{2drhop1}
\begin{split}
\eps^2\|\na_x^2\rho\|^2_2&+\la\eps\int_0^t\|\na_x^2\rho\|^2_2ds\lesssim N(0)+(\eps+\ka_0)N(t)+\eps\int_0^t\|[\pa_t\na_xu,\na_x^2\ta_m]\|^2_2ds,
\end{split}
\end{equation*}
and a similar calculation leads us to
\begin{equation}\label{2drhop2}
\begin{split}
\eps^2\|\na_x^2\pa_t^{\al_0}\rho\|^2_2&+\la\eps\int_0^t\|\na_x^2\pa_t^{\al_0}\rho\|^2_2ds\\
\lesssim& N(0)+\ka_0N(t)+\eps\int_0^t\|[\pa^{\al_0+1}_t\na_xu,\na_x^2\pa_t^{\al_0}\ta_m]\|^2_2ds,\ \ \text{for}\ \al_0\leq2.
\end{split}
\end{equation}
Let $m\rightarrow\infty$,
we thereupon conclude from \eqref{0engp2}, \eqref{pat3rho}, \eqref{na2u.p1}, \eqref{na3u}, \eqref{na2ta}, \eqref{na3ta},
\eqref{1drhop1}, \eqref{2diss.uta} and \eqref{2drhop2} that
\begin{equation}\label{basiceng}
\begin{split}
\sum\limits_{\al_0\leq 3}&\Big\|\pa_t^{\al_0}\left[\widetilde{\rho},u,\widetilde{\ta}\right](t)\Big\|_2^2
+\sum\limits_{\al_0\leq 2}\left\|\na_x\pa_t^{\al_0}\left[\widetilde{\rho},u,\widetilde{\ta}\right](t)\right\|_2^2
+\eps^2\sum\limits_{\al_0\leq 2}\left\|\na^2_x\pa_t^{\al_0}\widetilde{\rho}(t)\right\|_2^2
+\eps^2\sum\limits_{\al_0\leq 1}\left\|\na^2_x\pa_t^{\al_0}\left[u,\widetilde{\ta}\right](t)\right\|_2^2
\\&+\sum\limits_{\al_0\leq 1}\eps^4\left\|\na^3_x\pa_t^{\al_0}\left[u,\widetilde{\ta}\right](t)\right\|_2^2
+\sum\limits_{1\leq\al_0\leq 3}\la\eps\int_0^t\left\|\pa_t^{\al_0}\widetilde{\rho}(s)\right\|_2^2ds
+\sum\limits_{\al_0\leq 2}\la\eps\int_0^t\left\|\na_x\pa_t^{\al_0}\widetilde{\rho}(s)\right\|_2^2ds
\\&+\sum\limits_{\al_0\leq 3}\la\eps\int_0^t\left\|\na_x\pa_t^{\al_0}[u,\widetilde{\ta}](s)\right\|_2^2ds
+\sum\limits_{\al_0\leq 2}\la\eps\int_0^t\left\|\na^2_x\pa_t^{\al_0}[\widetilde{\rho},u,\widetilde{\ta}](s)\right\|_2^2ds
\\&+\sum\limits_{\al_0\leq 1}\la\eps^3\int_0^t\left\|\na^3_x\pa_t^{\al_0}[u,\ta]\right\|_2^2ds
\leq CN(0)+C(\eps+\ka_0)N(t).
\end{split}
\end{equation}
%To close the estimate \eqref{aps}, we now intend to complete the

{\it Step 4. Conormal energy estimates.}
In this step, with \eqref{basiceng} in our hands, we intend to obtain
\begin{equation}\label{coneng}
\begin{split}
\eps\sup\limits_{0\leq s\leq t}&\sum\limits_{\al_0\leq3,|\al|=2}
\left\|\pa_t^{\al_0}Z^{\al}\left[\widetilde{\rho},u,\widetilde{\ta}\right](s)\right\|_2^2
+\eps^2\sup\limits_{0\leq s\leq t}\sum\limits_{\al_0\leq 2,|\al|=2}
\left\|\na_x\pa_t^{\al_0}Z^{\al}\left[\widetilde{\rho},u,\widetilde{\ta}\right](s)\right\|_2^2
\\&+\sup\limits_{0\leq s\leq t}\sum\limits_{\al_0\leq 1}\eps^4\left\|\pa_t^{\al_0}\na^2[\rho,u,\ta]\right\|_{H_{co}^{2}}^2
+\eps^2\sum\limits_{\al_0\leq 3,|\al|=2}\int_0^t\|\pa_t^{\al_0}Z^{\al}\widetilde{\rho}\|_2^2ds
\\&+\eps^2\sum\limits_{\al_0\leq 2,|\al|=2}\int_0^t\left\|\na_x\pa_t^{\al_0}Z^{\al}\widetilde{\rho}(s)\right\|_2^2ds
+\eps^2\sum\limits_{\al_0\leq 3,|\al|=2}\int_0^t\left\|\na_x\pa_t^{\al_0}Z^{\al}\left[u,\widetilde{\ta}\right](s)\right\|_2^2ds
\\&+\eps^3\sum\limits_{\al_0\leq 1}\int_0^t\left\|\pa_t^{\al_0}\na^2_x\left[\widetilde{\rho},u,\widetilde{\ta}\right](s)\right\|_{H^2_{co}}^2ds
+\eps^3\sum\limits_{\al_0\leq1}\int_0^t\|\pa_t^{\al_0}\na^3_x[u,\ta]\|_{H^2_{co}}^2ds\\ \leq& CN(0)+C(\eps+\ka_0)N(t).
\end{split}
\end{equation}
To show \eqref{coneng}, we only prove the case of $\al_0=0$. For this, letting $|\al|=2$, taking the inner products of $Z^{\al}\eqref{NS}_1$,
$Z^{\al}\eqref{NS}_2$ and $Z^{\al}\eqref{NS}_3$ with $Z^\al\rho$, $Z^\al u$ and $Z^\al\ta$, respectively,
one has
\begin{equation}\label{rhoZ}
\begin{split}
(Z^{\al}\pa_t\rho,Z^\al\widetilde{\rho})&+(Z^{\al}(\na_x\widetilde{\rho}u),Z^{\al}\widetilde{\rho})
+(Z^{\al}(\widetilde{\rho}\na_x\cdot u),Z^{\al}\widetilde{\rho})
\\&+(\na_x\cdot Z^{\al}u,Z^{\al}\widetilde{\rho})+([Z^{\al},\na_x\cdot] u,Z^{\al}\widetilde{\rho})=0,
\end{split}
\end{equation}
\begin{align}\label{uZ}
(Z^\al&(\rho\pa_tu),Z^\al u)+(Z^\al(\rho u\cdot\na_xu),Z^\al u)+(\na_xZ^\al \widetilde{\rho},Z^\al u)
+([Z^\al,\na_x] \widetilde{\rho},Z^\al u)\notag\\&+(Z^\al\na_x \widetilde{\ta},Z^\al u)+(Z^\al(\widetilde{\ta}\na_x \widetilde{\rho}),Z^\al u)+(Z^\al(\widetilde{\rho}\na_x \widetilde{\ta}),Z^\al u)\\=&-\eps\left((\mu(\ta)\si(Z^\al u)),Z^\al u\right)
-\eps\left((\mu(\ta)[Z^\al ,\si](u),Z^\al u\right)-\eps\sum\limits_{|\al'|\geq1}C_{\al}^{\al'}\left(Z^{\al'}(\mu(\ta))Z^{\al-\al'}\si(u),Z^\al u\right)\notag,
\end{align}
and
\begin{equation}\label{taZ}
\begin{split}
\frac{3}{2}&(Z^\al(\rho\pa_t\ta),Z^\al\ta)+(Z^\al(u\cdot\na_x\ta),Z^\al\ta)+(Z^\al(\rho\ta \na_x\cdot u),Z^\al\ta)
\\=&\eps\left(\na_x\cdot\left\{\ka(\ta)\na_xZ^\al\ta\right\},Z^\al\ta\right)
+\eps\left(\na_x\cdot\left\{\ka(\ta)[Z^\al,\na_x]\ta\right\},Z^\al\ta\right)
\\&+\eps\sum\limits_{|\al'|\geq1}C_{\al}^{\al'}\left(\na_x\cdot\left\{Z^{\al'}\ka(\ta)Z^{\al-\al'}\na_x\ta\right\},Z^\al\ta\right)
\\&+\eps\left([Z^\al,\na_x\cdot]\left(\ka(\ta)\na_x\ta\right),Z^\al\ta\right)
+\frac{\eps}{2}(Z^\al\left\{\mu(\ta)\sigma(u):\sigma(u)\right\},Z^\al\ta),
\end{split}
\end{equation}
where $[\CA,\CB]$ denotes the commutator $[\CA,\CB]=\CA\CB-\CB\CA.$
Noticing that $Z^\al u|_{\pa\Om}=0$ and $Z^\al \ta|_{\pa\Om}=0$ for $\al>0$, and
\begin{align}\label{ZGn}
\|Zf\|_2\leq C\|\na_xf\|_2, \ Z=(Z_1,Z_2,Z_3), \ \textrm{for any}\ f\in H^1,
\end{align}
by using Lemma \ref{sob.ine} and the {\it a priori} assumption \ref{aps},
one gets from the summation of $\eps\eqref{rhoZ}$, $\eps\eqref{uZ}$ and $\eps\eqref{taZ}$ that
\begin{align}\label{Z2eng}
\eps\sum\limits_{|\al|=2}&\left\{\|Z^\al[\rho,u,\ta](t)\|^2\right\}
+\la\eps^2\sum\limits_{|\al|=2}\int_0^t\|\na_xZ^\al[u,\ta](s)\|^2ds\notag\\
\leq&C\eps\int_0^t\|\na_x[\rho,u,\ta](s)\|_{H^1}^2ds+CN(0)+C(\eps+\ka_0^2)N(t).
\end{align}
Moreover, we have by taking
the inner product of $\eps^2(Z^\al\eqref{NS}_1,\na_x\cdot Z^\al u)$ and
$\eps^2(\eqref{NS}_2,\na_xZ^\al\widetilde{\rho}/\rho)$ that for any $\eta>0$
\begin{equation}\label{Zrhodis}
\begin{split}
-\eps^2(Z^\al u,&\na_xZ^\al\widetilde{\rho})(t)+\la\eps^2\int_0^t\left\|\na_xZ^\al\widetilde{\rho}\right\|_2^2ds\\
\leq& C\eps^2|(Z^\al u,\na_xZ^\al\widetilde{\rho})(0)|
+C\eps^4\int_0^t\|\na_x^2Z^\al u\|_2^2ds+C(\ka_0^2\eps+\eps^2)\int_0^t\left\|\na_x\left[\rho,u,\widetilde{\ta}\right](s)\right\|_{H^1}^2ds
\\&+\la\ka_0\eps^2\sum\limits_{|\al|=2}\int_0^t\|\na_xZ^\al[u,\ta](s)\|^2ds.
\end{split}
\end{equation}
Next,
performing the similar calculations as for obtaining \eqref{p2u.p1}, employing Lemmas \ref{sob.ine} and \ref{Hdec} and the {\it a priori} assumption \ref{aps}, and using \eqref{ZGn} again, one has for $|\al|=2$
\begin{equation}\label{Zp2u}
\begin{split}
\eps^2(\na_x &\times Z^\al u,\na_x \times Z^\al u)+\la\eps\int_0^t\|\pa_tZ^\al u^2\|_2^2~ds\\
\leq& C\eps^2|(\na_x\times Z^\al u_0,\na_x\times Z^\al u_0)|+\eps^2|([Z^\al,\na_x \times\na_x \times] u, \pa_tZ^\al u)|+
\eps|(\pa_tZ^\al u^2,\pa_tZ^\al u^1)|
\\&+C\eps\sum\limits_{|\al'|\leq 1}\int_0^t\|Z^{\al'}\left\{\widetilde{\rho}\pa_tu\right\}\|_{H^1}^2ds
+C\eps\sum\limits_{|\al'|\leq 1}\int_0^t\|Z^{\al'}\left\{\rho u\cdot \na_xu\right\}\|_{H^1}^2ds
\\&+C\eps^3\sum\limits_{|\al'|\leq 1}\int_0^t\|Z^{\al'}\left\{\na_x\mu(\ta)\cdot\sigma(u)\right\}\|_{H^1}^2ds
+C\eps^3\sum\limits_{|\al'|\leq 1}\int_0^t\|Z^{\al'}\left\{(\mu(\ta)-\mu(1))\na_x\na_x\cdot u\right\}\|_{H^1}^2ds\\&+C\eps^3\sum\limits_{|\al'|\leq 1}\int_0^t\|Z^{\al'}\left\{(\mu(\ta)-\mu(1))\na_x\times\na_x\times u\right\}\|_{H^1}^2ds\\
\leq &C\eps|(\na_x\times Z^\al u_0,\na_x\times Z^\al u_0)|+(C\ka^2_0\eps^2+\eps^2)\int_0^t\left\|\na_x[\widetilde{\rho},u](s)\right\|_{H^1}^2ds
+C\eps^3\sum\limits_{\al'<\al}\int_0^t\|\na^2_xZ^{\al'} u(s)\|_{2}^2ds
\\&+C\eps^3\sum\limits_{|\al|=2}\int_0^t\|\na_xZ^\al[\rho,u,\ta](s)\|_{2}^2ds
+C\ka_0^2\eps^3\sum\limits_{|\al|\leq 2}\int_0^t\|\na^2_xZ^\al u(s)\|_{H^1}^2ds+C(\eps+\ka_0)N(t)\\
\leq &CN(0)+C(\eps+\ka_0)N(t)+C\eps^3\sum\limits_{\al'<\al}\int_0^t\|\na^2_xZ^{\al'} u(s)\|_{2}^2ds,
\end{split}
\end{equation}
where we have used the estimates
$|(\pa_tZ^\al u^2,\pa_tZ^\al u^1)|\lesssim\|\pa_t\na_xu\|^2_{H_{co}^1}$ for $|\al|=2.$

Similar as for obtaining \eqref{1drut}, we have for $|\al|=2$
\begin{equation*}\label{Z1drut}
\begin{split}
\frac{\eps^2}{2}\frac{d}{dt}&\|\na_xZ^\al\rho\|^2_2+\frac{\eps^2}{2}\frac{d}{dt}\|\na_x\cdot Z^\al u\|^2_2
+\frac{3\eps^2}{4}\frac{d}{dt}\|\na_xZ^\al\ta\|^2_2
\\&
+\eps^2(\na_xZ^\al(\na_x\rho \cdot u),\na_xZ^\al\rho)+\eps^2(\na_xZ^\al(\widetilde{\rho}\na_x\cdot u),\na_xZ^\al\rho)
+\eps^2(\na_xZ^\al(\tilde{\rho}\pa_tu),\na_x\na_x\cdot Z^\al u)\\
&+\eps^2(Z^\al(\rho u\cdot \na_xu),\na_x\na_x\cdot Z^\al u)
-\eps^2(Z^\al(\widetilde{\ta}\na_x \rho),\na_x\na_x\cdot Z^\al u)-\eps^2(Z^\al(\widetilde{\rho}\na_x \ta),\na_x\na_x\cdot Z^\al u) \\&+\eps^2(Z^\al(\na_x\mu(\ta)\cdot \si(u)),\na_x\na_x\cdot Z^\al u)+
\frac{4\eps^3}{3}\left(\mu(\ta)\na_x\na_x\cdot Z^\al u,\na_x\na_x\cdot Z^\al u\right)
\\&+\frac{4\eps^3}{3}\left(\mu(\ta)[Z^\al,\na_x\na_x\cdot] u,\na_x\na_x\cdot Z^\al u\right)
+\frac{4\eps^3}{3}\sum\limits_{|\al'|\geq1}C_{\al}^{\al'}\left(Z^{\al'}(\mu(\ta))Z^{\al-\al'}\na_x\na_x\cdot  u,\na_x\na_x\cdot Z^\al u\right)
\\&
-\eps^3\left(Z^\al(\mu(\ta)\na_x\times\na_x\times u),\na_x\na_x\cdot Z^\al u\right)+\frac{3\eps^2}{2}(\na_xZ^\al(\tilde{\rho}\pa_t\ta),\na_xZ^\al\ta)\\
&+\eps^2(\na_xZ^\al(\rho u\cdot\na_x\ta),\na_xZ^\al\ta)+\eps^2(\na_xZ^\al\{(\rho\ta-1)\na_x\cdot u\},\na_xZ^\al\ta)
\\&+\eps^2(\na_x[Z^\al,\na_x\cdot] u,\na_xZ^\al\ta)+\eps^2(\na_x\na_x\cdot Z^\al u,[\na_x,Z^\al]\ta)
\\&-\eps^3(\na_x Z^\al(\na_x\ka(\ta)\cdot\na_x\ta),\na_xZ^\al\ta)
-\eps^3\sum\limits_{|\al'|\geq1}C_{\al}^{\al'}\left(\na_x\{Z^{\al'}\ka(\ta)Z^{\al-\al'}\Delta_x\ta\},\na_xZ^\al\ta\right)
\\&-\eps^3(\na_x(\ka(\ta)Z^\al\Delta_x\ta),\na_xZ^\al\ta)
-\frac{\eps^3}{2}(\na_xZ^\al(\mu(\ta)\sigma(u):\sigma(u)),\na_xZ^\al\ta)=0,\\
\end{split}
\end{equation*}
which further implies
\begin{equation}\label{Z1drut.int}
\begin{split}
\eps^2&\|\na_xZ^\al[\rho,u,\ta](t)\|^2_2
+\la\eps^3\int_0^t\|\na_x\na_x\cdot Z^\al u(s)\|_2^2ds+\la\eps^3\int_0^t\|\na^2_x Z^\al \ta(s)\|_2^2ds
\\ \leq& C\eps^3\sum\limits_{|\al'|\leq|\al|-1}\int_0^t\|\na_x^2 Z^{\al'} \ta(s)\|_2^2ds
+C\eps^3\sum\limits_{|\al|\leq2}\int_0^t\|\na_x^2 Z^\al u(s)\|_2^2ds+CN(0)+C(\eps+\ka_0)N(t).
\end{split}
\end{equation}
Furthermore, as to the estimates for $\na_x^2Z^\al\rho$ with $|\al|=1,2$,
we have as for obtaining
\eqref{p1u.p1} that
\begin{equation*}\label{Z2p1u.p1}
\begin{split}
\frac{2\mu(1)\eps^4}{3}&\frac{d}{dt}\|Z^\al\na_x^2\rho\|^2_2
+\frac{4\mu(1)\eps^4}{3}(Z^\al\na^2_x(\na_x\rho\cdot u),Z^\al\na_x^2\rho)
+\frac{4\mu(1)\eps^4}{3}(Z^\al\na^2_x(\widetilde{\rho}\na_x\cdot u),Z^\al\na_x^2\rho)\\
&+\eps^3(Z^\al \na_xP_1\{\rho(\pa_tu+u\cdot \na_xu)\},Z^\al\na^2_x\rho)+\eps^3(\ta Z^\al\na_x^2\rho,Z^\al\na^2_x\rho)
+\eps^3(Z^\al\na_x(\na_x\ta\rho),Z^{\al}\na_x^2\rho)\\&+\eps^3(Z^\al(\na_x\ta\na_x\rho),Z^\al\na^2_x\rho)
+\eps^3\sum\limits_{|\al'|\geq1}C_{\al}^{\al'}(Z^{\al'}\ta Z^{\al-\al'}\na_x^2\rho,Z^\al\na^2_x\rho)
\\&-\frac{4\eps^4}{3}(Z^\al \na_xP_1\{(\mu(\ta)-\mu(1))\na_x\na_x\cdot u\},Z^\al\na^2_x\rho)\\&-
\eps^4(Z^\al\na_xP_1\{\na_x\mu(\ta)\cdot \si(u)\},Z^\al\na^2_x\rho)
+\frac{4\eps^4}{3}(Z^\al\na_xP_1\{(\mu(\ta)-\mu(1))\na_x\times\na_x\times u\},Z^\al\na^2_x\rho),
\end{split}
\end{equation*}
which yields
\begin{equation}\label{Z2p1u.p2}
\begin{split}
\eps^4&\sum\limits_{1\leq|\al|\leq2}\|Z^\al\na_x^2\rho\|^2_2+\eps^3\sum\limits_{1\leq|\al|\leq2}\int_0^t\|Z^\al\na_x^2\rho\|^2ds\\
\leq& C\ka_0^2\eps^5\sum\limits_{|\al|\leq2}\int_0^t\|\na_x^3 Z^\al u(s)\|_2^2ds+C\eps^3\sum\limits_{|\al|\leq2}\int_0^t\|\na_x^2 Z^\al \ta(s)\|_2^2ds+CN(0)+C(\eps+\ka_0)N(t).
\end{split}
\end{equation}
Finally, elliptic estimates as for obtaining \eqref{na2u.ep} and \eqref{na2ta} leads us to
\begin{equation}\label{Zna2u}
\begin{split}
\eps^4\sum\limits_{|\al|\leq 2}\|\na_x^2 \pa_t^{\al_0}Z^\al u\|_2^2
\lesssim& \eps^4\sum\limits_{|\al|\leq 2}\|\pa_t\pa_t^{\al_0}Z^\al\widetilde{\rho}\|^2_{H^1}+\ka^2_0\eps^4\sum\limits_{|\al|\leq 2}\left\|\na_x\pa_t^{\al_0}Z^\al[\widetilde{\rho},u]\right\|_2^2
\\
&+\ka_0^2\eps^5\sum\limits_{|\al|\leq 2}\left\|\na^2_x\pa_t^{\al_0}Z^\al[\widetilde{\rho},u](s)\right\|_2^2
+\eps^2\sum\limits_{|\al|\leq 2}\|\pa_t\pa_t^{\al_0}Z^\al u^2\|_2^2,
\end{split}
\end{equation}
\begin{equation}\label{Zna2ta}
\begin{split}
\eps^4\sum\limits_{|\al|\leq 2}\|\na_x^2 \pa_t^{\al_0}Z^\al\ta\|_2^2
\lesssim& \eps^2\sum\limits_{|\al|\leq 2}\|[\pa_t\pa_t^{\al_0}Z^\al\ta,\na_x\pa_t^{\al_0}Z^\al u]\|^2_2
\\&+\ka^2_0\eps^3\sum\limits_{|\al|\leq 2}\left\|\na_x\pa_t^{\al_0}Z^\al[u,\ta]\right\|_2^2
+\ka_0^2\eps^4\sum\limits_{|\al|\leq 2}\left\|\na^2_x\pa_t^{\al_0}Z^\al\ta\right\|_2^2,
\end{split}
\end{equation}
\begin{equation}\label{Zna3u}
\begin{split}
\eps^4\sum\limits_{|\al|\leq 2}\|\na_x^3 \pa_t^{\al_0}Z^\al u\|_2^2
\lesssim& \eps^2\sum\limits_{|\al|\leq 2}\|\pa_t\pa_t^{\al_0}Z^\al u\|^2_{H^1}
+\eps^2\sum\limits_{|\al|\leq 2}\left\|\na_x\pa_t^{\al_0}Z^{\al}[\widetilde{\rho},\ta]\right\|_{H^1}^2
\\&+\ka^2_0\eps^3\sum\limits_{|\al|\leq 2}\left\|\na_x\pa_t^{\al_0}Z^\al u\right\|_{H^1}^2
+\ka_0^2\sum\limits_{|\al|\leq 2}\eps^4\left\|\na^3_x\pa_t^{\al_0}u\right\|_2^2,
\end{split}
\end{equation}
and
\begin{equation}\label{Zna3ta}
\begin{split}
\eps^4\sum\limits_{|\al|\leq 2}\|\na_x^3 \pa_t^{\al_0}Z^\al\ta\|_2^2
\lesssim& \eps^2\sum\limits_{|\al|\leq 2}\|[\pa_t\pa_t^{\al_0}Z^\al\ta,\na_x\pa_t^{\al_0}Z^\al u]\|^2_{H^1}\\&+\ka^2_0\eps^4\sum\limits_{|\al|\leq 2}\left\|\na_x\pa_t^{\al_0}Z^\al[u,\ta]\right\|_2^2
+\ka_0^2\eps^4\sum\limits_{|\al|\leq 2}\left\|\na^3_x\pa_t^{\al_0}Z^\al[\ta,u](s)\right\|_2^2.
\end{split}
\end{equation}
Thereupon, \eqref{coneng} follows from a linear combination of \eqref{basiceng}, \eqref{Z2eng}, \eqref{Zrhodis}, \eqref{Zp2u}, \eqref{Z1drut.int}, \eqref{Z2p1u.p2},
\eqref{Zna2u}, \eqref{Zna2ta}, \eqref{Zna3u} and \eqref{Zna3ta}. This completes the proof of Theorem \ref{NSsol}.
\end{proof}

\section{$L^2-L^6$ estimates}\label{l2-l6th}
This section is dedicated to the $L^2-L^6$ estimates for the remainder $R$.

In view of \eqref{ID}, \eqref{dbd}, \eqref{epn} and \eqref{R}, we see that $R$ satisfies the initial boundary value problem:
\begin{eqnarray}\label{Rib}
\left\{\begin{split}
&\pa_tR+v\cdot\na_xR+\frac{1}{\eps}L_MR=\eps^{1/2}Q(R,R)+Q(R,G)+Q(G,R)\\
&\qquad\qquad\qquad\qquad\qquad\qquad+\eps^{-1/2}Q(G,G)-\eps^{-1/2}(\pa_tG+v\cdot \na_xG+H),\\
&R(0,x,v)\eqdef R_0(x,v)=-\eps^{-1/2} G(0,x,v),\\
&R_{-}=P_\ga R+\eps^{-1/2}r,
\end{split}\right.
\end{eqnarray}
where
$$P_\ga R=M^w\int_{n(x)\cdot v'>0}R(t,x,v')(n(x)\cdot v')dv',$$
and $r=P_\ga G-G.$
To solve \eqref{Rib}, we start from the following linear problem
\begin{eqnarray}\label{Rln}
\left\{\begin{split}
&\pa_tR+v\cdot\na_xR+\frac{1}{\eps}L_MR=g,\\
&R(0,x,v)= R_0(x,v),\\
&R_{-}=
P_\ga R+\eps^{-1/2}r,
\end{split}\right.
\end{eqnarray}
where $g=g(t,x,v)\in \ker^{\bot}(L_M)$ and $\ker^{\bot}(L_M)$ is the orthogonal complement of $\ker(L_M)$ which is the kernel of the linear operator $L_M$ generated by
\begin{eqnarray*}
\chi^{M}_0\left(v;\rho,u,\ta\right)&\equiv& \frac{1}{\sqrt{\rho}}M,\\[2mm]
\chi^{M}_{i}\left(v;\rho,u,\ta\right)&\equiv&\frac{v_i-u_i}
{\sqrt{\rho\ta}}M,\ \ i=1,2,3,\\[2mm]
\chi^{M}_{4}\left(v;\rho,u,\ta\right)
&\equiv&\frac{1}{\sqrt{6\rho}}\left(\frac{|v-u|^2}
{\ta}-3\right)M.\\[2mm]
%\left\langle\chi^{M}_i,\chi^{M}_j\right\rangle.
%&=&\delta^{ij},\ \ \text{for}\ \ i,j=0,1,2,3,4.
\end{eqnarray*}
One can further define the macroscopic projection $P^M_0$ and microscopic projection $P^M_1$ as follows
\begin{equation}\label{P0R}
\begin{split}
P_0^M R=&\frac{a(t,x)}{\sqrt{\rho}}\chi^{M}_0+\sum\limits_i\frac{b_i(t,x)}{\sqrt{\rho\ta}}\chi^{M}_i+\frac{c(t,x)}{\ta}\sqrt{\frac{6}{\rho}}\chi^{M}_4\\
=&\frac{a(t,x)}{\rho}M+\frac{b(t,x)\cdot (v-u)}{\rho\ta}M+\frac{c(t,x)}{\rho\ta}\left(\frac{|v-u|^2}
{\ta}-3\right)M,
\end{split}
\end{equation}
and $P_1^MR=R-P_0^MR.$

\begin{lemma}\label{l2es}
Assume $R(t,x,v)$ is a global solution of the initial boundary value problem \eqref{Rln}, for any $\eta>0$, there exists $C_\eta>0$ depends on $\eta$ such that
\begin{equation}\label{Rl2}
\begin{split}
\frac{1}{2}\frac{d}{dt}&\int_{\Om\times \R^3}\sum\limits_{\al_0\leq1}\frac{(\pa_t^{\al_0}R)^2}{M_-}dxdv+\frac{\de_0-3\eta}{\eps}\sum\limits_{\al_0\leq1}\int_{\Om\times \R^3}\frac{\nu_M(P_1^M\pa^{\al_0}_tR)^2}{M_-}dxdv\\&+\frac{1}{2}\sum\limits_{\al_0\leq1}\int_{\ga_+}\frac{|(I-P_\ga)\pa^{\al_0}_tR|^2}{M_-}d\ga\\
\leq& C_\eta\eps\sum\limits_{\al_0\leq1}\int_{\Om\times \R^3}\frac{\nu_M^{-1}(\pa^{\al_0}_tg)^2}{M_-}dxdv+C_\eta\ka_0^2\eps\sum\limits_{\al_0\leq1}\|\pa_t^{\al_0}[a,b,c]\|_2^2
+\frac{1}{2\eps}\sum\limits_{\al_0\leq1}\int_{\ga_-}\frac{|\pa^{\al_0}_tr|^2}{M_-}d\ga,
\end{split}
\end{equation}
here $\de_0$ is given in Lemma \ref{co.est.}.
\end{lemma}
\begin{proof}
We prove $\al_0=1$ only, because the case $\al_0=0$ is similar and much easier. Act $\pa_t$ to $\eqref{Rln}_1$, take inner product with $\frac{\pa_tR}{M_-}$ over $\Om\times \R^3$  and use Green's formula to obtain
\begin{equation}\label{Rtip}
\begin{split}
\frac{1}{2}\frac{d}{dt}&\int_{\Om\times\R^3}\frac{(\pa_tR)^2}{M_-}dxdv
+\frac{1}{2}\int_{\pa\Om\times\R^3}\frac{v\cdot n(\pa_tR)^2}{M_-}dS_xdv\\&+\frac{1}{\eps}\int_{\Om\times\R^3}\frac{\pa_tRL_MP_1^M\pa_tR}{M_-}dxdv
+\frac{1}{\eps}\int_{\Om\times\R^3}\frac{\pa_tRL_{\pa_tM}P_1^MR}{M_-}dxdv=\int_{\Om\times\R^3}\frac{\pa_tg\pa_t R}{M_-},
\end{split}
\end{equation}
where
$L_{\pa_tM}R=-\left\{Q(\pa_tM,R)+Q(R,\pa_tM)\right\}.$

Notice that on the boundary
$$
\pa_tR_{-}=P_\ga \pa_tR+\eps^{-1/2}\pa_tr,
$$
and $M_-=\frac{1}{\sqrt{2\pi}}M^w$, moreover the difference of $M$ and $M_-$ is small in the sense
\begin{equation}\label{dfM}
|M-M_-|\lesssim|[\rho-1,u,\ta-1]|M^{1-\beta},\ 0<\beta\ll1.
\end{equation}
We now perform the calculations for the second, third and fourth term in the left hand side of \eqref{Rtip} as follows.
For the second term,
\begin{equation}\label{bdl2}
\begin{split}
\frac{1}{2}\int_{\pa\Om\times\R^3}&\frac{v\cdot n(\pa_tR)^2}{M_-}dS_xdv\\=&\frac{1}{2\sqrt{2\pi}}\int_{\pa\Om\times\R^3}\frac{v\cdot n(\pa_tR)^2}{M^w}dS_xdv\\
=&\frac{1}{2\sqrt{2\pi}}\int_{\ga_+}\frac{v\cdot n(\pa_tR)^2}{M^w}dS_xdv+\frac{1}{2\sqrt{2\pi}}\int_{\ga_-}\frac{v\cdot n(P_\ga \pa_tR+\eps^{-1/2}\pa_tr)^2}{M^w}dS_xdv
\\
=&\frac{1}{2\sqrt{2\pi}}\int_{\ga_+}\frac{v\cdot n((I-P_\ga)\pa_tR)^2}{M^w}dS_xdv+\frac{1}{2\eps\sqrt{2\pi}}\int_{\ga_-}\frac{v\cdot n(\pa_tr)^2}{M^w}dS_xdv,
\end{split}
\end{equation}
here we also used the fact that
$$
-\int_{\ga_+}\frac{v\cdot n(P_\ga\pa_tR)^2}{M^w}dS_xdv=\int_{\ga_-}\frac{v\cdot n(P_\ga\pa_tR)^2}{M^w}dS_xdv,
$$
and
$$
\int_{\ga_-}\frac{v\cdot n(P_\ga\pa_tR)\pa_t r}{M^w}dS_xdv=0.
$$
By virtue of \eqref{dfM}, Lemmas \ref{co.est.} and \ref{nopp1}, for $\frac{\ta}{2}<1$, we have by using $\int_{\Om\times\R^3}\frac{(P_0^M\pa_tR)L_MP_1^M\pa_tR}{M}dxdv=0$
\begin{equation}\label{col2}
\begin{split}
\frac{1}{\eps}\int_{\Om\times\R^3}&\frac{\pa_tRL_M\pa_tR}{M_-}dxdv
\\=&\frac{1}{\eps}\int_{\Om\times\R^3}\frac{P_1^M\pa_tRL_MP_1^M\pa_tR}{M_-}dxdv
+\frac{1}{\eps}\int_{\Om\times\R^3}\frac{(M-M_-)P_0^M\pa_tRL_MP_1^M\pa_tR}{MM_-}dxdv
\\ \geq&\frac{\de_0}{\eps}\int_{\Om\times \R^3}\frac{\nu_M(P_1^M\pa_tR)^2}{M_-}dxdv-\frac{\eta}{\eps}\int_{\Om\times \R^3}\frac{\nu_M(P_1^M\pa_tR)^2}{M_-}dxdv\\
&-\frac{C_\eta}{\eps}\sum\limits_{\al_0\leq1}\|\pa_t^{\al_0}[a,b,c]\pa_t^{\al_0}[\rho-1,u,\ta-1]\|_2^2.
\end{split}
\end{equation}
Thanks to Theorem \ref{NSsol} and Sobolev's inequality \eqref{sobinep1}, the last term in the right hand side of \eqref{col2} is majorized by
\begin{equation}\label{col2p1}
\begin{split}
\frac{C_\eta}{\eps}&\sum\limits_{\al_0\leq1}\|\pa_t^{\al_0}[a,b,c]\pa_t^{\al_0}[\rho-1,u,\ta-1]\|_2^2\\
\leq& \frac{C_\eta}{\eps}\sum\limits_{\al_0\leq1}\|\pa_t^{\al_0}[a,b,c]\|_2^2\|\na_x\pa_t^{\al_0}[\rho-1,u,\ta-1]\|_{H^1_{co}}
\|\pa_t^{\al_0}[\rho-1,u,\ta-1]\|_{H^2_{co}}\\
\leq& \ka_0^2C_\eta\eps\sum\limits_{\al_0\leq1}\|\pa_t^{\al_0}[a,b,c]\|_2^2.
\end{split}
\end{equation}
By Cauchy-Schwarz's inequality, Lemma \ref{nopp1}, Sobolev's inequality \eqref{sobinep1} and Theorem \ref{NSsol}, the fourth term is bounded by
\begin{equation}\label{col2p2}
\begin{split}
\frac{1}{\eps}\int_{\Om\times\R^3}&\frac{\pa_tRL_{\pa_tM}P_1^MR}{M_-}dxdv
\\=&\frac{1}{\eps}\int_{\Om\times\R^3}\frac{P_1^M\pa_tRL_{\pa_tM}P_1^MR}{M_-}dxdv
+\frac{1}{\eps}\int_{\Om\times\R^3}\frac{P_0^M\pa_tRL_{\pa_tM}P_1^MR}{M_-}dxdv
\\ \leq&\left(\frac{\eta}{\eps}+\ka_0^2C_\eta\eps\right)\sum\limits_{\al_0\leq1}\int_{\Om\times \R^3}\frac{\nu_M(P_1^M\pa^{\al_0}_tR)^2}{M_-}dxdv+\ka_0^2C_\eta\eps\sum\limits_{\al_0\leq1}\|\pa_t^{\al_0}[a,b,c]\|_2^2.
\end{split}
\end{equation}
For the term in the right hand side of \eqref{Rtip}, since $g\in \ker(L_M)$, by Cauchy-Schwarz's inequality, we have
\begin{equation}\label{gip}
\begin{split}
\int_{\Om\times\R^3}\frac{\pa_tg\pa_t R}{M_-}=&\int_{\Om\times\R^3}\frac{(M-M_-)\pa_tgP_0^M\pa_t R}{M_-M}+\int_{\Om\times\R^3}\frac{\pa_tgP_1^M\pa_t R}{M_-}\\
\leq&\frac{\eta}{\eps}\int_{\Om\times \R^3}\frac{\nu_M(P_1^M\pa_tR)^2}{M_-}dxdv+C_\eta\eps\int_{\Om\times \R^3}\frac{\nu_M^{-1}(\pa_tg)^2}{M_-}dxdv
\\&+\ka_0^2C_\eta\eps\sum\limits_{\al_0\leq1}\|\pa_t^{\al_0}[a,b,c]\|_2^2.
\end{split}
\end{equation}
Finally, \eqref{Rl2} for $\al_0=1$ follows from \eqref{Rtip}, \eqref{bdl2}, \eqref{col2}, \eqref{col2p1}, \eqref{col2p2} and \eqref{gip}. This completes the proof of Lemma \ref{l2es}.
\end{proof}

Now we turn to deduce the $L^2$ and $L^6$ estimates for the macroscopic part $[a,b,c](t,x)$. For results in this direction, we have
\begin{lemma}\label{l2-l6}
Let $R$ be a solution to \eqref{Rln} in the sense of distribution. Then there exists $\mathbb{G}(t)$ satisfying $|\mathbb{G}(t)|\leq \sum\limits_{\al_0\leq1}\|\pa_t^{\al_0}R/\sqrt{M_-}\|_2^2$ such that
\begin{equation}\label{P1-P2}
\begin{split}
\eps\sum\limits_{\al_0\leq1}&\int_0^t\|\pa_t^{\al_0}[a,b,c]\|_2^{2}ds\\ \lesssim&  \eps|\mathbb{G}(t)-\mathbb{G}(0)|+\eps^{-1}\sum\limits_{\al_0\leq1}\int_0^t\int_{\Om\times \R^3}\frac{\nu_M(P_1^M\pa^{\al_0}_tR)^2}{M_-}dxdvds
\\&+\eps\sum\limits_{\al_0\leq1}\int_0^t\int_{\ga_+}\frac{|(I-P_\ga)\pa_t^{\al_0}R|^2}{M_-}d\ga ds
+\sum\limits_{\al_0\leq1}\int_0^t\int_{\ga_-} \frac{|\pa_t^{\al_0}r|^2}{M_-} d\ga ds\\&+\eps\sum\limits_{\al_0\leq1}\int_0^t\int_{\Om\times \R^3}\frac{\nu_M^{-1}\pa_t^{\al_0}g^2}{M_-}dxdvds,
\end{split}
\end{equation}
and moreover it holds that for $\eta, \eta'>0$
\begin{equation}\label{P1-P6}
\begin{split}
\eps&\sup\limits_{0\leq s\leq t}\| [a,b,c](s)\|^2_{6} \\
 \lesssim&  \int_{\ga_+}\frac{|(I-P_\ga)R_0|^2}{M_-}d\ga+\eps^{-1}\int_{\Om\times \R^3}\frac{(P_1^MR_0)^2}{M_-}dxdv
 +\ka_0^2\eps\int_0^t\|[a,b,c](s)\|_2^2ds
 \\&+ \eps^{-1}\sum\limits_{\al_0\leq1}\int_0^t\int_{\Om\times \R^3}\frac{\nu_M(P_1^M\pa^{\al_0}_tR)^2}{M_-}dxdvds
 +C_{\eta,\eta'}\sum\limits_{\al_0\leq1}\int_0^t\int_{\ga_+}\frac{|(I-P_\ga)\pa^{\al_0}_tR|^2}{M_-}d\ga ds\\&+(\eta+\eta'+\eps)\eps^2\sup\limits_{0\leq s\leq t}\left\|\frac{R(s)}{\sqrt{M_-}}\right\|_\infty^2
+\eps\sup\limits_{0\leq s\leq t}\int_{\Om\times \R^3}\frac{\nu_M^{-1}g^2}{M_-}dxdv
+\sup\limits_{0\leq s\leq t}\|\na_x[u,\ta](s)\|_{H^1}^2.
\end{split}\end{equation}
\end{lemma}

\begin{proof}
We only prove the estimates for $c(t,x)$ in \eqref{P1-P6} and \eqref{P1-P2}, the estimates for $a$ and $b$ being similar. Let $\psi\in C^\infty((0,+\infty)\times\Om\times \R^3)$ be a test function, by Green's formula, one has
\begin{equation}\label{Gid}
(\pa_t R,\psi)+\int_{\ga}v\cdot nR\psi d\ga-\int_{\Om\times \R^3}v\cdot \na_x \psi Rdxdv=-\frac{1}{\eps}(L_MR,\psi)+(g,\psi).
\end{equation}
Recall $R$ can be decomposed as
\begin{equation}\label{Rdec}
\begin{split}
R=\frac{a(t,x)}{\rho}M+\frac{b(t,x)\cdot (v-u)}{\rho\ta}M+\frac{c(t,x)}{\rho\ta}\left(\frac{|v-u|^2}
{\ta}-3\right)M+P_1^M R,
\end{split}
\end{equation}
and at the boundary
\begin{equation}\label{Rbddec}
R|_\ga={\bf 1}_{\ga_+}(I-P_\ga)R+P_\ga R+\eps^{-1/2}{\bf 1}_{\ga_-}r.
\end{equation}
%In what follows, we estimate $a, b$ and $c$ term by term.

%$\underline{\textrm{Estimates for}\ c(t,x)}.$
For $k=2,6$, we choose the test functions
\begin{equation}
\psi =\psi_{c,k}\equiv \left(\frac{|v-u|^{2}}{\ta}-\beta _{c}\right)(v-u)\cdot \nabla _{x}\varphi
_{c,k}(t,x),  \ \ \textrm{where} \    -\Delta _{x}\varphi _{c,k}(x)=c^{k-1}(x),\ \  \varphi _{c,k}|_{\partial \Omega }=0, \label{phic}
\end{equation}
and $\beta _{c}$ is a constant to be determined later.

By the standard elliptic
estimate, one has
\begin{equation}\label{c22}
\| \varphi _{c,2}\| _{H^{2}}\lesssim \| c\| _{2}.
\end{equation}
In view of the choice \eqref{phic} with $k=2$, the right hand side of (\ref{Gid})
is dominated  by
\be \textrm{r.h.s.} (\ref{Gid})\lesssim\| c\|_{2}\left\{\eps^{-1}\left\| \frac{P_1^MR}{\sqrt{M_-}}\right\|_{2}
+\left\| \frac{\nu_M^{-1/2}g}{\sqrt{M_-}}\right\|_{2}\right\}.\notag\ee

If instead $k=6$, by the Sobolev-Gagliardo-Nirenberg inequality, we have for any $q\in [\frac 6 5, 2]$,
\be\label{epes} \|\nabla \varphi_{c,6}\|_q\lesssim \|\varphi_{c,6}\|_{W^{2,\frac 6 5}},\ee
from which, we obtain at this stage
\begin{equation*}%\label{boundrhs6}
\begin{split} \textrm{r.h.s.} (\ref{Gid})\lesssim&\| \na_x\varphi_{c,6}\|_{2}\left\{\eps^{-1}\left\| \frac{P_1^MR}{\sqrt{M_-}}\right\|_{2}
+\left\| \frac{\nu_M^{-1/2}g}{\sqrt{M_-}}\right\|_{2}\right\}\\
\lesssim&\| c^5\|_{6/5}\left\{\eps^{-1}\left\| \frac{P_1^MR}{\sqrt{M_-}}\right\|_{2}
+\left\| \frac{\nu_M^{-1/2}g}{\sqrt{M_-}}\right\|_{2}\right\}\\
\lesssim&\eta\| c\|_{6}^6+C_\eta\left\{\eps^{-6}\left\| \frac{P_1^MR}{\sqrt{M_-}}\right\|^6_{2}
+\left\| \frac{\nu_M^{-1/2}g}{\sqrt{M_-}}\right\|^6_{2}\right\},
\end{split}
\end{equation*}
where we also used Young's inequality and the standard $W^{k,p}$ estimates for the Poisson equation. Next, it follows from Cauchy-Schwarz's inequality
\begin{equation}\label{eps-1eng}
\begin{split}
\sup\limits_{0\leq s\leq t}\eps^{-1}\left\| \frac{P_1^MR}{\sqrt{M_-}}\right\|^2_{2}
\leq& \eps^{-1}\left\| \frac{P_1^MR_0}{\sqrt{M_-}}\right\|^2_{2}+2\eps^{-1}\int_0^tds
\left\| \frac{P_1^MR}{\sqrt{M_-}}\right\|_{2}\left\| \frac{P_1^M\pa_sR-[P_1^M,\pa_s]R}{\sqrt{M_-}}\right\|_{2}\\
\leq& \eps^{-1}\left\| \frac{P_1^MR_0}{\sqrt{M_-}}\right\|^2_{2}+\eps^{-1}\int_0^t
\left\| \frac{P_1^MR}{\sqrt{M_-}}\right\|_{2}^2 ds+2\eps^{-1}\int_0^t
\left\| \frac{P_1^M\pa_sR}{\sqrt{M_-}}\right\|^2_{2}ds\\&+2\eps^{-1}\int_0^t
\left\| \frac{[P_1^M,\pa_s]R}{\sqrt{M_-}}\right\|^2_{2}ds.
\end{split}
\end{equation}
Here, $[P_1^M,\pa_s]$ is the commutator of $P_1^M$ and $\pa_s$, defined by
$$
[P_1^M,\pa_s]=P_1^M\pa_s-\pa_sP_1^M.
$$
We furthermore obtain thanks to the definition $P_1^MR=R-P_0^M R$ and \eqref{P0R}
\begin{equation*}%\label{eps-1engp1}
\begin{split}
\eps^{-1}\int_0^t
\left\| \frac{[P_1^M,\pa_s]R}{\sqrt{M_-}}\right\|^2_{2}ds\leq&
\frac{C}{\eps}\int_0^t
\left\| \frac{\langle R,\pa_s(\chi^M_i/M)\rangle\chi_i^M}{\sqrt{M_-}}\right\|^2_{2}ds
+\frac{C}{\eps}\int_0^t
\left\| \frac{\langle R,\chi^M_i/M\rangle\pa_s\chi_i^M}{\sqrt{M_-}}\right\|^2_{2}ds\\
\leq& C\ka_0\eps\int_0^t\|[a,b,c](s)\|_2^2ds+C\eps^{-1}\int_0^t
\left\| \frac{P_1^MR}{\sqrt{M_-}}\right\|_{2}^2 ds.
\end{split}
\end{equation*}

Now we turn to compute the third term in the left hand side of \eqref{Gid}, for this, let us first write
\begin{equation*}%\label{trans}
\begin{split}
\int_{\Om\times \R^3}v\cdot \na_x \psi Rdxdv
=&\int_{\Om\times \R^3}(v-u)\cdot \na_x \psi Rdxdv+\int_{\Om\times \R^3}u\cdot \na_x \psi Rdxdv\\
=&\int_{\Om\times \R^3}(v-u)\cdot \na_x \psi P_0^MRdxdv+\int_{\Om\times \R^3}(v-u)\cdot \na_x \psi P_1^MRdxdv
\\&+\int_{\Om\times \R^3}u\cdot \na_x \psi P_0^MRdxdv+\int_{\Om\times \R^3}u\cdot \na_x \psi P_1^MRdxdv,
\end{split}
\end{equation*}
and perform the elementary calculations
\begin{equation*}%\label{papc}
\begin{split}
\pa_{j}\psi_{c,k}=&-2\pa_{j}u_{l}\frac{(v-u)_l}{\ta}(v-u)\cdot\na_x\varphi
_{c,k}(t,x)
-\pa_j\ta\frac{|v-u|^{2}}{\ta^2}(v-u)\cdot \nabla _{x}\varphi
_{c,k}(t,x)\\
&-\left(\frac{|v-u|^{2}}{\ta}-\beta _{c}\right)\pa_ju_i\pa_i\varphi
_{c,k}(t,x)+\left(\frac{|v-u|^{2}}{\ta}-\beta _{c}\right)(v-u)_i \pa_j\pa_i\varphi
_{c,k}(t,x).
\end{split}
\end{equation*}
We then have by using \eqref{P0R} that
\begin{equation}\label{papc1}
\begin{split}
\int_{\Om\times \R^3}&(v-u)\cdot \na_x \psi_{c,k} P_0^MRdxdv\\
=&-2\int_{\Om\times \R^3}(v-u)_j\pa_{j}u_{l}\frac{(v-u)_l}{\ta}(v-u)\cdot\na_x\varphi
_{c,k}(t,x)\frac{b(t,x)\cdot (v-u)}{\rho\ta}Mdxdv\\
&-\int_{\Om\times \R^3}(v-u)_j\pa_j\ta\frac{|v-u|^{2}}{\ta^2}(v-u)\cdot \nabla _{x}\varphi
_{c,k}(t,x)\left\{\frac{a(t,x)}{\rho}+\frac{c(t,x)}{\rho\ta}\left(\frac{|v-u|^2}
{\ta}-3\right)\right\}Mdxdv
\\
&-\int_{\Om\times \R^3}(v-u)_j\left(\frac{|v-u|^{2}}{\ta}-\beta _{c}\right)\pa_ju_i\pa_i\varphi
_{c,k}(t,x)\frac{b(t,x)\cdot (v-u)}{\rho\ta}Mdxdv
\\
&+\int_{\Om\times \R^3}(v-u)_j \left(\frac{|v-u|^{2}}{\ta}-\beta _{c}\right)(v-u)_i \pa_j\pa_i\varphi
_{c,k}(t,x)\frac{a(t,x)}{\rho}Mdxdv
\\
&+\int_{\Om\times \R^3}(v-u)_j \left(\frac{|v-u|^{2}}{\ta}-\beta _{c}\right)(v-u)_i \pa_j\pa_i\varphi
_{c,k}(t,x)\frac{c(t,x)}{\rho\ta}\left(\frac{|v-u|^2}
{\ta}-3\right)Mdxdv.
\end{split}
\end{equation}
By H$\ddot{o}$lder's inequality and the elliptic estimates \eqref{c22} and \eqref{epes}, we see that
the first, second and third line in the right hand side of \eqref{papc1} are bounded by
\begin{equation*}%\label{1-3line}
\begin{split}
C\|\na_x[u,\ta]\|_{\infty}\|\na_x\varphi_{c,k}\|_{k'}\|[a,b,c]\|_{k} \leq& C\|\na_x[u,\ta]\|_{\infty}\left\{\|c\|^k_{k}+\|[a,b,c]\|^k_{k}\right\}\\
\leq& C\ka_0\eps\left\{\|c\|^k_{k}+\|[a,b,c]\|^k_{k}\right\},
\end{split}
\end{equation*}
here $1/k'+1/k=1$.

For the fourth and fifth term,
let us choose $\beta_c=5$ so that
\begin{eqnarray*}%\label{4-5line}
\left\{\begin{array}{rll}
\begin{split}
&\int_{\Om\times \R^3}(v-u)^2_i \left(\frac{|v-u|^{2}}{\ta}-\beta _{c}\right)Mdxdv=0,\\
&\int_{\Om\times \R^3}(v-u)_i^2 \left(\frac{|v-u|^{2}}{\ta}-\beta _{c}\right)\pa^2_i\varphi
_{c,k}(t,x)\frac{c(t,x)}{\rho\ta}\left(\frac{|v-u|^2}
{\ta}-3\right)Mdxdv=-10\|c\|_k^k.
\end{split}\end{array}\right.
\end{eqnarray*}
Next, we get from H$\ddot{o}$lder's inequality, Young's inequality and the elliptic estimates \eqref{c22} and \eqref{epes} that
for $k=2$
\begin{equation*}%\label{papc2}
\begin{split}
\int_{\Om\times \R^3}(v-u)\cdot \na_x \psi_{c,2} P_1^MRdxdv
\lesssim& \|\na_x \psi_{c,2}\|\left\| \frac{P_1^MR}{\sqrt{M_-}}\right\|_2\\
\lesssim& \eps\|c\|_2^2+\eps\|\na_x[u,\ta]\|^2_\infty\|\na_x\varphi_{c,2}\|_2^2+\frac{1}{\eps}\left\| \frac{P_1^MR}{\sqrt{M_-}}\right\|_2^2
\\
\lesssim& \eps\|c\|_2^2+\frac{1}{\eps}\left\| \frac{P_1^MR}{\sqrt{M_-}}\right\|_2^2,
\end{split}
\end{equation*}
and for $k=6$
\begin{equation*}%\label{papc6}
\begin{split}
\int_{\Om\times \R^3}&(v-u)\cdot \na_x \psi_{c,6} P_1^MRdxdv\\
\lesssim& \|\na_x \psi_{c,k}\|_{6/5}\left\| \frac{P_1^MR}{\sqrt{M_-}}\right\|_6 +\|\na_x[u,\ta]\|_\infty\|\na_x\varphi_{c,6}\|_{2}\left\| \frac{P_1^MR}{\sqrt{M_-}}\right\|_2\\
\lesssim&  \eta\|c\|_6^6+C_\eta\left\| \frac{P_1^MR}{\sqrt{M_-}}\right\|_6^6+\|\na_x[u,\ta]\|_\infty\|c\|_6^6+\|\na_x[u,\ta]\|_\infty\left\| \frac{P_1^MR}{\sqrt{M_-}}\right\|_2^6
\\
\lesssim&  (\eta+\sqrt{\eps})\|c\|_6^6+C_\eta\eps^{6}\left\| \frac{P_1^MR}{\sqrt{M_-}}\right\|_\infty^6+C_\eta\eps^{-6}\left\| \frac{P_1^MR}{\sqrt{M_-}}\right\|_2^6+\ka_0\sqrt{\eps}\left\| \frac{P_1^MR}{\sqrt{M_-}}\right\|_2^6,
\end{split}
\end{equation*}
the last inequality resulting from
\begin{equation*}%\label{6-if-2}
\left\| \frac{P_1^MR}{\sqrt{M_-}}\right\|_6^6\lesssim \eps^{6}\left\| \frac{P_1^MR}{\sqrt{M_-}}\right\|_\infty^6+\eps^{-6}\left\| \frac{P_1^MR}{\sqrt{M_-}}\right\|_2^6.
\end{equation*}
%was used in the last inequality.
Similarly, one has
\begin{equation*}%\label{papc7}
\begin{split}
\int_{\Om\times \R^3}u\cdot \na_x \psi_{c,k} P_0^MRdxdv\lesssim& \|u\|_{\infty}
\|\na_x[u,\ta]\|_{\infty}\|\na_x\varphi_{c,k}\|_{k'}\|[a,b,c]\|_k+\|u\|_{\infty}
\|\na^2_x\varphi_{c,k}\|_{k'}\|[a,b,c]\|_k\\
\lesssim& \eps\|[a,b,c]\|^k_k,
\end{split}
\end{equation*}
and
\begin{equation*}%\label{papc8}
\begin{split}
\int_{\Om\times \R^3}u\cdot \na_x \psi_{c,k} P_1^MRdxdv\lesssim& \|u\|_{\infty}
\|\na_x[u,\ta]\|_{\infty}\|\na_x\varphi_{c,k}\|_{k'}\left\|\frac{P_1^MR}{\sqrt{M_-}}\right\|_k+\|u\|_{\infty}
\|\na^2_x\varphi_{c,k}\|_{k'}\left\|\frac{P_1^MR}{\sqrt{M_-}}\right\|_k\\
\lesssim& \eta\|c\|_k^k+C_\eta\eps^k\left\|\frac{P_1^MR}{\sqrt{M_-}}\right\|^k_k.
\end{split}
\end{equation*}
We now turn to deal with the boundary term $\int_{\ga}v\cdot nR\psi_{c,k} dS_xdv$. By employing \eqref{Rbddec},
we have
\begin{equation*}%\label{cbdterm}
\begin{split}
\int_{\ga}v\cdot nR\psi_{c,k} dS_xdv=\int_{\ga_+}v\cdot n (I-P_\ga)R\psi_{c,k} dS_xdv
+\eps^{-1/2}\int_{\ga_-}v\cdot n r\psi_{c,k} dS_xdv,
\end{split}
\end{equation*}
in view of
$$
\int_{\ga}v\cdot nP_\ga R\psi_{c,k} dS_xdv=0.
$$
Furthermore, for $k=2$, it follows from Cauchy-Schwarz's inequality and Sobolev trace inequality \eqref{sobinep2} that
\begin{equation*}%\label{cbdtermp1}
\begin{split}
\int_{\ga_+}v\cdot n (I-P_\ga)R\psi_{c,2} dS_xdv
\leq& \eta\|\na_x\varphi_{c,2}\|^2_{H^1}+C_\eta\int_{\ga_+}v\cdot n \frac{|(I-P_\ga)R|^2}{M_-} dS_xdv\\
\leq& \eta\|c\|_2^2+C_\eta\int_{\ga_+}v\cdot n \frac{|(I-P_\ga)R|^2}{M_-} dS_xdv,
\end{split}
\end{equation*}
and
\begin{equation*}%\label{cbdtermp2}
\begin{split}
\eps^{-1/2}\int_{\ga_-}v\cdot n r\psi_{c,2} dS_xdv
\leq& \eta\|\na_x\varphi_{c,2}\|^2_{H^1}+C_\eta\eps^{-1}\int_{\ga_-}|v\cdot n| \frac{|r|^2}{M_-} dS_xdv\\
\leq& \eta\|c\|_2^2+C_\eta\eps^{-1}\int_{\ga_-}|v\cdot n| \frac{|r|^2}{M_-} dS_xdv,
\end{split}
\end{equation*}
The corresponding estimates for $k=6$ are somehow different.  Thanks to the trace inequality
\be \notag
\left(\int_{\pa \Om} d S_x |f|^{\frac{p(N-1)}{N-p}}\right)^{\frac{N-p}{p(N-1)}}\leq C(N,p)\left( {\int_{\Om} dx |f|^{p}}+\int_{\Om} dx |\nabla f|^{p}\right)^{\frac{1}{p}}
\ee
with $N=3$ and $p=6/5$, we deduce
\begin{equation}\label{cbdtermp3}
\begin{split}
\int_{\ga_+}v\cdot n (I-P_\ga)R\psi_{c,6} dS_xdv
\lesssim& |\na_x\varphi_{c,6}|_{4/3}\left(\int_{\ga_+}v\cdot n \frac{|(I-P_\ga)R|^4}{M^2_-} dS_xdv\right)^{1/4}\\
\lesssim& \|\na_x\varphi_{c,6}\|_{W^{1,6/5}}\left(\int_{\ga_+}v\cdot n \frac{|(I-P_\ga)R|^4}{M^2_-} dS_xdv\right)^{1/4}\\
\lesssim& \|c^5\|_{6/5}\left(\int_{\ga_+}v\cdot n \frac{|(I-P_\ga)R|^4}{M^2_-} dS_xdv\right)^{1/4}\\
\lesssim& \eta\|c\|_{6}^6+C_\eta\left(\int_{\ga_+}v\cdot n \frac{|(I-P_\ga)R|^4}{M^2_-} dS_xdv\right)^{3/2}\\
\lesssim& \eta\|c\|_{6}^6+C_\eta\left\|\frac{(I-P_\ga)R}{\sqrt{M_-}}\right\|^3_{\infty}\left(\int_{\ga_+}v\cdot n \frac{|(I-P_\ga)R|^2}{M_-} dS_xdv\right)^{3/2}\\
\lesssim& \eta\|c\|_{6}^6+\eta'C_\eta\eps^3\left\|\frac{(I-P_\ga)R}{\sqrt{M_-}}\right\|^6_{\infty}+C_{\eta'}C_\eta\eps^{-3}\left(\int_{\ga_+} \frac{|(I-P_\ga)R|^2}{M_-} d\ga\right)^{3},
\end{split}
\end{equation}
and
\begin{equation*}%\label{cbdtermp4}
\begin{split}
\eps^{-1/2}\int_{\ga_-}v\cdot n r\psi_{c,6} dS_xdv\lesssim &\eps^{-1/2}
|\na_x\varphi_{c,6}|_{4/3}\left(\int_{\ga_-}|v\cdot n| \frac{|r|^4}{M^2_-} dS_xdv\right)^{1/4}\\
\lesssim &\eps^{-1/2}
\|\na_x\varphi_{c,6}\|_{W^{1,6/5}}\left(\int_{\ga_-}|v\cdot n| \frac{|r|^4}{M^2_-} dS_xdv\right)^{1/4}\\
\lesssim&\eta\|c\|_{6}^6+C_\eta\eps^{-3}\left(\int_{\ga_-}|v\cdot n| \frac{|r|^4}{M^2_-} dS_xdv\right)^{3/2}
\\
\lesssim&\eta\|c\|_{6}^6+C_\eta\eps^{-3}|\na_x[u,\ta]|_4^6\lesssim \eta\|c\|_{6}^6+C_\eta\eps^{-3}\|\na_x[u,\ta]\|_{H^1}^6.
\end{split}
\end{equation*}
Moreover, as \eqref{eps-1eng}, the boundary term in the right hand side of \eqref{cbdtermp3} can be dominated as follows
\begin{equation*}%\label{eps-2eng}
\begin{split}
\sup\limits_{0\leq s\leq t}&\int_{\ga_+} \frac{|(I-P_\ga)R|^2}{M_-} d\ga\\
\leq& \left| \frac{(I-P_\ga)R_0}{\sqrt{M_-}}\right|^2_{2,+}+2\int_0^tds
\left| \frac{(I-P_\ga)R}{\sqrt{M_-}}\right|_{2,+}\left| \frac{(I-P_\ga)\pa_sR}{\sqrt{M_-}}\right|_{2,+}\\
\leq& \left| \frac{(I-P_\ga)R_0}{\sqrt{M_-}}\right|^2_{2,+}+\int_0^t
\left| \frac{(I-P_\ga)R}{\sqrt{M_-}}\right|_{2,+}^2 ds+\int_0^t
\left| \frac{(I-P_\ga)\pa_sR}{\sqrt{M_-}}\right|^2_{2,+}ds.
\end{split}
\end{equation*}

To complete the estimates for $c(t,x)$, it remains now to deal with the term $(\pa_tR,\psi_{c,k})$.
For $k=6$, recall $R=P_1^MR+P_0^MR$,
H$\ddot{o}$lder's inequality and Cauchy-Schwarz's inequality yield
\begin{equation*}%\label{ptc61}
\begin{split}
(\pa_tP_1^MR,\psi_{c,6})\lesssim& \left\|M_-^{1/2}\psi_{c,6}\right\|_2\left\|\frac{\pa_tP_1^MR}{\sqrt{M_-}}\right\|_2
\lesssim \left\|\na_x\varphi_{c,6}\right\|_2\left\|\frac{\pa_tP_1^MR}{\sqrt{M_-}}\right\|_2\\
\lesssim& \left\|c^5\right\|_{6/5}\left\|\frac{\pa_tP_1^MR}{\sqrt{M_-}}\right\|_2\lesssim \eta\|c\|_6^6+C_\eta\left\|\frac{\pa_tP_1^MR}{\sqrt{M_-}}\right\|_2^6.
\end{split}
\end{equation*}
Note that
\begin{equation}\label{pamim}
\begin{split}
\pa P_0^M R=&\frac{\pa a}{\rho }M+\frac{\pa b\cdot (v-u)}{\rho \ta}M+\left(\frac{|v-u|^2}{2\ta}-\frac{3}{2}\right)\frac{\pa c}{\rho\ta}M\\
&+\frac{a}{\rho}M\frac{v-u}{\ta}\cdot \pa u+\frac{a}{\rho}M\left(\frac{|v-u|^2}{2\ta}-\frac{3}{2}\right)\frac{\pa \ta}{\ta}-\frac{b\cdot\pa u}{\rho\ta}M\\
&+\frac{b\cdot (v-u)(v-u)\cdot \pa u}{\rho\ta}M+\frac{b\cdot(v-u)}{\rho\ta}
\left(\frac{|v-u|^2}{2\ta}-\frac{5}{2}\right)\frac{\pa \ta}{\ta}M\\
&+\left(\frac{|v-u|^2}{2\ta}-\frac{3}{2}\right)\frac{c}{\rho\ta}\frac{v-u}{\ta}\cdot \pa u M
+\left(\frac{|v-u|^2}{2\ta}-\frac{3}{2}\right)\left(\frac{|v-u|^2}{2\ta}-\frac{5}{2}\right)\frac{c}{\rho\ta}\frac{\pa \ta}{\ta}M
\\&-\frac{|v-u|^2}{2\ta}\frac{c}{\rho\ta}\frac{\pa \ta}{\ta}M,
\end{split}
\end{equation}
here $\pa=\pa_t\ \text{or}\ \pa_{x_j} (j=1,2,3).$ With this, we see that
\begin{equation*}%\label{ptc62}
\begin{split}
(\pa_tP_0^MR,\psi_{c,6})=&\left(\frac{b\cdot(v-u)}{\rho\ta}
\left(\frac{|v-u|^2}{2\ta}-\frac{5}{2}\right)\frac{\pa_t \ta}{\ta}M+\left(\frac{|v-u|^2}{2\ta}-\frac{3}{2}\right)\frac{c}{\rho\ta}\frac{v-u}{\ta}\cdot \pa_t u M,\psi_{c,6}\right)\\
\lesssim& \left\|M_-^{1/2}\psi_{c,6}\right\|_2\|[b,c]\|_6\|\pa_t\ta\|_3
\lesssim \left\|\na_x\varphi_{c,6}\right\|_2\|[b,c]\|_6\|\pa_t[u,\ta]\|_3\\
\lesssim& \left\|c^5\right\|_{6/5}\|[b,c]\|_6\|\pa_t[u,\ta]\|_3\lesssim \ka_0\eps\|[b,c]\|_6^6.
\end{split}
\end{equation*}
The corresponding estimates for $k=2$ will be more complicate, since we have to deal with the integral
$\int_0^t(\pa_tR,\psi_{c,2})ds.$ For this, we first choose a test function $\varphi(x)\left(\frac{|v-u|^2}{2\ta}-\frac{3}{2}\right)$, then in the sense of distribution, we get from \eqref{Rln} and \eqref{pamim}
that
\begin{equation*}%\label{pach-1}
\begin{split}
\frac{3}{2}(\pa_t c, \varphi)&+(a\pa_t\ta+b\cdot\pa_tu,\varphi)-(b,\na_x\varphi)+\frac{5}{2}\left(b\cdot\na_x\ta,\varphi\right)
+((a+c)\na_x\cdot u,\varphi)\\
&+\frac{3}{2}(u\cdot\na_x\varphi,c)-\frac{3}{2}\left(\frac{a+c}{\ta}u\cdot\na_x\ta,\varphi\right)+(u\cdot\na_xu \cdot b,\varphi)\\
&+\left(v\cdot \na_x\varphi\left(\frac{|v-u|^2}{2\ta}-\frac{3}{2}\right),P_1^MR\right)-\left(v\cdot \na_x\ta\varphi\frac{|v-u|^2}{2\ta^2},P_1^MR\right)\\&+\left(v\cdot \na_xu \cdot\frac{(v-u)}{\ta}\varphi,P_1^MR\right)=0,
\end{split}
\end{equation*}
from which, one directly has
\begin{equation*}%\label{pach-2}
\begin{split}
|(\pa_t c, \varphi)|\lesssim \left\{\|b\|_2+\eps\|[a,c]\|_2+\left\|\frac{P_1^MR}{\sqrt{M_-}}\right\|_2\right\}\|\varphi\|_{H^1_0},
\end{split}
\end{equation*}
hence
\begin{equation}\label{pach-3}
\begin{split}
\|\pa_tc\|^2_{H^{-1}}\lesssim \|b\|_2^2+\eps\|[a,c]\|_2^2+\left\|\frac{P_1^MR}{\sqrt{M_-}}\right\|_2^2.
\end{split}
\end{equation}
By integration by parts, we obtain
\begin{equation*}%\label{patRc}
\begin{split}
\int_0^t(\pa_tR,\psi_{c,2})ds=(R,\psi_{c,2})(t)-(R,\psi_{c,2})(0)-\int_0^t(R,\pa_t\psi_{c,2})ds.
\end{split}
\end{equation*}
We next employ this computation in \eqref{pach-3}, to deduce
\begin{equation*}%\label{patRcp2}
\begin{split}
|(R,\pa_t\psi_{c,2})|\leq&\left|\left(R,\left(\frac{|v-u|^{2}}{\ta}-5\right)(v-u)\cdot \nabla _{x}\pa_t\varphi
_{c,2}\right)\right|
+\left|\left(R,\left(\frac{|v-u|^{2}}{\ta}-5\right) \pa_tu\cdot \nabla _{x}\varphi
_{c,2}\right)\right|
\\&+\left|\left(R,\frac{2(v-u)\cdot \pa_t u}{\ta}(v-u)\cdot \nabla _{x}\varphi_{c,2}\right)\right|+\left|\left(R,\frac{|v-u|^2 \pa_t \ta}{\ta^2}(v-u)\cdot \nabla _{x}\varphi
_{c,2}\right)\right|\\
\lesssim &C_\eta\left\|\frac{P_1^MR}{\sqrt{M_-}}\right\|_2^2+\eta\|\nabla _{x}\Delta^{-1}\pa_t c\|_2^2+\|\pa_t[u,\ta]\|_{\infty}\left\{\|\varphi_{c,2}\|_2^2+\|[a,b,c]\|_2^2\right\}\\
\lesssim &C_\eta\left\|\frac{P_1^MR}{\sqrt{M_-}}\right\|_2^2+\eta\|b\|_2^2+\eps\|[a,b,c]\|_2^2,
\end{split}
\end{equation*}
here the decomposition $R=P_1^MR+P_0^MR$ was used.

Combing all the estimates above together, we arrive at
\begin{equation}\label{cl2}
\begin{split}
\eps\int_0^t\|c\|_2^{2}ds \lesssim&
\eps|(R,\psi_{c,2})(t)-(R,\psi_{c,2})(0)|+
\eps^{-1}\int_0^t\int_{\Om\times \R^3}\frac{\nu_M(P_1^MR)^2}{M_-}dxdvds\\&+\eps\int_0^t\int_{\ga_+}\frac{|(I-P_\ga)R|^2}{M_-}d\ga ds
+\int_0^t\int_{\ga_-} \frac{|r|^2}{M_-} d\ga ds\\&+\eps\int_0^t\int_{\Om\times \R^3}\frac{\nu_M^{-1}g^2}{M_-}dxdvds
+(\eta+\eps)\eps\int_0^t\|[a,b]\|_2^{2}ds,
\end{split}
\end{equation}
and
\begin{equation*}%\label{P1-P6-c}
\begin{split}
\eps\sup\limits_{0\leq s\leq t}\| c(s)\|^2_{6}
 \lesssim& \int_{\ga_+}\frac{|(I-P_\ga)R_0|^2}{M_-}d\ga+\eps^{-1}\int_{\Om\times \R^3}\frac{\nu_M(P_1^MR_0)^2}{M_-}dxdv
 \\&+ \eps^{-1}\sum\limits_{\al_0\leq1}\int_0^t\int_{\Om\times \R^3}\frac{\nu_M(P_1^M\pa^{\al_0}_tR)^2}{M_-}dxdvds\\&+C_{\eta,\eta'}
 \sum\limits_{\al_0\leq1}\int_0^t\int_{\ga_+}\frac{|(I-P_\ga)\pa^{\al_0}_tR|^2}{M_-}d\ga ds+(\eta+\eta'+\eps)\eps^2\sup\limits_{0\leq s\leq t}\left\|\frac{R(s)}{\sqrt{M_-}}\right\|_\infty^2\\
&+\eps\sup\limits_{0\leq s\leq t}\int_{\Om\times \R^3}\frac{\nu_M^{-1}g^2}{M_-}dxdv
+\eps^{4/3}\sup\limits_{0\leq s\leq t}\| [a,b](s)\|^2_{6}
+\sup\limits_{0\leq s\leq t}\|\na_x[u,\ta]\|_{H^1}^2\\&+\ka_0^2\eps\int_0^t\|[a,b,c](s)\|_2^2ds.
\end{split}
\end{equation*}
To complete the estimates for $c(t,x)$, it remains now to estimate $\int_0^t\|\pa_tc\|^{2}ds$. To do this, we begin with
the following equations
\begin{eqnarray*}%\label{tGid}
\left\{\begin{array}{rll}
\begin{split}
&(\pa^2_t R,\psi)+\int_{\ga}v\cdot n\pa_tR\psi dS_xdv-\int_{\Om\times \R^3}v\cdot \na_x \psi \pa_tRdxdv=-\frac{1}{\eps}(\pa_tL_MR,\psi)+(\pa_tg,\psi),\\
&\pa_tR|_\ga={\bf 1}_{\ga_+}(I-P_\ga)\pa_tR+P_\ga \pa_tR+\eps^{-1/2}{\bf 1}_{\ga_-}\pa_tr,
\end{split}
\end{array}\right.
\end{eqnarray*}
here as in \eqref{Gid} $\psi$ is a test function.
Once again $\pa_tR$ can be decomposed as
\begin{equation*}%\label{Rdec2}
\begin{split}
\pa_tR=&P_0^M\pa_tR+P_1^M\pa_tR\\
=&\frac{\pa_ta(t,x)}{\rho}M+\frac{\pa_tb(t,x)\cdot (v-u)}{\rho\ta}M
+\frac{\pa_tc(t,x)}{\rho\ta}\left(\frac{|v-u|^2}{\ta}-3\right)M\\
&+a(t,x)\sum\limits_i\left\langle\pa_t\left(\frac{M}{\rho}\right),\chi_i/M\right\rangle\chi_i
+\sum\limits_ib(t,x)\cdot\left\langle\pa_t\left(\frac{ (v-u)}{\rho\ta }M\right),\chi_i/M\right\rangle\chi_i
\\&+c(t,x)\sum\limits_i\left\langle\pa_t\left(\left(\frac{|v-u|^2}{\ta}-3\right)\frac{M}{\rho\ta}\right),\chi_i/M\right\rangle\chi_i
+P_1^M \pa_tR.
\end{split}
\end{equation*}
Next, we set
\begin{equation*}
\psi =\psi_{c,t}\equiv \left(\frac{|v-u|^{2}}{\ta}-5\right)(v-u)\cdot \nabla _{x}\varphi
_{c,t}(t,x),  \ \ \textrm{where} \    -\Delta _{x}\varphi _{c,t}(x)=\pa_tc(x),\ \  \varphi _{c,t}|_{\partial \Omega }=0. %\label{tphic}
\end{equation*}
Then performing the similar calculations as for obtaining \eqref{cl2}, one has
\begin{equation*}%\label{P1-P2-c}
\begin{split}
\eps\int_0^t\|\pa_tc\|_2^{2}ds \lesssim&
\eps|(\pa_tR,\psi_{c,t})(t)-(\pa_tR,\psi_{c,t})(0)|+
\eps^{-1}\sum\limits_{\al_0\leq1}\int_0^t\int_{\Om\times \R^3}\frac{\nu_M(P_1^M\pa_t^{\al_0}R)^2}{M_-}dxdvds
\\&+\eps\sum\limits_{\al_0\leq1}\int_0^t\int_{\ga_+}\frac{|(I-P_\ga)\pa_t^{\al_0}R|^2}{M_-}d\ga ds
+\sum\limits_{\al_0\leq1}\int_0^t\int_{\ga_-} \frac{|\pa_t^{\al_0}r|^2}{M_-} d\ga ds\\&+\eps\sum\limits_{\al_0\leq1}\int_0^t\int_{\Om\times \R^3}\frac{\nu_M^{-1}\pa_t^{\al_0}g^2}{M_-}dxdvds
+(\eta+\eps)\eps\int_0^t\|\pa_t[a,b]\|_2^{2}ds\\&+\eps^2\int_0^t\|[a,b,c]\|_2^{2}ds.
\end{split}
\end{equation*}
Thus, the proof for Lemma \ref{l2-l6} is finished.
\end{proof}

\section{$L^\infty$ estimates}%\label{lif-th}

In this section, we deduce the $L^\infty$ estimates for the solutions of the remainder equation \eqref{R}.
We begin our analysis of the linear equation
\begin{eqnarray}\label{Rln2}
\left\{\begin{split}
&\pa_tR+v\cdot\na_xR+\frac{1}{\eps}\nu_MR=\frac{1}{\eps}K_MR+g,\\
&R(0,x,v)= R_0(x,v),\\
&R_{-}=P_\ga R+\eps^{-1/2}r.
\end{split}\right.
\end{eqnarray}
The characteristic method is employed to obtain the $L^\infty$ estimates of solutions of the above equations. The key point in this line is that
a ``smallness" can be gained once the the particle satifying the diffusive boundary condition $\eqref{Rln2}_3$ hits the boundary frequently.
Given $(t,x,v)$, we let $[X(s),V(s)]$ satisfy
\begin{equation*}%\label{ch.eq}
\frac{dX(s)}{ds}=V(s),\ \ \frac{dV(s)}{ds}=0,
\end{equation*}
with the initial data $[X(t;t,x,v),V(t;t,x,v)]=[x,v]$. Then $[X(s;t,x,v),V(s;t,x,v)]$$=[x-(t-s)v,v]$$=[X(s),V(s)]$, which is called as the backward characteristic trajectory for the Boltzmann equation
\eqref{BE}. For $%
(x,v)\in \Om\times \R^{3}$, the \textit{backward exit time} $t_{\mathbf{b}}(x,v)>0$ is defined to be the first moment at which the backward characteristic line
$[X(s;0,x,v),V(s;0,x,v)]$ emerges from $\Om$:
\begin{equation*}%\label{b.t}
t_{\mathbf{b}}(x,v)=\inf \{\ t> 0:x-tv\notin \Omega \},
\end{equation*}
%\label{backexit}
%\end{equation}%
and we also define $x_{\mathbf{b}}(x,v)=x-t_{\mathbf{b}}(x,v)v\in \partial \Omega $. Note that for any $(x,v)$, we use $t_{\mathbf{b}}(x,v)$
whenever it is well-defined.

%We then
%define the following stochastic cycles.
\begin{definition}[Stochastic Cycles]
%\label{diffusecycles}
Fixed any point $(t,x,v)$ with $(x,v)\notin \gamma _{0}$%
, let $(t_{0},x_{0},v_{0})$ $=(t,x,v)$. For $v_{k+1}$ such
that $v_{k+1}\cdot n(x_{k+1})>0$, define the $(k+1)$-component of the
back-time cycle as
\begin{equation*}
(t_{k+1},x_{k+1},v_{k+1})=(t_{k}-t_{\mathbf{b}}(x_{k},v_{k}),x_{%
\mathbf{b}}(x_{k},v_{k}),v_{k+1}). %\label{diffusecycle}
\end{equation*}%
Set
\begin{eqnarray*}
X_{\mathbf{cl}}(s;t,x,v) &=&\sum_{k}\mathbf{1}_{[t_{k+1},t_{k})}(s)%
\{x_{k}+(s-t_{k})v_{k}\}, \\
V_{\mathbf{cl}}(s;t,x,v) &=&\sum_{k}\mathbf{1}%
_{[t_{k+1},t_{k})}(s)v_{k}.
\end{eqnarray*}%
Define $\mathcal{V}_{k+1}=\{v\in \R^{3}\ |\ v\cdot
n(x_{k+1})>0\}$, and let the iterated integral for $k\geq 2$ be defined as
\begin{equation*}
\int_{\Pi_{j=1}^{k-1}\mathcal{V}_{j}}\cdots\Pi_{j=1}^{k-1}d\sigma_{j}\equiv \int_{\mathcal{V}_1}\cdots\Big\{\int_{\mathcal{V}_{k-1}}d\sigma_{k-1}\Big\} d\sigma_1,  %\label{sigma}
\end{equation*}
where $d\sigma _{j}=M^w (v)(n(x_{j})\cdot v)dv$ is a probability
measure.
\end{definition}
The following lemma is borrowed from \cite[Lemma 23, pp.781]{Guo-2010} and \cite[Lemma 3.3, pp.489]{LY-2016}
\begin{lemma}\label{k}
Let $T_{0}>0$ and large enough, there exist constants $%
C_{1},C_{2}>0$ independent of $T_0$, such that for $k=C_1T_0^{5/4}$,
and $(t,x,v)\in \lbrack 0,\infty)\times \overline{\Omega }\times \R^{3},$
\begin{equation*}
\int_{\Pi _{j=1}^{k-1}\mathcal{V}_{j}}\mathbf{1}_{\{t_{k}(t,x,v,v
_{1},v_{2},\cdots ,v_{k-1})>0\}}\Pi _{j=1}^{k-1}d\sigma _{j}\leq
\left\{ \frac{1}{2}\right\} ^{C_{2}T_0^{5/4}}.  %\label{largek}
\end{equation*}
Moreover, there exist constants $C_3,\ C_4>0$ independent of $k$ such that
\begin{equation*}%\label{ktildew1}
\begin{split}
\int_{\Pi _{j=1}^{k-1}\mathcal{V}_{j}}\sum_{l=1}^{k-1}\mathbf{1}_{\{t_{l+1}\leq 0<t_{l}\}}
\int_0^{t_l} d\Sigma^w_l(s)ds\leq C_3,
\end{split}
\end{equation*}
and
\begin{equation*}%\label{ktildew2}
\begin{split}
\int_{\Pi _{j=1}^{k-1}\mathcal{V}_{j}}\sum_{l=1}^{k-1}\mathbf{1}_{\{t_{l+1}>0\}}\int_{t_{l+1}}^{t_l} d\Sigma^w_l(s)ds\leq
C_4,
\end{split}
\end{equation*}
where
\begin{eqnarray*}%\label{measure}
%d\Sigma _{l} &=&\{\Pi _{j=l+1}^{k-1}d\sigma _{j}\}\times\{\tilde{w}_{q,\ta}%
%(v_{l})d\sigma _{l}\}\times \Pi _{j=1}^{l-1} d\sigma _{j},  \\
d\Sigma^w_{l}(s) &=&\{\Pi _{j=l+1}^{k-1}d\sigma_{j}\}\times\{e^{\nu_M
(v_{l})(s-t_{l})}\tilde{w}_{\ell}(v_{l})d\sigma_{l}\}\times \Pi _{j=1}^{l-1}\{{{%
e^{\nu_M (v_{j})(t_{j+1}-t_{j})} d\sigma_{j}}}\},
\end{eqnarray*}
and
$
\widetilde{w}_{\ell}=\frac{1}{w_{\ell}\sqrt{M_-}}.
$
\end{lemma}

\begin{proposition}\label{point_dyn} Let $R$ satisfy \eqref{Rln2}.
Then, for $w_\ell(v) = (1+|v|^2)^{\ell/2}$ with $\ell\geq 0$, it holds that
\begin{equation}\label{point1}
\begin{split}
 \eps\left\|w_\ell \frac{R(t)}{\sqrt{M_-}}\right\|_{\infty}
  \lesssim&
   \eps\left\|  w_\ell \frac{R_{0}}{\sqrt{M_-}} \right\|_{\infty} + \eps^{\frac 1 2} \sup_{0 \leq s \leq t}\left\|  w_\ell \frac{r(s)}{\sqrt{M_-}} \right\|_{\infty} + \eps^{2} \sup_{0 \leq s \leq t} \left\|  w_{\ell-1} \frac{g(s)}{\sqrt{M_-}}\right\|_{\infty}\\
  &+  \sup_{0\leq s \leq  t}\eps^{1/2}\left\|\frac{P_0^M R(s)}{\sqrt{M_-}}\right\|_{6} +
  \eps^{-1/2}  \sup_{0 \leq s \leq  t}\left\|\frac{P_1^M R(s)}{\sqrt{M_-}}\right\|_2,
    \end{split}
\end{equation}
\begin{equation}\label{point2}
\begin{split}
\eps\left\|w_\ell \frac{R(t)}{\sqrt{M_-}}\right\|_{\infty}
\lesssim&
\eps\left\|  w_\ell \frac{R_{0}}{\sqrt{M_-}} \right\|_{\infty} + \eps^{\frac 1 2} \sup_{0 \leq s \leq t}\left\|  w_\ell \frac{r(s)}{\sqrt{M_-}} \right\|_{\infty} + \eps^2 \sup_{0 \leq s \leq t} \left\|  w_{\ell-1} \frac{g(s)}{\sqrt{M_-}}\right\|_{\infty}\\
&+\eps^{-1/2}  \sup_{0 \leq s \leq  t}\left\|\frac{ R(s)}{\sqrt{M_-}}\right\|_2,
\end{split}
\end{equation}
and
\begin{equation}\label{point3}
\begin{split}
\eps\left\|w_\ell \frac{\pa_tR(t)}{\sqrt{M_-}}\right\|_{\infty}
\lesssim&
\eps\left\|  w_\ell \frac{\pa_tR_{0}}{\sqrt{M_-}} \right\|_{\infty} + \eps^{\frac 1 2} \sup_{0 \leq s \leq t}\left\|  w_\ell \frac{\pa_tr(s)}{\sqrt{M_-}} \right\|_{\infty} + \eps^2 \sup_{0 \leq s \leq t} \left\|  w_{\ell-1} \frac{\pa_tg(s)}{\sqrt{M_-}}\right\|_{\infty}\\
&+\eps^{-1/2}  \sup_{0 \leq s \leq  t}\left\|\frac{ \pa_tR(s)}{\sqrt{M_-}}\right\|_2
+\ka_0\eps\sup_{0 \leq s \leq t}\left\|  w_\ell \frac{R(s)}{\sqrt{M_-}} \right\|_{\infty}.
\end{split}
\end{equation}

\end{proposition}
\begin{proof}
Define $h=\frac{w_\ell R}{\sqrt{M_-}}$, $K_{M,w}(\cdot)=\frac{w_\ell}{\sqrt{M_-}}K_M \left(\frac{\sqrt{M_-}}{w_\ell}\cdot\right)$ and $\widetilde{w}_\ell=\frac{1}{w_\ell\sqrt{M_-}}$, then we see from \eqref{Rln2} that $h$ satisfies
\begin{eqnarray*}%\label{hln}
\left\{\begin{array}{rll}\begin{split}
&\pa_th+v\cdot\na_xh+\frac{1}{\eps}\nu_Mh=\frac{1}{\eps}K_{M,w}h+\frac{w_\ell}{\sqrt{M_-}}g,\\
&h(0,x,v)= h_0(x,v),\\
&h_{-}=\frac{w_\ell M^w}{\sqrt{M_-}}\int_{n(x)\cdot v'>0}\left(\frac{h}{w_\ell}\sqrt{M_-}\right)(t,x,v')(n(x)\cdot v')dv'+
\eps^{-1/2}\frac{w_\ell r}{\sqrt{M_-}}\\&
\quad=\frac{1}{\widetilde{w}_\ell}\int_{n(x)\cdot v'>0}h(t,x,v')\widetilde{w}_\ell(v')M^w(v')(n(x)\cdot v')dv'+
\eps^{-1/2}\frac{w_\ell r}{\sqrt{M_-}}\\&
\quad\eqdef
P^w_\ga R+\eps^{-1/2}\frac{w_\ell r}{\sqrt{M_-}}.
\end{split}\end{array}\right.
\end{eqnarray*}
%For simplicity, we as follows denote $K_{M,w}(\cdot)=\frac{w_l}{\sqrt{M_-}}K \left(\frac{\sqrt{M_-}}{w_l}\cdot\right)$ and $\widetilde{w}_l=\frac{1}{w_l\sqrt{M_-}}$.

We now intend to show for $t\in[nT_0,(n+1)T_0]$ with any $n\in\N$ and $T_0$ in Lemma \ref{k} that
\begin{equation}\label{claim1}
\begin{split}
 \eps\left\|h(t)\right\|_{\infty}
  \leq&CT_0^{5/2}e^{-\frac{\nu_0(t-nT_0)}{2\eps}}
   \eps\left\|  h(nT_0) \right\|_{\infty} +  C_{T_0}\sup_{0 \leq s \leq t}\eps^{\frac 1 2}\left\|  w_\ell \frac{r(s)}{\sqrt{M_-}} \right\|_{\infty} \\&  +  C_{T_0}\sup_{0 \leq s \leq t} \eps^{2}\left\|  w_{\ell-1} \frac{g(s)}{\sqrt{M_-}}\right\|_{\infty}+  CT_0^{5/2}\sup_{nT_0\leq s \leq  (n+1)T_0}\eps^{1/2}\left\|\frac{P_0^M R(s)}{\sqrt{M_-}}\right\|_{6} \\
  &+
 CT_0^{5/2} \eps^{-1/2}  \sup_{nT_0\leq s \leq  (n+1)T_0}\left\|\frac{P_1^M R(s)}{\sqrt{M_-}}\right\|_2
 \\&+\left[CT_0^{5/4}\left(\frac{1}{2}\right)^{C_2T_0^{5/4}}+o(1)CT_0^{5/2}\right]\sup_{nT_0\leq s \leq  (n+1)T_0} \eps\left\|h(s)\right\|_{\infty},
\end{split}
\end{equation}
and
\begin{equation}\label{claim2}
\begin{split}
 \eps\left\|h(t)\right\|_{\infty}
  \leq&CT_0^{5/2}e^{-\frac{\nu_0(t-nT_0)}{2\eps}}
   \eps\left\|  h(nT_0) \right\|_{\infty} +  C_{T_0}\sup_{0 \leq s \leq t}\eps^{\frac 1 2}\left\|  w_\ell \frac{r(s)}{\sqrt{M_-}} \right\|_{\infty} \\&  +  C_{T_0}\sup_{0 \leq s \leq t} \eps^{2}\left\|  w_{\ell-1} \frac{g(s)}{\sqrt{M_-}}\right\|_{\infty}+  CT_0^{5/2}\sup_{nT_0\leq s \leq  (n+1)T_0}\eps^{-1/2}\left\|\frac{ R(s)}{\sqrt{M_-}}\right\|_{2} \\
  &+\left[CT_0^{5/4}\left(\frac{1}{2}\right)^{C_2T_0^{5/4}}+o(1)CT_0^{5/2}\right]\sup_{nT_0\leq s \leq  (n+1)T_0} \eps\left\|h(s)\right\|_{\infty},
\end{split}
\end{equation}
where $\nu_0$ is a constant such that $\nu_M\geq\nu_0$ for all $v\in \R^3$ and $[\rho,u,\ta]\in\FX_\eps.$

To prove \eqref{claim1} and \eqref{claim2}, we only consider the case $n=0$, $n\geq1$ being similar. To do this,  along the stochastic cycles, we first have for $k=C_1T_0^{5/4}$,
\begin{equation}\label{iteration}
\begin{split}
  |h(t,x,v)| \leq&
\underbrace{\left\{\mathbf{1}_{t_{1}\leq 0}
\int_{0}^{t}+\mathbf{1}_{t_{1}>0}
\int_{t_1}^{t}\right\}
e^{-\frac{\nu_M(v)}{\eps}(t-s)}\frac{1}{\eps}|K_{M,w}h(s,x-(t-s){v},v)|ds}_{I_1}
\\
&
+\underbrace{\left\{\mathbf{1}_{t_{1}\leq 0}
\int_{0}^{t}+\mathbf{1}_{t_{1}>0}
\int_{t_1}^{t}\right\}e^{-\frac{\nu_M(v)}{\eps}(t-s)}\left|\frac{w_\ell}{\sqrt{M_-}}g(s,x-(t-s){v},v)\right|ds}_{I_2}
\\&
+\underbrace{\mathbf{1}_{t_{1}>0}\eps^{-1/2}e^{-\frac{\nu_M(v)}{\eps}(t-t_1)}\frac{w_\ell }{\sqrt{M_-}}(v)|r(t_1,x_1,v)|}_{I_3}+
\sum\limits_{n=4}^{8}I_n,
\end{split}
\end{equation}%
with
\begin{equation*}
\begin{split}
I_4=&\mathbf{1}_{t_{1}\leq 0}e^{-\frac{\nu_M(v)}{\eps}t}|h(0,x-t{v},v)|
\\&+\frac{e^{-\frac{\nu_M (v)}{\eps}(t-t_{1})}}{\widetilde{w}_l}\int_{\prod_{j=1}^{k-1}%
\mathcal{V}_{j}}\sum_{l=1}^{k-1}\mathbf{1}_{\{t_{l+1}\leq
0<t_{l}\}}|h(0,x_{l}-t_{l}{v}_{l},v_{l})|d\Sigma _{l}(0),
%\label{I1}
\end{split}
\end{equation*}
\begin{equation*}
\begin{split}
I_5=&\frac{e^{-\frac{\nu_M (v)}{\eps}(t-t_{1})}}{\widetilde{w}_\ell}\bigg\{\int_{\prod_{j=1}^{k-1}%
\mathcal{V}_{j}}\mathbf{1}_{\{t_{l+1}\leq
0<t_{l}\}}\sum_{l=1}^{k-1}\int_{0}^{t_l}\frac{1}{\eps}|[K_{M,w} h](s,x_{l}-(t_{l}-s){v}_{l},v_{l})|d\Sigma _{l}(s)ds
\\&+\int_{\prod_{j=1}^{k-1}%
\mathcal{V}_{j}}\sum_{l=1}^{k-1}\mathbf{1}_{\{0<t_{l+1}\}}\int_{t_{l+1}}^{t_{l}}\frac{1}{\eps}
|[K_{M,w}h](s,x_{l}-(t_{l}-s){v}_{l},v_{l})|d\Sigma _{l}(s)ds\bigg\},
\end{split}
\end{equation*}
\begin{equation*}
\begin{split}
I_6=&\frac{e^{-\frac{\nu_M (v)}{\eps}(t-t_{1})}}{\widetilde{w}_\ell}\bigg\{\int_{\prod_{j=1}^{k-1}%
\mathcal{V}_{j}}\sum_{l=1}^{k-1}\mathbf{1}_{\{t_{l+1}\leq
0<t_{l}\}}\int_{0}^{t_l}\left|\frac{w_\ell}{\sqrt{M_-}}g(s,x_{l}-(t_{l}-s){v}_{l},v_{l})\right|d\Sigma _{l}(s)ds \\&+\int_{\prod_{j=1}^{k-1}%
\mathcal{V}_{j}}\sum_{l=1}^{k-1}\mathbf{1}_{\{0<t_{l+1}\}}\int_{t_{l+1}}^{t_{l}}
\left|\frac{w_\ell}{\sqrt{M_-}}g(s,x_{l}-(t_{l}-s){v}_{l},v_{l})\right|d\Sigma _{l}(s)ds\bigg\},
\end{split}
\end{equation*}
\begin{equation*}
I_7=\frac{e^{-\frac{\nu_M (v)}{\eps}(t-t_{1})}}{\widetilde{w}_\ell}\int_{\prod_{j=1}^{k-1}%
\mathcal{V}_{j}}\mathbf{1}_{\{0<t_{k}\}}|h(t_{k},x_{k},v_{k-1})|d\Sigma
_{k-1}(t_{k}),\ \ k\geq2,
\end{equation*}
\begin{equation*}
I_8=\eps^{-1/2}\frac{e^{-\frac{\nu_M (v)}{\eps}(t-t_{1})}}{\widetilde{w}_\ell}\sum\limits_{l=1}^{k-1}\mathbf{1}_{0<t_{l+1}}d\Sigma_l^r,\ \ k\geq2,
\end{equation*}
and
\begin{eqnarray*}%\label{measure}
d\Sigma_{l}(s) &=&\left\{\Pi _{j=l+1}^{k-1}d\sigma_{j}\right\}\times\left\{e^{\frac{\nu_M(v_{l})(s-t_{l})}{\eps}}
\tilde{w}_\ell(v_{l})d\sigma_{l}\right\}\times \Pi _{j=1}^{l-1}
\left\{e^{\frac{\nu_M (v_{j})(t_{j+1}-t_{j})}{\eps}} d\sigma_{j}\right\},
\end{eqnarray*}
\begin{equation*}%\label{measure1}
\begin{split}
d\Sigma^r_{l} =&\{\Pi _{j=l+1}^{k-1}d\sigma_{j}\}\times \left\{e^{\frac{\nu_M(v_{l})(t_{l+1}-t_{l})}{\eps}}
\tilde{w}_\ell(v_{l})\frac{w_\ell }{\sqrt{M_-}}(v_l)|r(t_{l+1},x_{l+1},v_l)|d\sigma_{l}\right\}
\\&\times \Pi _{j=1}^{l-1}
\left\{e^{\frac{\nu_M (v_{j})(t_{j+1}-t_{j})}{\eps}} d\sigma_{j}\right\}.
\end{split}
\end{equation*}
We as follows estimate $I_n$ $(1\leq n\leq8)$ term by term. First of all, it is straightforward to see that
\begin{equation*}
I_2\leq C\eps\sup\limits_{0\leq s\leq t}\left\|\frac{w_{\ell-1}}{\sqrt{M_-}}g(s)\right\|_{\infty},\ \ I_3\leq C\sup_{0 \leq s \leq t}\eps^{-1/2}\left\|  w_\ell \frac{r(s)}{\sqrt{M_-}} \right\|_{\infty}.
\end{equation*}
Next, in light of Lemma \ref{k}, one has
\begin{equation*}
I_4\leq Ce^{-\frac{\nu_0t}{2\eps}}\left\|h_0\right\|_{\infty},
\end{equation*}
\begin{equation*}
\begin{split}
I_7\leq& C\int_{\prod_{j=1}^{k-2}%
\mathcal{V}_{j}}\mathbf{1}_{\{0<t_{k-1}\}}\Pi _{j=1}^{k-2} d\sigma _{j}
\sup_{0\leq s\leq t}\Vert h(s)\Vert _{\infty }
\leq
C\left\{ \frac{1}{2}\right\} ^{C_{2}T_0^{5/4}}\sup_{0\leq s\leq t}\Vert h(s)\Vert _{\infty},
\end{split}
\end{equation*}
and
\begin{equation*}
\begin{split}
I_6\leq C\eps\sup\limits_{0\leq s\leq t}\left\|\frac{w_{\ell-1}}{\sqrt{M_-}}g(s)\right\|_{\infty},\ \
I_8\leq C\sup_{0 \leq s \leq t}\eps^{-1/2}\left\|  w_\ell \frac{r(s)}{\sqrt{M_-}} \right\|_{\infty}.
\end{split}
\end{equation*}

We now conclude from \eqref{iteration} and the above estimates for $I_2$, $I_3$, $I_4$, $I_6$, $I_7$ and $I_8$ that
\begin{equation}\label{hmain}
\begin{split}
|h(t,x,v)|
\leq& \left\{\mathbf{1}_{t_{1}\leq 0}
\int_{0}^{t}+\mathbf{1}_{t_{1}>0}
\int_{t_1}^{t}\right\}e^{-\frac{\nu_M
(v)}{\eps}(t-s)}\frac{1}{\eps}|K_{M,w}h(s,x-(t-s)v,v)|ds \\
&+\frac{e^{-\frac{\nu_M (v)}{\eps}(t-t_{1})}}{\tilde{w}_\ell}\times
\int_{\prod_{j=1}^{k-1}\mathcal{V}_{j}}\sum_{l=1}^{k-1}\bigg\{%
\int_{0}^{t_{l}}\mathbf{1}_{\{t_{l+1}\leq 0<t_{l}\}}\frac{1}{\eps}|K_{M,w}h(s,X_{\mathbf{cl}}(s),v_{l})|   \\
&+\int_{t_{l+1}}^{t_{l}}\mathbf{1}_{\{0<t_{l+1}\}}\frac{1}{\eps}|K_{M,w}h(s,X_{\mathbf{cl}}(s),v_{l})|\bigg\}d\Sigma _{l}(s)ds+\CR(t)
\eqdef I_9+I_{10}+\CR(t),
\end{split}
\end{equation}
where %$A_{1 }(t)$ denotes
\begin{equation*}
\begin{split}
\CR(t)=&
C\eps\sup\limits_{0\leq s\leq t}\left\|\frac{w_{\ell-1}}{\sqrt{M_-}}g(s)\right\|_{\infty}+Ce^{-\frac{\nu_0t}{2\eps}}\Vert h(0)\Vert _{\infty}
\\&+C\left(\frac{1}{2}\right)^{C_2T_0^{5/4}}\sup_{0\leq s\leq t}\Vert h(s)\Vert _{\infty}
+CT_0^{5/4}\sup_{0 \leq s \leq t}\eps^{-1/2}\left\|  w_\ell \frac{r(s)}{\sqrt{M_-}} \right\|_{\infty}.
\end{split}
\end{equation*}
We compute next the delicate terms $I_9$ and $I_{10}$. By splitting the time integration into a ``small part" compared with $\eps$ and a ``remainder part" which ensures an invertible variable transformation, one has
\begin{equation*}
\begin{split}
I_9=\int_{t-\ka\eps}^{t}\cdots+\underbrace{\int_{\max\{t_1,0\}}^{t-\ka\eps}\cdots}_{I_{9,1}},
\ \ I_{10}=\int_{t_l-\ka\eps}^{t_l}\cdots+\underbrace{\int_{\max\{t_{l+1},0\}}^{t_l-\ka\eps}\cdots}_{I_{10,1}}.
\end{split}
\end{equation*}
It follows from direct calculations and Lemma \ref{k} that the ``small part" in $I_9$ and $I_{10}$ above can be controlled by
\begin{equation*}
\begin{split}
C\ka T_0^{5/4}\sup_{0\leq s\leq t}\Vert h(s)\Vert _{\infty}.
\end{split}
\end{equation*}
For the ``remainder part", we further decompose the velocity integration into a suitable ``bounded domain" and a ``unbounded domain". For the ``bounded domain", one can transform the $L^\infty$ norm into $L^2$ or $L^6$ norm which have been constructed in Lemma \ref{l2-l6} in Section \ref{l2-l6th}, while the ``unbounded domain" will be majorized by the decay property of the operator $K_{M,w}$, cf. Lemma \ref{K}.
To be more specific, we split $K_{M,w}=K_{M,w}-K_{w,m}+K_{w,m}$ with
\begin{equation*}%\label{k_m}
K_{w,m} (v,v') : = \mathbf{1}_{|v-v'|\geq \frac{1}{m}}
\mathbf{1}_{|v| \leq m} \mathbf{1}_{|v'| \leq m}
K_{M,w} (v,v'),
\end{equation*}
and
\begin{equation*}%\label{k_m}
\sup_{v} \int_{\mathbb{R}^{3}} | K_{w,m} (v,v')
-K_{M,w} (v,v') | d v' \leq \frac{1}{N(m)}, \ \ N(m)\gg 1.
\end{equation*}
The difference $K_{M,w}-K_{w,m}$ would lead to a small contribution in $I_{9,1}$ and $I_{10,1}$ as, for $N(m) \gg T_0^{5/4}$,
\begin{equation*}
\frac{k}{N(m)}   \sup_{0 \leq s \leq t}\|h (s) \|_{\infty} \leq \frac{ C T_{0}^{5/4}}{N(m)} \sup_{0 \leq s \leq t}\| h (s) \|_{\infty}.
\end{equation*}
Plugging the above small contributions into \eqref{hmain} and noticing that $K_{w,m}$ is bounded, we arrive at
\begin{equation}\label{hmain2}
\begin{split}
|h(t,x,v)|
\leq& C\int_{\max\{t_1,0\}}^{t-\ka\eps}\frac{1}{\eps}e^{-\frac{\nu_M(v)}{\eps}(t-s)}\int_{|v'|\leq m}|h(s,x-(t-s)v,v')|dv'ds \\
&+\frac{C}{\eps}\frac{e^{-\frac{\nu_M (v)}{\eps}(t-t_{1})}}{\tilde{w}_\ell}
\int_{\prod_{j=1}^{k-1}\mathcal{V}_{j}}\sum_{l=1}^{k-1}
\int_{\max\{t_{l+1},0\}}^{t_l-\ka\eps}\int_{|v'|\leq m}|h(s,X_{\mathbf{cl}}(s),v')|dv'   d\Sigma _{l}(s)ds\\&
+\CR(t)+o(1)T_{0}^{5/4}\sup_{0 \leq s \leq t}\| h (s) \|_{\infty}.
\end{split}
\end{equation}
Let us now denote $(t_{0}',x_{0}',v_{0}')=(s,X_{\mathbf{cl}}(s),v')$, for $v_{l'+1}'\in \mathcal {V}'_{l'+1}=\{v_{l'+1}'\cdot n(x_{l'+1}')>0\},$
and define a new back-time cycle as
$$
(t_{l'+1}',x_{l'+1}',v_{l'+1}')=(t_{l'}'-t_{\mathbf{b}}(x_{l'}',v_{l'}'),x_{\mathbf{b}}(x_{l'}',v_{l'}'),v_{l'+1}').
$$
We then
iterate \eqref{hmain2} to get a more elaborate  estimate as
\begin{equation*}%\label{diff3}
\begin{split}
|h(t,x,v)|
\leq& \frac{C}{\eps^2}\int_{\max\{t_1,0\}}^{t-\ka\eps}e^{-\frac{\nu_M(v)}{\eps}(t-s)}\iint\limits_{|v'|\leq m,|v''|\leq m} \int_{\max\{t_1',0\}}^{s-\ka\eps}e^{-\frac{\nu_M(v^{\prime })}{\eps}(s-s_{1})}\\&\qquad\qquad\times |h(s_{1},X_{\mathbf{cl}}(s)-(s-s_{1})v^{\prime },v^{\prime \prime
})|ds_{1}dv^{\prime }dv^{\prime \prime }ds\\
&+\frac{C}{\eps^2}\int_{\max\{t_1,0\}}^{t-\ka\eps}e^{-\frac{\nu_M(v)}{\eps}(s-s_1)}\iint\limits_{|v'|\leq m,|v''|\leq m}
\frac{e^{-\frac{\nu_M (v')}{\eps}(s-t'_{1})}}{\tilde{w}_\ell}
\int_{\prod_{j=1}^{k-1}\mathcal{V}_{j}}\sum_{l'=1}^{k-1}
\int_{\max\{t'_{l'+1},0\}}^{t'_{l'}-\ka\eps}\\&\qquad\qquad\times|h(s_{1,}x_{l^{\prime }}^{\prime
}+(s_{1}-t_{l^{\prime }}^{\prime })v_{l^{\prime }}^{\prime },v^{\prime
\prime })|dv''dv'
 d\Sigma_{l'}(s_1)ds_1ds \\
&+\frac{C}{\eps^2}\frac{e^{-\frac{\nu_M (v)}{\eps}(t-t_{1})}}{\tilde{w}_\ell}
\int_{\prod_{j=1}^{k-1}\mathcal{V}_{j}}\sum_{l=1}^{k-1}\iint\limits_{|v'|\leq m,|v''|\leq m}
\int_{\max\{t_{l+1},0\}}^{t_l-\ka\eps}\int_{\max\{t'_1,0\}}^{s-\ka\eps}e^{-\frac{\nu_M(v')}{\eps}(s-s_1)}
\\&\qquad\qquad\times|h(s_1,X_{\mathbf{cl}}(s_1)-(s-s_1)v',v'')|dv'dv''   d\Sigma _{l}(s)dsds_1
\\
&+\frac{C}{\eps^2}\frac{e^{-\frac{\nu_M (v)}{\eps}(t-t_{1})}}{\tilde{w}_\ell}
\int_{\prod_{j=1}^{k-1}\mathcal{V}_{j}}\sum_{l=1}^{k-1}\iint\limits_{|v'|\leq m,|v''|\leq m}
\int_{\max\{t_{l+1},0\}}^{t_l-\ka\eps}\frac{e^{-\frac{\nu_M(v')}{\eps}(s-t'_1)}}{\tilde{w}_\ell(v')}
\\&\quad\times\int_{\prod_{j=1}^{k-1}\mathcal{V}'_{j}}\sum_{l'=1}^{k-1}
\int_{\max\{t'_1,0\}}^{t'_{l'}-\ka\eps}
|h(s_1,x_{l^{\prime }}^{\prime }+(s_{1}-t_{l^{\prime}}^{\prime })v_{l^{\prime }}^{\prime },v'')|dv'dv''d\Sigma _{l'}(s_1)   d\Sigma _{l}(s)dsds_1\\
&+T_0^{5/4}\CR(t)+o(1)T_{0}^{5/4}\sup_{0 \leq s \leq t}\| h (s) \|_{\infty}\\ \eqdef&\sum\limits_{N=11}^{13}I_n
+T_0^{5/4}\CR(t)+o(1)T_{0}^{5/4}\sup_{0 \leq s \leq t}\| h (s) \|_{\infty}.
\end{split}
\end{equation*}
Notice that the Jacobian determinants
$$
\left|\frac{\pa(X_{\mathbf{cl}}(s)-(s-s_{1})v^{\prime })}{\pa v'}\right|,\ \ \left|\frac{\pa(x_{l^{\prime }}^{\prime }+(s_{1}-t_{l^{\prime}}^{\prime })v_{l^{\prime }}^{\prime })}{\pa v'}\right|\gtrsim \ka^3\eps^{3}.
$$
In light of Lemma \ref{k}, one has, by performing the similar calculations as \cite[Proposition 3.2, pp.489]{LY-2016} and \cite[Lemma 4.2, pp.204]{EGKM-13},
\begin{equation}\label{upl6-l3}
\sum\limits_{n=11}^{13}|I_n|\leq C\eps^{-3/2}  \sup_{0 \leq s \leq  t}\left\|\frac{ R(s)}{\sqrt{M_-}}\right\|_2,
\end{equation}
or
\begin{equation}\label{upl2}
\sum\limits_{n=11}^{13}|I_n|\leq C\eps^{-1/2}  \sup_{0 \leq s \leq  t}\left\|\frac{ P_0^MR(s)}{\sqrt{M_-}}\right\|_6
+\eps^{-3/2}  \sup_{0 \leq s \leq  t}\left\|\frac{ P_1^MR(s)}{\sqrt{M_-}}\right\|_2.
\end{equation}
Thus, \eqref{claim1} and \eqref{claim2} are valid.

Finally, using (\ref{claim1}) $n$ times gives
\begin{equation*}
\begin{split}
\|\eps h( n  T_{0})\|_{\infty}
\leq & CT_{0}^{5/2} e^{- \frac{\nu_{0} T_{0}}{2\eps}} \| \eps h((n-1) T_{0}) \|_{\infty} +  \sup_{(n-1)  T_{0} \leq s \leq n  T_{0} }  D(s)
\\
\leq& \Big[ CT_{0}^{5/2} e^{- \frac{\nu_{0} T_{0}}{2\eps}}\Big]^{2}\|\eps  h((n-2)  T_{0}) \|_{\infty} + \sum_{j=0}^{1}  \Big[ CT_{0}^{5/2}
e^{-\frac{\nu_0T_0}{\eps}}\Big]^{j}
 \sup_{(n-2) T_{0} \leq s \leq n  T_{0}}D(s)
\\
\vdots&  \\
\leq & \Big[ CT_{0}^{5/2} e^{- \frac{\nu_{0} T_{0}}{2\eps}}\Big]^{n}\| \eps h_{0}  \|_{\infty}  +  \sum_{j=0}^{n-1} \Big[ CT_{0}^{5/2} e^{- \frac{\nu_{0} T_{0}}{\eps}}\Big]^{j}
  \sup_{0 \leq s \leq  n  T_{0}}  D(s),
  \end{split}
\end{equation*}
where
\begin{equation*}
\begin{split}
D(s)=&
  CT_0^{5/2}\eps^{1/2}\left\|  w \frac{r (s)}{\sqrt{M_-}} \right\|_{\infty} +  C T_{0}^{5/2}\eps^{2}  \left\| \langle v\rangle^{-1}  w \frac{g(s)}{\sqrt{M_-}}\right\|_{\infty} + C T_{0}^{5/2}\eps
  \left\| \frac{P_0^MR}{\sqrt{M_-}} \right\|_{6}  \\
&+ C T_{0}^{5/2} \eps^{-1/2}
    \left\| \frac{P_1^MR}{\sqrt{M_-}} \right\|_{2}
 +  \left[  CT_{0}^{5/4} \Big\{\frac{1}{2}\Big\}^{C_{2} T_{0}^{5/4}} + o(1) CT_{0}^{5/4} \right]  \eps\|  h(s)\|_{\infty}.
 \end{split}
\end{equation*}
Notice that $\sum\limits_{j} \Big[ CT_{0}^{5/2} e^{- \frac{\nu_{0} T_{0}}{2\eps}}\Big]^{j} < \infty$.

Combining the above estimate with (\ref{claim1}), for $t \in [n  T_{0}, (n+1)  T_{0}]$, and absorbing the last term, one has
\begin{equation*}
\begin{split}
CT_{0}^{5/2}& e^{- \frac{\nu_{0} ( t- n  T_{0})}{2\eps}} \sum_{j=0}^{n-1} \Big[ CT_{0}^{5/2} e^{- \frac{\nu_{0} T_{0}}{2\eps}}\Big]^{j}\left[CT_{0}^{5/4}\Big\{\frac{1}{2} \Big\}^{C_{2} T_{0}^{5/4}}   + o(1)CT_{0}^{5/4}    \right]   \sup_{0 \leq s \leq  t}
   \eps\| h (s)\|_{\infty}\\
   \lesssim&   \frac{ T_{0}^{5/2}}{1- CT_{0}^{5/2} e^{- \frac{\nu_{0} T_{0}}{2\eps}}} \left[CT_{0}^{5/4}\Big\{\frac{1}{2} \Big\}^{C_{2} T_{0}^{5/4}}   + o(1)CT_{0}^{5/2}    \right] \times  \sup_{0 \leq s \leq  t}
   \eps\| h (s)\|_{\infty} \\ \lesssim& \  o(1) \sup_{0 \leq s \leq  t}
  \eps \| h (s)\|_{\infty},
\end{split}
\end{equation*}
where we used $$ T_{0}^{5/2}\left[CT_{0}^{5/4}\Big\{\frac{1}{2} \Big\}^{C_{2} T_{0}^{5/4}}   + o(1)CT_{0}^{5/2}    \right] \ll 1. $$
Therefore \eqref{point1} is valid. \eqref{point2} and \eqref{point3} can be proved similarly, we skip the details here for brevity.
\end{proof}

\section{$L^3$ estimates}

In this section, we deduce a crucial $L^3$ estimate of the macroscopic part $P_0^MR$, such an estimate plays a significant role in
controlling the nonlinear operator $Q(P_0^MR,P_0^MR)$.
\begin{lemma}\label{l3lem}
Let $\frac{\widetilde{R}}{\sqrt{M_-}}\in L_{t,x,v}^2(\R\times\R^3\times\R^3)$ be a function satisfying the transport equation
\be  \pa_t \widetilde{R}+v\cdot \nabla_x \widetilde{R}=\widetilde{g},\notag\ee
in the sense of distributions, with $\frac{\widetilde{g}}{\sqrt{M_-}}\in L_{t,x,v}^2(\R\times\R^3\times\R^3)$. Let $\psi_e$ be a test function that decays very quickly as $|v|\rightarrow\infty$. Then it holds that
\be\notag
\left\| \left\langle \frac{\widetilde{R}}{\sqrt{M_-}}, \psi_e\right\rangle\right\|_{L^2_tL^3_x}\lesssim \left\|\langle v\rangle^{-1/2}\frac{\widetilde{g}}{\sqrt{M_-}}\right\|_{L^2_{t,x,v}}.
\ee
\end{lemma}
With Lemma \ref{l3lem} in hand, one can now show by using the technique developed in \cite{EGKM-15} the following
critical estimates.
\begin{lemma}\label{l3lem2}
Assume $[\rho,u,\ta]\in\FX_\eps$, $\frac{\sqrt{\eps}g}{\sqrt{M_-}} \in L^{2} (\mathbb{R}_{+} \times \Omega \times \mathbb{R}^{3})$, $\frac{\sqrt{\eps}R_{0}}{\sqrt{M_-}} \in L^{2} (\Omega \times \mathbb{R}^{3})$, and $\sqrt{\eps}R_{\gamma}\in L^{2} (  \mathbb{R}_{+}\times \gamma)$. Let $\sqrt{\eps}R \in L^{\infty}(  \mathbb{R}_{+}; L^{2} (\Omega \times \mathbb{R}^{3}))$ solve
\begin{eqnarray*}
\left\{\begin{split}
&\pa_tR+v\cdot\na_xR=\frac{1}{\eps}L_MR+g,\\
&R(0,x,v)= R_0(x,v),\\
&R_{-}=P_\ga R+\eps^{-1/2}r,
\end{split}\right.
\end{eqnarray*}
in the sense of distribution.
Then, for $\al_0\leq1$, it holds that
\begin{equation}\label{abcp1}
\begin{split}
|\pa_t^{\al_0}[a,b,c]|\lesssim~ S_1^{\al_0}(R)+S_2^{\al_0}(R)+S_3^{\al_0}(R),
 \end{split}
 \end{equation}
with
\begin{equation}\label{abcp2}
\begin{split}
S_1^{\al_0}(R)=&\sum\limits_{i=1}^5\left|\int_{\mathbb{R}^{3}} \frac{\pa^{\al_0}_tR_{\delta}}{\sqrt{M_-}} (t,x,v) \psi_{e,i} (v) d v\right|,\\
S_2^{\al_0}(R)=&\sum\limits_{i=1}^5\int_{\mathbb{R}^{3}} | P_1^M\pa_t^{\al_0}R (t,x,v) |  \frac{|\psi_{e,i}|}{\sqrt{M_-}}(v) d v,\\
S_3^{\al_0}(R)=&\sum\limits_{i=1}^5\int_{\mathbb{R}^{3}}  |\pa_t^{\al_0}R_{0} (x,v)| \frac{|\psi_{e,i}|}{\sqrt{M_-}}(v) d v,
 \end{split}
 \end{equation}
here $[a,b,c](t,x)$, $R_\de$ and $\psi_{e,i}$ $(1\leq i\leq5)$ are given by \eqref{P0R}, \eqref{Rde} and \eqref{pei}, respectively.
Moreover,
\begin{equation}\label{l3es2}
\begin{split}
\sqrt{\eps}\sum\limits_{\al_0\leq1}&\left\|S_1^{\al_0}(R)\right\|_{L^2_tL^3_x}\\
\lesssim&
\sqrt{\eps}\sum\limits_{\al_0\leq1}\left\|\frac{\pa^{\al_0}_tR_0}{\sqrt{M_-}}\right\|_{L^2_{x,v}}
+\sqrt{\eps}\sum\limits_{\al_0\leq1}\left\|\frac{v\cdot\na_x\pa^{\al_0}_tR_0}{\sqrt{M_-}}\right\|_{L^2_{x,v}}
 +\sqrt{\eps}\sum\limits_{\al_0\leq1}\left\|\langle v\rangle^{-1/2}\frac{\pa^{\al_0}_tg}{\sqrt{M_-}}\right\|_{L^2_{t,x,v}}
\\& +\frac{1}{\sqrt{\eps}}\sum\limits_{\al_0\leq1}\left\|\frac{P_1^M\pa_t^{\al_0}R}{\sqrt{M_-}}\right\|_{L^2_{t,x,v}}
 +\ka_0\sqrt{\eps}\left\|\frac{R}{\sqrt{M_-}}\right\|_{L^2_{t,x,v}}
 +\sqrt{\eps}\sum\limits_{\al_0\leq1}\left\|\frac{\pa_t^{\al_0}R}{\sqrt{M_-}}\right\|_{L^2(\R_+\times\ga_+)}
 \\&+\sum\limits_{\al_0\leq1}\left|\frac{\pa_t^{\al_0}r}{\sqrt{M_-}}\right|_{2}
 +\sqrt{\eps}\sum\limits_{\al_0\leq1}\left|\frac{\pa^{\al_0}_tR_0}{\sqrt{M_-}}\right|_{L^2_\ga},
\end{split}
\end{equation}
\begin{equation}\label{l3es3}
\begin{split}
\sum\limits_{\al_0\leq1}\left\|S_2^{\al_0}\right\|_{L^2_tL^2_x}
\lesssim\sum\limits_{\al_0\leq1}\left\|\frac{P_1^M\pa_t^{\al_0}R}{\sqrt{M_-}}\right\|_{L^2_{t,x,v}}\ \text{and}\ \sum\limits_{\al_0\leq1}\left\|S_3^{\al_0}\right\|_{L^\infty_x}
\lesssim\sum\limits_{\al_0\leq1}\left\|\frac{\pa_t^{\al_0}R_0}{\sqrt{M_-}}\right\|_{\infty}.
\end{split}
\end{equation}

\end{lemma}
\begin{proof}
Define
$$
\xi(x)=\left\{\begin{array}{rll}
-\text{dist}(x,\pa\Om),&\ x\in\Om,\\[2mm]
\text{dist}(x,\pa\Om),&\ x\notin\Om,
\end{array}\right.
$$
and
$$
\chi\in C_0^\infty\ \text{such that}\ 0\leq \chi\leq1\ \text{and}\ \chi(x)=\left\{\begin{array}{rll}
1,&\ |x|\leq1/2,\\[2mm]
0,&\ |x|\geq1.
\end{array}\right.
$$
Then, for $(t,x,v) \in \mathbb{R} \times  \bar{\Omega} \times \mathbb{R}^{3}$ and $0 < \delta \ll 1$, set
\begin{equation}\label{Rde}
\begin{split}
R_{\delta}(t,x,v) = \left[1-\chi\left(\frac{n(x) \cdot v}{\delta}\right) \chi \left(\frac{
 \xi(x)}{\delta}\right)\right ]  \chi(\delta|v|) \big\{ \mathbf{1}_{ t\in[0, \infty)}  R(t,x,v)+ \mathbf{1}_{t\in(-\infty,0 ]} \chi(t) R_{0}(x,v) \big\}.
 \end{split}
\end{equation}
We now {\it claim} that
there exists $\overline{R}(t,x,v) \in L^{2}( \mathbb{R} \times   \mathbb{R}^{3} \times \mathbb{R}^{3})$, an extension of
$R_{\delta}$ in \eqref{Rde}, such that
\begin{equation}\label{Rex}
  \overline{R}|_{\Omega \times \mathbb{R}^{3}}= R_{\delta}    \  \text{ and } \  \overline{R}  |_{\gamma}= R_{ \delta} |_{\gamma}   \  \text{ and } \ \overline{R} |_{t=0} = R_{\delta} |_{t=0}.
\end{equation} In addition, in the sense of distributions on $\mathbb{R} \times \mathbb{R}^{3} \times \mathbb{R}^{3}$,
\begin{equation}\label{Rdeex}
\partial_{t} \overline{R}+  v\cdot \nabla_{x} \overline{R} = \bar{h} \eqdef \bar{h}_{1} + \bar{h}_{2} + \bar{h}_{3}+ \bar{h}_{4},
\end{equation}
with
\begin{eqnarray}
\bar{h}_{1} (t,x,v)&=& \mathbf{1}_{(x,v) \in \Omega \times \mathbb{R}^{3}}  \left [1-\chi\left(\frac{n(x) \cdot v}{\delta}\right)  \chi \left( \frac{\xi(x)}{\delta}\right)\right] \chi(\delta|v|)  \notag\\
&& \times
\left[ \mathbf{1}_{t \in [0,\infty)}
\left(g(t,x,v) +\eps^{-1} L_M\overline{R}\right)  + \mathbf{1}_{t \in ( - \infty, 0 ]} \chi(t)\left\{ \frac{ \chi^{\prime} (t)}{\chi(t)} + v \cdot \nabla_{x}\right\} R_{0} (x,v)\right],\notag\\
\bar{h}_{2} (t,x,v) &=&\mathbf{1}_{(x,v) \in \Omega \times \mathbb{R}^{3}}  \left[\mathbf{1}_{t \in [0,\infty)} R(t,x,v) + \mathbf{1}_{t \in (- \infty, 0 ]} \chi(t) R_{0} (x,v)
\right] \notag\\
&& \times
 \{v\cdot \nabla_{x}\} \left(  \left[1-\chi\left(\frac{n(x) \cdot v}{\delta}\right)  \chi \left( \frac{\xi(x)}{\delta}\right) \right] \chi(\delta|v|)\right),\notag\\
\bar{h}_{3} (t,x,v) &=&  \mathbf{1}_{{ (x,v) \in    [\Omega_{\tilde{C} \delta^{4}} \backslash \bar{\Omega}]\times \mathbb{R}^{3}}}   \ \frac{1}{\tilde{C} \delta^{4}}v \cdot \nabla_{x} \xi(x) \chi^{\prime}  \left( \frac{\xi(x)}{\tilde{C} \delta^{4}} \right)\notag \\
 &&     \times
 \bigg[ R_{\delta}(t-  t_{\bf b}^{*} (x,v), x_{\bf b}^{*}(x,v), v) \mathbf{1}_{x_{\bf b}^{*}(x,v) \in \partial\Omega} \notag\\
 && \ \  \   +
   R_{\delta}(t + t_{\bf f}^{*} (x,v),  x_{\bf f}^{*}(x,v), v) \mathbf{1}_{x_{\bf f}^{*}(x,v) \in \partial\Omega}
  \bigg] ,\notag\\
  \bar{h}_{4} (t,x,v) &=&  \mathbf{1}_{{ (x,v) \in    [\Omega_{\tilde{C} \delta^{4}} \backslash \bar{\Omega}]\times
  \mathbb{R}^{3}}} R_{\delta}(t-  t_{\bf b}^{*} (x,v), x_{\bf b}^{*}(x,v), v) \chi \left( \frac{\xi(x)}{\tilde{C} \delta^{4}} \right) \chi^{\prime}(x_{\bf b}^{*}(x,v)) \mathbf{1}_{x_{\bf b}^{*}(x,v) \in \partial\Omega}\notag\\
  &&+ \mathbf{1}_{{ (x,v) \in    [\Omega_{\tilde{C} \delta^{4}} \backslash \bar{\Omega}]
  \times \mathbb{R}^{3}}} R_{\delta}( t + t_{\bf f}^{*} (x,v),  x_{\bf f}^{*}(x,v), v) \chi \left( \frac{\xi(x)}{\tilde{C} \delta^{4}} \right) \chi^{\prime}(x_{\bf f}^{*}(x,v)) \mathbf{1}_{ x_{\bf f}^{*}(x,v) \in \partial\Omega},\notag\\&&\notag
\end{eqnarray}
here
$
\Omega_{\tilde{C}\delta^{4}}=  \ \big\{ x \in \mathbb{R}^{3}:  \  \xi(x) < \tilde{C}\delta ^{4},\ \tilde{C}>0\big\},
$
and
for $(x,v) \in\Omega_{\tilde{C}\delta^{4}}\backslash \overline{\Omega}$, with $\overline{\Omega}=\Omega\cup\pa\Om$,
\begin{eqnarray}
&&t_{\bf b}^{*}(x,v)= \  \inf\{  s>0: 0 <\xi(X(\tau;0,x,v)) <  \tilde{C} \delta^{4}\ \ \text{for all } 0 <\tau< s \}, \ \
t_{\bf f}^{*}(x,v) = \ t_{\bf b}^{*}(x,-v) , \notag\\
&&(x_{\bf b}^{*} (x,v),v ) = (X(- t_{\bf b}^{*}(x,v);0,x,v), v), \notag  \\
&&(x_{\bf f}^{*} (x,v),v)=  (X(  t_{\bf f}^{*}(x,v);0,x,v), v). \notag
\end{eqnarray}
Moreover, it holds that
\begin{eqnarray}\label{h1-4}
\left\{\begin{array}{rlll}
\begin{split}
&\left\| \frac{\langle v\rangle^{-1/2}\bar{h}_{1}}{\sqrt{M_-}} \right\|_{  L^{2}( \mathbb{R} \times   \mathbb{R}^{3} \times \mathbb{R}^{3})}   \lesssim
\left\|\frac{\langle v\rangle^{-1/2}g}{\sqrt{M_-}}\right\|_{L^{2}(
\mathbb{R}_{+} \times
\Omega \times \mathbb{R}^{3})}+\frac{1}{\eps}\left\| \frac{\langle v\rangle^{-1/2}L_M\overline{R}}{\sqrt{M_-}}\right\|_{L^{2}(
\mathbb{R}_{+} \times
\Omega \times \mathbb{R}^{3})}\\&\qquad\qquad\qquad\qquad\qquad\qquad+ \left\| \frac{\overline{R}_{0}}{\sqrt{M_-}} \right\|_{L^{2} (\Omega \times \mathbb{R}^{3})}
 + \left\| \frac{v\cdot \nabla_{x} \overline{R}_{0}}{\sqrt{M_-}} \right\|_{L^{2 } (\Omega \times \mathbb{R}^{3})}
  ,\\
&\left\| \frac{\bar{h}_{2}}{\sqrt{M_-}} \right\|_{L^{2} ( \mathbb{R} \times \mathbb{R}^{3} \times \mathbb{R}^{3})} \lesssim      \left\|\frac{\overline{R}}{\sqrt{M_-}}\right\|_{L^{2}(
\mathbb{R}_{+} \times
 \Omega \times \mathbb{R}^{3})}  + \left\| \frac{\overline{R}_{0}}{\sqrt{M_-}} \right\|_{L^{2} (\Omega \times \mathbb{R}^{3})},  \\
&\left\| \frac{\bar{h}_{3}}{\sqrt{M_-}} \right\|_{L^{2} (  \mathbb{R} \times \mathbb{R}^{3} \times \mathbb{R}^{3})}
+\left\| \frac{\bar{h}_{4}}{\sqrt{M_-}} \right\|_{L^{2} (  \mathbb{R} \times\mathbb{R}^{3} \times \mathbb{R}^{3})}
  \lesssim \left\| \frac{\overline{R}_{\gamma}}{\sqrt{M_-}} \right\|_{L^{2} ( \mathbb{R}_{+} \times \gamma)}
  + \left\| \frac{\overline{R}_{0}}{\sqrt{M_-}} \right\|_{L^{2} (\gamma)}.
\end{split}\end{array}\right.
\end{eqnarray}
The proof for the above {\it claim} is similar and much easier than that of \cite[Lemma 3.6, pp.40]{EGKM-15}, the details will be omitted for brevity.

We next define
\begin{equation}\label{pei}
[\psi_{e,1},\psi_{e,2},\psi_{e,3},\psi_{e,4},\psi_{e,5}]
=\left[\sqrt{M_-},v_1\sqrt{M_-},v_2\sqrt{M_-},v_3\sqrt{M_-},\frac{|v|^2-3}{2}\sqrt{M_-}\right],
\end{equation}
then one has thanks to \eqref{P0R}
\begin{equation*}
\begin{split}
&\int_{\mathbb{R}^{3}} \frac{R_{\delta}}{\sqrt{M_-}} (t,x,v) \psi_{e,i} (v) d v\\
=& \int_{\mathbb{R}^{3}}  \left[ 1- \chi\left( \frac{n(x) \cdot v}{\delta} \right) \chi\left( \frac{\xi(x)}{\delta}\right)\right]
\chi( \delta |v|)\\& \times
\left\{ \mathbf{1}_{t \geq 0} \frac{R_{\delta}}{\sqrt{M_-}}(t,x,v) + \mathbf{1}_{t \leq 0} \chi(t) \frac{R_{0}}{\sqrt{M_-}} (x,v)  \right\}
 \psi_{e,i} (v) d v \\
 =& \mathbf{1}_{t \geq 0}  \int_{\mathbb{R}^{3}}  \left[ 1- \chi\left( \frac{n(x) \cdot v}{\delta} \right) \chi\left( \frac{\xi(x)}{\delta}\right)\right] \chi( \delta |v|)\left\{ P_0^MR + P_1^MR\right\}(t,x,v) \frac{\psi_{e,i}}{\sqrt{M_-}}(v) d v\\
 &+ \mathbf{1}_{t \leq 0}   \int_{\mathbb{R}^{3}}  \left[ 1- \chi\left( \frac{n(x) \cdot v}{\delta} \right) \chi\left( \frac{\xi(x)}{\delta}\right)\right] \chi( \delta |v|)
  \chi(t) \frac{R_{0}}{\sqrt{M_-}}(x,v)\psi_{e,i}(v) d v\\
  =& \mathbf{1}_{t \geq 0}  \int_{\mathbb{R}^{3}}  \left[ 1- \chi\left( \frac{n(x) \cdot v}{\delta} \right) \chi\left( \frac{\xi(x)}{\delta}\right)\right] \chi( \delta |v|)\\&\times
  \left\{ \frac{a(t,x)}{\rho}+\frac{b(t,x)\cdot (v-u)}{\rho\ta}+\frac{c(t,x)}{\rho\ta}\left(\frac{|v-u|^2}
{\ta}-3\right)\right\}M(t,x,v) \frac{\psi_{e,i}}{\sqrt{M_-}}(v) d v\\
  &+\mathbf{1}_{t \geq 0}  \int_{\mathbb{R}^{3}}  \left[ 1- \chi\left( \frac{n(x) \cdot v}{\delta} \right) \chi\left( \frac{\xi(x)}{\delta}\right)\right] \chi( \delta |v|)P_1^MR(t,x,v) \frac{\psi_{e,i}}{\sqrt{M_-}}(v) d v\\
 &+ \mathbf{1}_{t \leq 0}   \int_{\mathbb{R}^{3}}  \left[ 1- \chi\left( \frac{n(x) \cdot v}{\delta} \right) \chi\left( \frac{\xi(x)}{\delta}\right)\right] \chi( \delta |v|)
  \chi(t) \frac{R_{0}}{\sqrt{M_-}}(x,v)\psi_{e,i}(v) d v\\
\geq& \mathbf{1}_{t\geq 0} \left\{a_{i} (t,x)  - C(\delta+\ka_0\eps) \sum_{i=1}^{5} |a_{i} (t,x)| - C_{\delta} \int_{\mathbb{R}^{3}} | P_1^MR (t,x,v) |  \frac{|\psi_{e,i}|}{\sqrt{M_-}}(v) d v\right\}\\
 &- \mathbf{1} _{ t \leq 0} \chi(t) \int_{\mathbb{R}^{3}}  |R_{0} (x,v)| \frac{|\psi_{e,i}|}{\sqrt{M_-}}(v) d v,\ 1\leq i\leq5.
 \end{split}
 \end{equation*}
Here, $[a_1,a_2,a_3,a_4,a_5]=[a,b_1,b_2,b_3,c],$ and the fact that
$$
\int_{|v|\geq 1/\de\ \text{or}\ |n\cdot v|\leq\de/2\ \text{and}\ |\xi(x)|\leq\de/2}| P_0^MR(t,x,v)| \frac{|\psi_{e,i}|}{\sqrt{M_-}}(v) d v
\lesssim O(\delta),\ \ \|[u,\ta]\|_{\infty}\leq \ka_0\eps
$$
was used.

As a consequence, we have
\begin{equation}\label{abcav}
\begin{split}
|a(t,x)|+|b(t,x)|+|c(t,x)|\lesssim& \sum_{i=1}^{5}\left|\int_{\mathbb{R}^{3}} \frac{R_{\delta}}{\sqrt{M_-}} (t,x,v) \psi_{e,i} (v) d v\right|
+\sum_{i=1}^{5}\int_{\mathbb{R}^{3}} | P_1^MR (t,x,v) |  \frac{|\psi_{e,i}|}{\sqrt{M_-}}(v) d v
\\&+\sum_{i=1}^{5}\int_{\mathbb{R}^{3}}  |R_{0} (x,v)| \frac{|\psi_{e,i}|}{\sqrt{M_-}}(v) d v \eqdef S_1^{0}(R)+S_2^{0}(R)+S_3^{0}(R).
 \end{split}
 \end{equation}
We now pay our attention to $S_1^{0}(R)$. Clearly, the function $\bar R$ defined by \eqref{Rex} satisfies
 \[
  \left|\int_{\mathbb{R}^{3}} \frac{R_{\delta} (t,x,v)\psi_{e,i}}{\sqrt{M_-}}  d v\right| \lesssim \left|\int_{\mathbb{R}^{3}} \frac{\bar{R} (t,x,v) \psi_{e,i}}{\sqrt{M_-}} d v\right|,\ x\in\Om,
 \]
for $1\leq i\leq5.$ Then applying Lemma \ref{l3lem} to \eqref{Rdeex}, one has
\begin{equation}\label{Rl3f}
\begin{split}
\sqrt{\eps}\left\|\int_{\mathbb{R}^{3}} \frac{\bar{R} (t,x,v) \psi_{e,i}}{\sqrt{M_-}} d v\right\|_{L^2_tL_x^3}
\lesssim \sqrt{\eps}\sum\limits_{i=1}^4\left\|\frac{\langle v\rangle^{-1/2}\bar{h}_i}{\sqrt{M_-}}\right\|_{L^2_tL_x^2}
\end{split}
\end{equation}
Therefore \eqref{abcp1}, \eqref{abcp2}, \eqref{l3es2} and \eqref{l3es3} with $\al_0=0$ follows from \eqref{h1-4}, \eqref{abcav} and \eqref{Rl3f}, the case that $\al_0=1$ can be proved similarly. This ends the proof of Lemma \ref{l3lem2}.
\end{proof}

\section{Global-in-time existence}
In this section, we will deduce the global existence for the system \eqref{Rib} with the aid of the results obtained in previous sections.
That is we are going to complete
\begin{proof}[The proof of Theorem 1]
We design the following iteration sequence:
\begin{eqnarray}\label{itform}
\left\{\begin{split}
&\pa_tR^{n+1}+v\cdot\na_xR^{n+1}+\frac{1}{\eps}L_MR^{n+1}=
\eps^{1/2}Q(R^{n},R^{n})+Q(R^{n},G)+Q(G,R^{n})\\
&\qquad\qquad\qquad\qquad\qquad\qquad+\eps^{-1/2}Q(G,G)-\eps^{-1/2}(\pa_tG+v\cdot \na_xG+H),\\
&R^{n+1}(0,x,v)= R_0(x,v)=-\eps^{-1/2} G(0,x,v),\\
&R^{n+1}_{-}=P_\ga R^{n+1}+\eps^{-1/2}r,
\end{split}\right.
\end{eqnarray}
with $R^0=R_0(x,v).$ We emphasize that the iteration scheme \eqref{itform} coincides with the linearized equation \eqref{Rln} so that Lemmas \ref{l2es}, \ref{l2-l6} and \ref{l3lem2} and  Proposition \ref{point_dyn} can be directly used. It is convenient to set the following energy functional
\begin{equation*}
\begin{split}
\CE(R)(t)=&\eps^2\sup\limits_{0\leq s\leq t}\left\|w_{\ell}\frac{R(s)}{\sqrt{M_-}}\right\|^2_{\infty}+\eps^3\sup\limits_{0\leq s\leq t}\left\|w_{\ell}\frac{\pa_tR(s)}{\sqrt{M_-}}\right\|^2_{\infty}
\\&+\eps\sup\limits_{0\leq s\leq t}\left\|\frac{P_0^MR(s)}{\sqrt{M_-}}\right\|_6^2
+\sup\limits_{0\leq s\leq t}\sum\limits_{\al_0\leq1}\left\|\frac{\pa_t^{\al_0}R(s)}{\sqrt{M_-}}\right\|_2^2,
\end{split}
\end{equation*}
and the dissipation
\begin{equation*}
\begin{split}
\mathcal {D}(R)(t)=&\eps\sum\limits_{\al_0\leq1}\left\|\frac{P_0^M\pa_t^{\al_0}R}{\sqrt{M_-}}\right\|_2^3
+\eps\sum\limits_{\al_0\leq1}\left\|S_1^{\al_0}(R)\right\|_3^2
+\sum\limits_{\al_0\leq1}\left|\frac{(I-P_\ga)\pa_t^{\al_0}R}{\sqrt{M_-}}\right|_{2}^2
+\eps\sum\limits_{\al_0\leq1}\left|\frac{P_\ga\pa_t^{\al_0}R}{\sqrt{M_-}}\right|_{2}^2
\\&+\frac{1}{\eps}\sum\limits_{\al_0\leq1}\left\|\frac{\sqrt{\nu_M}P_1^M\pa_t^{\al_0}R}{\sqrt{M_-}}\right\|_2^2.
\end{split}
\end{equation*}
For later use, we also define a Banach space
$$
\FX_{\tilde{\de}}=\left\{R~|~\CE(R)(t)+\int_0^t\mathcal {D}(R)(s)ds<\tilde{\de},\ \ \tilde{\de}>0\right\},
$$
associated with the norm
$$
\FX_{\tilde{\de}}(f)(t)=\sup\limits_{0\leq s\leq t}\CE(f)(s)+\int_0^t\mathcal {D}(f)(s)ds.
$$

In light of Lemmas \ref{l2es} and \ref{l2-l6}, we have
\begin{equation}\label{eng1}
\begin{split}
\sup\limits_{0\leq s\leq t}&\sum\limits_{\al_0\leq1}\int_{\Om\times \R^3}\frac{(\pa_t^{\al_0}R^{n+1}(s))^2}{M_-}dxdv+\sup\limits_{0\leq s\leq t}\eps \left\|\left[a^{n+1},b^{n+1},c^{n+1}\right](s)\right\|^2_{6}\\&+\eps\int_0^t\left\|\left[a^{n+1},b^{n+1},c^{n+1}\right](s)\right\|_2^2ds
+\frac{1}{\eps}\sum\limits_{\al_0\leq1}\int_0^t\int_{\Om\times \R^3}\frac{\nu_M(P_1^M\pa_t^{\al_0}R^{n+1}(s))^2}{M_-}dxdvds\\&
+\sum\limits_{\al_0\leq1}\int_0^t\int_{\ga_+}\frac{|(I-P_\ga)\pa_t^{\al_0}R^{n+1}|^2}{M_-}d\ga ds\\
\lesssim& \sum\limits_{\al_0\leq1}\int_{\Om\times \R^3}\frac{(\pa_t^{\al_0}R_0)^2}{M_-}dxdv
+\int_{\ga_+}\frac{|(I-P_\ga)R_0|^2}{M_-}d\ga+\eps^{-1}\int_{\Om\times \R^3}\frac{(P_1^MR_0)^2}{M_-}dxdv\\&+\eps\sum\limits_{\al_0\leq1}\int_0^t\int_{\ga_-} \frac{|\pa_t^{\al_0}r|^2}{M_-} d\ga ds+\eps\sum\limits_{\al_0\leq1}\int_0^t\int_{\Om\times \R^3}\frac{\nu_M^{-1}|\pa_t^{\al_0}\CG|^2}{M_-}dxdvds
\\&+\eps\sup\limits_{0\leq s\leq t}\int_{\Om\times \R^3}\frac{\nu_M^{-1}|\CG(s)|^2}{M_-}dxdv+\sup\limits_{0\leq s\leq t}\|\na_x[u,\ta]\|_{H^1}^2,
\end{split}
\end{equation}
where $\CG=\CG_1+\CG_2$ with
$$
\CG_1=\eps^{1/2}Q(R^{n},R^{n}),\ \ \CG_2=Q(R^{n},G)+Q(G,R^{n})+\eps^{-1/2}Q(G,G)-\eps^{-1/2}(\pa_tG+v\cdot \na_xG+H),
$$
and $[a^{n+1},b^{n+1},c^{n+1}]$ denote the corresponding macroscopic quantities of $R^{n+1}$ defined as \eqref{Rdec}.

Moreover, thanks to Proposition \ref{point_dyn}, one has
\begin{equation}\label{point4}
\begin{split}
 \eps^2\left\|w_\ell \frac{R^{n+1}(t)}{\sqrt{M_-}}\right\|^2_{\infty}
  \lesssim&
   \eps^2\left\|  w_\ell \frac{R_{0}}{\sqrt{M_-}} \right\|^2_{\infty} + \eps\sup_{0 \leq s \leq t}\left\|  w_\ell \frac{r(s)}{\sqrt{M_-}} \right\|^2_{\infty} + \eps^{4} \sup_{0 \leq s \leq t} \left\|  w_{\ell-1} \frac{\CG(s)}{\sqrt{M_-}}\right\|^2_{\infty}\\
  &+  \sup_{0\leq s \leq  t}\eps\left\|[a^{n+1},b^{n+1},c^{n+1}]\right\|^2_{6} %+\ka_0^2\eps\int_0^t\|[a^{n+1},b^{n+1},c^{n+1}]\|_2^2ds
 % \\&+
 + \eps^{-1}  \sup_{0 \leq s \leq  t}\left\|\frac{P_1^M R^{n+1}(s)}{\sqrt{M_-}}\right\|^2_2,
    \end{split}
\end{equation}
and
\begin{equation}\label{point5}
\begin{split}
\eps^3\left\|w_\ell \frac{\pa_tR^{n+1}(t)}{\sqrt{M_-}}\right\|^2_{\infty}
\lesssim&
\eps^3\left\|  w_\ell \frac{\pa_tR_{0}}{\sqrt{M_-}} \right\|_{\infty} + \eps^2 \sup_{0 \leq s \leq t}\left\|  w_\ell \frac{\pa_tr(s)}{\sqrt{M_-}} \right\|^2_{\infty} + \eps^{5} \sup_{0 \leq s \leq t} \left\|  w_{\ell-1} \frac{\pa_t\CG(s)}{\sqrt{M_-}}\right\|^2_{\infty}\\
&+\sup_{0 \leq s \leq  t}\left\|\frac{ \pa_tR^{n+1}(s)}{\sqrt{M_-}}\right\|^2_2
+\ka^2_0\eps^3\sup_{0 \leq s \leq t}\left\|  w_\ell \frac{R^{n+1}(s)}{\sqrt{M_-}} \right\|^2_{\infty}.
\end{split}
\end{equation}
Therefore, a suitable linear combination of \eqref{eng1}, \eqref{point4} and \eqref{point5} yields
\begin{equation}\label{eng2}
\begin{split}
\sup\limits_{0\leq s\leq t}&\int_{\Om\times \R^3}\frac{(\pa_t^{\al_0}R^{n+1}(s))^2}{M_-}dxdv+\sup\limits_{0\leq s\leq t}\eps \left\|[a^{n+1},b^{n+1},c^{n+1}](s)\right\|^2_{6}+\eps\int_0^t\left\|[a^{n+1},b^{n+1},c^{n+1}](s)\right\|_2^2ds\\
&+\eps^2\sup\limits_{0\leq s\leq t}\left\|w_\ell \frac{R^{n+1}(s)}{\sqrt{M_-}}\right\|^2_{\infty}+\eps^3\sup\limits_{0\leq s\leq t}\left\|w_\ell \frac{\pa_tR^{n+1}(s)}{\sqrt{M_-}}\right\|^2_{\infty}
\\
&+\frac{1}{\eps}\sum\limits_{\al_0\leq1}\int_0^t\int_{\Om\times \R^3}\frac{\nu_M(P_1^M\pa_t^{\al_0}R^{n+1}(t))^2}{M_-}dxdvds
+\sum\limits_{\al_0\leq1}\int_0^t\int_{\ga_+}\frac{|(I-P_\ga)\pa_t^{\al_0}R^{n+1}|^2}{M_-}d\ga ds\\
\lesssim& \sum\limits_{\al_0\leq1}\int_{\Om\times \R^3}\frac{(\pa_t^{\al_0}R_0)^2}{M_-}dxdv
+\int_{\ga_+}\frac{|(I-P_\ga)R_0|^2}{M_-}d\ga+\eps^{-1}\int_{\Om\times \R^3}\frac{(P_1^MR_0)^2}{M_-}dxdv
\\&+\eps^2\left\|  w_\ell \frac{R_{0}}{\sqrt{M_-}} \right\|^2_{\infty}+\eps^3\left\|  w_\ell \frac{\pa_tR_{0}}{\sqrt{M_-}} \right\|_{\infty}+ \eps\sup_{0 \leq s \leq t}\left\|  w_\ell \frac{r(s)}{\sqrt{M_-}} \right\|^2_{\infty}
+ \eps^{4} \sup_{0 \leq s \leq t} \left\|  w_{\ell-1} \frac{\CG(s)}{\sqrt{M_-}}\right\|^2_{\infty}\\&+ \eps^2 \sup_{0 \leq s \leq t}\left\|  w_\ell \frac{\pa_tr(s)}{\sqrt{M_-}} \right\|^2_{\infty}
+ \eps^{5} \sup_{0 \leq s \leq t} \left\|  w_{\ell-1} \frac{\pa_t\CG(s)}{\sqrt{M_-}}\right\|^2_{\infty} +\sum\limits_{\al_0\leq1}\int_0^t\int_{\ga_-} \frac{|\pa_t^{\al_0}r|^2}{M_-} d\ga ds \\&
+\eps\sup\limits_{0\leq s\leq t}\int_{\Om\times \R^3}\frac{\nu_M^{-1}|\CG(s)|^2}{M_-}dxdv
+\eps\sum\limits_{\al_0\leq1}\int_0^t\int_{\Om\times \R^3}\frac{\nu_M^{-1}|\pa_t^{\al_0}\CG|^2}{M_-}dxdvds+\sup\limits_{0\leq s\leq t}\|\na_x[u,\ta]\|_{H^1}^2.
\end{split}
\end{equation}
In addition, we get from Lemma \ref{l3lem2} that
\begin{equation}\label{l3es4}
\begin{split}
\eps\sum\limits_{\al_0\leq1}&\left\|S_1^{\al_0}(R^{n+1})\right\|^2_{L^2_tL^3_x}\\
\lesssim&
\eps\sum\limits_{\al_0\leq1}\left\|\frac{\pa^{\al_0}_tR_0}{\sqrt{M_-}}\right\|^2_{L^2_{x,v}}
+\eps\sum\limits_{\al_0\leq1}\left\|\frac{v\cdot\na_x\pa^{\al_0}_tR_0}{\sqrt{M_-}}\right\|^2_{L^2_{x,v}}
\\&+
 \eps\sum\limits_{\al_0\leq1}\left\|\langle v\rangle^{-1}\frac{\pa_t^{\al_0}\CG}{\sqrt{M_-}}\right\|^2_{L^2_{t,x,v}}
 +\frac{1}{\eps}\sum\limits_{\al_0\leq1}\left\|\frac{P_1^M\pa_t^{\al_0}R^{n+1}}{\sqrt{M_-}}\right\|^2_{L^2_{t,x,v}}
 +\ka^2_0\eps\sum\limits_{\al_0\leq1}\left\|\frac{\pa_t^{\al_0}R^{n+1}}{\sqrt{M_-}}\right\|^2_{L^2_{t,x,v}}\\
 &+\eps\sum\limits_{\al_0\leq1}\left\|\frac{\pa_t^{\al_0}R^{n+1}}{\sqrt{M_-}}\right\|^2_{L^2(\R_+\times\ga_+)}
 +\sum\limits_{\al_0\leq1}\left\|\frac{\pa_t^{\al_0}r}{\sqrt{M_-}}\right\|^2_{L^2(\R_+\times\ga_-)}
 +\eps\sum\limits_{\al_0\leq1}\left|\frac{\pa^{\al_0}_tR_0}{\sqrt{M_-}}\right|^2_{2}.
\end{split}
\end{equation}
For the boundary term in the right hand side of \eqref{l3es4}, the trace Lemma \ref{traceth} implies
\begin{equation}\label{bdt}
\begin{split}
 \sum\limits_{\al_0\leq1}&\int_0^t\left |  \frac{P_{\gamma}\pa_t^{\al_0}R^{n+1}}{\sqrt{M_-}}\right|_{2,+}^{2}ds\\  \lesssim&   \sum\limits_{\al_0\leq1}\int_0^t\left| \mathbf{1}_{\gamma_{+}^{\delta}} \frac{P_{\gamma} \pa_t^{\al_0}R^{n+1}}{\sqrt{M_-}} \right|_{2,+}^{2}ds
 +\sum\limits_{\al_0\leq1}\int_0^t\left| \frac{(I- P_{\gamma})\pa_t^{\al_0} R^{n+1}}{\sqrt{M_-}} \right|_{2,+}^{2}ds\\
 \lesssim&  \sum\limits_{\al_0\leq1}\int_0^t\left| \mathbf{1}_{\gamma_{+}^{\delta}}  \frac{ \pa_t^{\al_0}R^{n+1}}{\sqrt{M_-}} \right|_{2,+}^{2}ds
 + \sum\limits_{\al_0\leq1}\int_0^t\left| \frac{(I- P_{\gamma})\pa_t^{\al_0} R^{n+1}}{\sqrt{M_-}} \right|_{2,+}^{2}ds
 \\
\lesssim & \sum\limits_{\al_0\leq1}\left\| \frac{\pa_t^{\al_0}R_0}{\sqrt{M_-}} \right\|_{2}^{2}+
\sum\limits_{\al_0\leq1}\int_0^t\left\| \frac{\pa_t^{\al_0}R^{n+1}}{\sqrt{M_-}} \right\|_{2}^{2}ds
 \\&+ \eps^{-1}\sum\limits_{\al_0\leq1}\int_0^t\left \|( L_M \pa_t^{\al_0}P_1^M R^{n+1}) \frac{\pa_t^{\al_0}R^{n+1}}{\sqrt{M_-}} \right\|_{1}ds+\sum\limits_{\al_0\leq1}\int_0^t\left| \frac{(I- P_{\gamma})\pa_t^{\al_0} R^{n+1}}{\sqrt{M_-}} \right|_{2,+}^{2}ds
  \\&+ \eps^{-1}\sum\limits_{\al_0\leq1}\int_0^t\left \| L_{\pa_tM} R^{n+1} \frac{\pa_t^{\al_0}R^{n+1}}{\sqrt{M_-}} \right\|_{1}ds+ \sum\limits_{\al_0\leq1}\int_0^t\left|\left\langle \pa^{\al_0}_tR^{n+1},\frac{\pa_t^{\al_0}\CG}{\sqrt{M_-}}\right\rangle\right| ds
 \\
\lesssim& \sum\limits_{\al_0\leq1}\left\| \frac{\pa_t^{\al_0}R_0}{\sqrt{M_-}} \right\|_{2}^{2}+
\sum\limits_{\al_0\leq1}\int_0^t\left\| \frac{\pa_t^{\al_0}R^{n+1}}{\sqrt{M_-}} \right\|_{2}^{2}ds
+\eps^{-1}\sum\limits_{\al_0\leq1}\int_0^t\left\| \frac{P_1^M\pa_t^{\al_0}R^{n+1}}{\sqrt{M_-}} \right\|_{2}^{2}ds
 \\& +\sum\limits_{\al_0\leq1}\int_0^t\left| \frac{(I- P_{\gamma})\pa_t^{\al_0} R^{n+1}}{\sqrt{M_-}} \right|_{2,+}^{2}ds
  + \eps\sum\limits_{\al_0\leq1}\int_0^t\left\| \frac{\nu_M^{-1/2}\pa_t^{\al_0}\CG}{\sqrt{M_-}}\right\|_{2}^{2}ds.
\end{split}
\end{equation}
Then, on the one hand, combing \eqref{eng2}, \eqref{l3es4} and \eqref{bdt}, we arrive at
\begin{equation}\label{eng3}
\begin{split}
\CE(R^{n+1})(t)&+\int_0^t\mathcal {D}(R^{n+1})(s)ds\\
\lesssim& \sum\limits_{\al_0\leq1}\int_{\Om\times \R^3}\frac{(\pa_t^{\al_0}R_0)^2}{M_-}dxdv
+\int_{\ga_+}\frac{|(I-P_\ga)R_0|^2}{M_-}d\ga+\eps^{-1}\int_{\Om\times \R^3}\frac{(P_1^MR_0)^2}{M_-}dxdv
\\&+\eps^2\left\|  w_\ell \frac{R_{0}}{\sqrt{M_-}} \right\|^2_{\infty}+\eps^3\left\|  w_\ell \frac{\pa_tR_{0}}{\sqrt{M_-}} \right\|_{\infty}+ \eps\sup_{0 \leq s \leq t}\left\|  w_\ell \frac{r(s)}{\sqrt{M_-}} \right\|^2_{\infty}
+ \eps^{4} \sup_{0 \leq s \leq t} \left\|  w_{\ell-1} \frac{\CG(s)}{\sqrt{M_-}}\right\|^2_{\infty}\\&+ \eps^2 \sup_{0 \leq s \leq t}\left\|  w_\ell \frac{\pa_tr(s)}{\sqrt{M_-}} \right\|^2_{\infty}
+ \eps^{5} \sup_{0 \leq s \leq t} \left\|  w_{\ell-1} \frac{\pa_t\CG(s)}{\sqrt{M_-}}\right\|^2_{\infty} +\sum\limits_{\al_0\leq1}\int_0^t\int_{\ga_-} \frac{|\pa_t^{\al_0}r|^2}{M_-} d\ga ds\\&+\eps\sum\limits_{\al_0\leq1}\int_0^t\int_{\Om\times \R^3}\frac{\nu_M^{-1}|\pa_t^{\al_0}\CG|^2}{M_-}dxdvds+\eps\sup\limits_{0\leq s\leq t}\int_{\Om\times \R^3}\frac{\nu_M^{-1}|\CG(s)|^2}{M_-}dxdv
\\&+\sup\limits_{0\leq s\leq t}\|\na_x[u,\ta]\|_{H^1}^2+\eps\sum\limits_{\al_0\leq1}\|\pa_t^{\al_0}[\rho,\na_x[u_0,\ta_0]]\|_{H^1}^2.
\end{split}
\end{equation}
On the other hand, Lemmas \ref{nopp1}, \ref{sob.ine}, Corollary \ref{inv.L.} and Theorem \ref{NSsol} give rise to
\begin{equation}\label{eng4}
\begin{split}
 \eps\sup_{0 \leq s \leq t}&\left\|  w_\ell \frac{r(s)}{\sqrt{M_-}} \right\|^2_{\infty}
+ \eps^2 \sup_{0 \leq s \leq t}\left\|  w_\ell \frac{\pa_tr(s)}{\sqrt{M_-}} \right\|^2_{\infty} +\sum\limits_{\al_0\leq1}\int_0^t\int_{\ga_-} \frac{|\pa_t^{\al_0}r|^2}{M_-} d\ga ds\\
\lesssim& \eps\sum\limits_{\al_0\leq1}\|\pa^{\al_0}_t\na^2_x[u,\ta]\|_{H^1_{co}}\|\pa^{\al_0}_t\na_x[u,\ta]\|_{H^2_{co}}
+\sum\limits_{\al_0\leq1}\int_0^t\|\pa^{\al_0}_t\na_x[u,\ta]\|_{H^1}^2ds\lesssim \ka_0^2\eps+\ka_0^2,
\end{split}
\end{equation}
and
\begin{equation}\label{eng5}
\begin{split}
\eps^{4} \sup_{0 \leq s \leq t} &\left\|  w_{\ell-1} \frac{\CG(s)}{\sqrt{M_-}}\right\|^2_{\infty}
+ \eps^{5} \sup_{0 \leq s \leq t} \left\|  w_{\ell-1} \frac{\pa_t\CG(s)}{\sqrt{M_-}}\right\|^2_{\infty} +\eps\sum\limits_{\al_0\leq1}\int_0^t\int_{\Om\times \R^3}\frac{\nu_M^{-1}|\pa_t^{\al_0}\CG|^2}{M_-}dxdvds\\
&+\eps\sup\limits_{0\leq s\leq t}\int_{\Om\times \R^3}\frac{\nu_M^{-1}|\CG(s)|^2}{M_-}dxdv\\
\lesssim& \CE^2(R^{n}(t))+\CE(R^{n}(t))\int_0^t\mathcal {D}(R^{n})(s)ds\\&+(\eps+\ka_0)\left(\CE(R^{n}(t))+\int_0^t\mathcal {D}(R^{n})(s)ds\right)
+\eps^2+\ka_0^2.
\end{split}
\end{equation}
Noticing that $R_0=-\eps^{-1/2}G_0$, one also can show that by \eqref{sobinep2} in Lemma \eqref{sob.ine} and Corollary \ref{inv.L.}
\begin{equation}\label{eng51}
\begin{split}
\int_{\ga_+}\frac{|(I-P_\ga)R_0|^2}{M_-}d\ga\lesssim \eps^{-1}\|\na_x[u_0,\ta_0]\|^2_{L^2(\pa\Om)}\lesssim\eps^{-1}\|\na^2_x[u_0,\ta_0]\|_2\|\na_x[u_0,\ta_0]\|_2\lesssim\ka_0^2,
\end{split}
\end{equation}
and
\begin{equation}\label{eng52}
\begin{split}
\eps^{-1}\int_{\Om\times \R^3}\frac{(P_1^MR_0)^2}{M_-}dxdv\lesssim\eps^{-2}\|\na_x[u_0,\ta_0]\|^2_2\lesssim\ka_0^2.
\end{split}
\end{equation}
We now get from \eqref{eng3}, \eqref{eng4}, \eqref{eng5}, \eqref{eng51} and \eqref{eng52} that
\begin{equation}\label{eng6}
\begin{split}
\FX_{\tilde{\de}}(R^{n+1})(t)\lesssim \CE(R_0)+\FX^2_{\tilde{\de}}(R^{n})(t)+(\eps+\ka_0)\FX_{\tilde{\de}}(R^{n})(t)+\eps^2+\ka_0^2+\ka_0^2\eps,
\end{split}
\end{equation}
which further implies $\FX_{\tilde{\de}}(R^{n+1})(t)<\de$ provided $R^{n}\in\FX_{\tilde{\de}}$  and $\CE(R_0)$, $\eps$ and $\ka_0$ are suitably small. Indeed $\CE(R_0)$ is small since $R_0=-\eps^{-1/2}G_0=-\eps^{-1/2}G(\rho_0,u_0,\ta_0)$ and
\begin{equation*}
\begin{split}
\eps\left\|w_{\ell}\frac{R_0}{\sqrt{M_-}}\right\|_{\infty}+\eps^{3/2}\left\|w_{\ell}\frac{\pa_tR_0}{\sqrt{M_-}}\right\|_{\infty}
+\sum\limits_{\al_0\leq1}\left\|\frac{\pa_t^{\al_0}R_0}{\sqrt{M_-}}\right\|_2\leq \ka_0\eps,
\end{split}
\end{equation*}
due to Lemma \ref{sob.ine}.

We as follows prove the strong convergence of the iteration sequence $\{R^{n}\}_{n=0}^{\infty}$ constructed above. To see this,
by taking difference of the equations that $R^{n +1}$ and $R^{n }$ satisfy, we deduce that%
\begin{eqnarray*}
\left\{\begin{array}{rll}
\begin{split}
&\partial_{t}[R^{n+1}-R^{n}]+v\cdot \nabla _{x}[R^{n
+1}-R^{n}]+\frac{1}{\eps}L_M[R^{n+1}-R^{n}]\\& \qquad=\eps^{1/2}\left\{Q (R^{n
}-R^{n -1},R^{n})+Q(R^{n -1},R^{n }-R^{n -1})\right\}\\&\qquad\quad+Q(R^{n}-R^{n-1},G)+Q(G,R^{n}-R^{n-1}), \\
&[R^{n +1}-R^{n }]_{-}=P_{\gamma }[R^{n +1}-R^{n }],
\end{split}
\end{array}\right.
\end{eqnarray*}%
with $R^{n +1}-R^{n }=0$ initially. Repeating the same argument as for obtaining \eqref{eng6}, we
obtain
\begin{equation*}%\label{Xfmin}
\FX_{\tilde{\de}}(R^{n+1}-R^{n})(t)\leq C\left\{\FX_{\tilde{\de}}(R^n)+\FX_{\tilde{\de}}(R^{n-1})+\eps^2\right\}\FX_{\tilde{\de}}(R^{n}-R^{n-1})(t).
\end{equation*}
This implies that $\{R^{n}\}_{n=0}^{\infty}$ is a Cauchy sequence in $\FX_{\tilde{\de}}$ for $\tilde{\de}$ suitably small. Moreover, take $R$ as the limit of the sequence $\{R^{n}\}_{n=0}^{\infty}$ in $\FX_{\tilde{\de}}$, then $R$ satisfies
\begin{equation*}%\label{sol.es}
\begin{split}
\CE(R)(t)+\int_0^t\mathcal {D}(R)(s)ds
\leq C\CE(R)(0)+C(\eps^2+\ka_0^2).
\end{split}
\end{equation*}
Finally, \eqref{diff} is valid due to $F-M=\eps G+\eps^{3/2}R$ and the uniqueness and positivity of $F$ follows the same argument as the proof of the
the Theorem 3 in \cite[pp.804]{Guo-2010},
this completes the proof of Theorem \ref{mr}.
\end{proof}

\section{Appendix}
In this final appendix, we collect some significant estimates used in the previous sections.

\begin{lemma}\label{est.nonop}
For any $p\in [1,+\infty]$ and $\mathbbm{a}%\bbalpha
\in[0,1]$, there exists a positive constant $C>0$ such
that
\begin{equation*}%\label{est.nonop.ine.}
\begin{split}
\left\|\frac{\nu_{\widetilde{M}}^{-\mathbbm{a}}(v)
Q(f,g)}{\sqrt{\widetilde{M}}}\right\|_{L^p_v}
\leq C\Bigg\{\left\|
\frac{\nu^{1-\mathbbm{a}}_{\widetilde{M}}(v)f}{\sqrt{\widetilde{M}}} \right\|_{L^p_v}
\left\|
\frac{g}{\sqrt{\widetilde{M}}} \right\|_{L^p_v}
+\left\|
\frac{\nu^{1-\mathbbm{a}}_{\widetilde{M}}(v)g}{\sqrt{\widetilde{M}}} \right\|_{L^p_v}
\left\|
\frac{f}{\sqrt{\widetilde{M}}} \right\|_{L^p_v}\Bigg\},
\end{split}
\end{equation*}
where $\widetilde{M}$ is any Maxwellian such that the above
integrals are well defined.
\end{lemma}
Lemma \ref{est.nonop} implies immediately the following significant estimates.
\begin{lemma}\label{nopp1}
It holds that
\begin{equation*}%\label{est.nonop.ine.}
\begin{split}
\left\|\frac{w_{\ell-1}
Q(f,g)}{\sqrt{\widetilde{M}}}\right\|_{\infty}
\leq C\left\|
\frac{w_\ell f}{\sqrt{\widetilde{M}}} \right\|_{\infty}
\left\|
\frac{w_\ell g}{\sqrt{\widetilde{M}}} \right\|_{\infty},
\end{split}
\end{equation*}
and for $\ell>-3/2$,
\begin{equation*}%\label{est.nonop.ine.}
\begin{split}
\left\|\frac{\langle v\rangle^{-1/2}
Q(f,g)}{\sqrt{\widetilde{M}}}\right\|_{2}
\leq C\left\{\left\|
\frac{w_{\ell} f}{\sqrt{\widetilde{M}}} \right\|_{\infty}
\left\|
\frac{\langle v\rangle^{1/2}g}{\sqrt{\widetilde{M}}} \right\|_{2}+\left\|
\frac{w_{\ell} g}{\sqrt{\widetilde{M}}} \right\|_{\infty}
\left\|
\frac{\langle v\rangle^{1/2}f}{\sqrt{\widetilde{M}}} \right\|_{2}\right\}.
\end{split}
\end{equation*}
Moreover, let $P_0^M R$ be given as \eqref{P0R} with $\rho, u$ and $\ta$ satisfying \eqref{apes}
then for
$\ta/2<1$, it also holds
\begin{equation*}
\begin{split}
&\left\|\frac{\langle v\rangle^{-1/2}
Q(P_0^M R,P_0^M R)}{\sqrt{M_-}}\right\|_{2}\\&\quad
\leq C\|[a,b,c]\|_6\left\|S_1^{0}(R)\right\|_3+C\|[a,b,c]\|_\infty\left\|\frac{P_1^MR}{\sqrt{M_-}}\right\|_2
+C\|[a,b,c]\|_2\left\|\frac{R_0}{\sqrt{M_-}}\right\|_{\infty},
\end{split}
\end{equation*}
\begin{equation}\label{noppt}
\begin{split}
&\left\|\frac{\langle v\rangle^{-1/2}
Q(\pa_tP_0^M R,P_0^M R)}{\sqrt{M_-}}\right\|_{2}+
\left\|\frac{\langle v\rangle^{-1/2}
Q(P_0^M R,\pa_tP_0^M R)}{\sqrt{M_-}}\right\|_{2}\\&\quad
\leq C\|[a,b,c]\|_6\left\|S_1^{1}(R)\right\|_3+C\|[a,b,c]\|_\infty\left\|\frac{P_1^M\pa_tR}{\sqrt{M_-}}\right\|_2
+C\|[a,b,c]\|_2\left\|\frac{\pa_tR_0}{\sqrt{M_-}}\right\|_{\infty}
\\&\qquad+
+C\|\pa_t[\rho,u,\ta]\|_{\infty}\|[a,b,c]\|_\infty\|[a,b,c]\|_2,
\end{split}
\end{equation}
\begin{equation*}
\begin{split}
&\left\|\frac{\langle v\rangle^{-1/2}
Q(P_0^M R,P_1^M R)}{\sqrt{M_-}}\right\|_{2}+
\left\|\frac{\langle v\rangle^{-1/2}
Q(P_1^M R,P_0^M R)}{\sqrt{M_-}}\right\|_{2}\\&\qquad
\leq C\left\|\frac{\langle v\rangle^{1/2} P_1^M R}{\sqrt{M_-}}\right\|_2\|[a,b,c]\|_\infty,
\end{split}
\end{equation*}
and
\begin{equation*}
\begin{split}
&\left\|\frac{\langle v\rangle^{-1/2}
Q(\pa_tP_0^M R,P_1^M R)}{\sqrt{M_-}}\right\|_{2}+
\left\|\frac{\langle v\rangle^{-1/2}
Q(P_1^M R,\pa_tP_0^M R)}{\sqrt{M_-}}\right\|_{2}\\&\qquad
\leq C\left\|\frac{\langle v\rangle^{1/2} P_1^M R}{\sqrt{M_-}}\right\|_2\|\pa_t[a,b,c]\|_\infty
+C\|\pa_t[\rho,u,\ta]\|_{\infty}\left\|\frac{\langle v\rangle^{1/2} P_1^M R}{\sqrt{M_-}}\right\|_2\|\pa_t[a,b,c]\|_\infty.
\end{split}
\end{equation*}
\end{lemma}
\begin{proof}
We only prove \eqref{noppt}, the proofs for the others above are similar.
From \eqref{pamim}, Lemma \ref{l3lem2} and Lemma \ref{est.nonop} with $\mathbbm{a}=1/2$, we see that
\begin{equation*}
\begin{split}
&\left\|\frac{\langle v\rangle^{-1/2}
Q(\pa_tP_0^M R,P_0^M R)}{\sqrt{M_-}}\right\|^2_{2}+
\left\|\frac{\langle v\rangle^{-1/2}
Q(P_0^M R,\pa_tP_0^M R)}{\sqrt{M_-}}\right\|^2_{2}\\&\quad
\lesssim \int_{\Om}\left\|\frac{\langle v\rangle^{1/2}\pa_tP_0^M R}{\sqrt{M_-}}\right\|^2_{L^2_v}\left\|\frac{\langle v\rangle^{1/2}P_0^M R}{\sqrt{M_-}}\right\|^2_{L^2_v}dx
\\&\quad
\lesssim \int_{\Om}|[a,b,c]\pa_t[\rho,u,\ta]|^2|[a,b,c]|^2dx+\int_{\Om}|\pa_t[a,b,c]|^2|[a,b,c]|^2dx
\\&\quad
\lesssim \int_{\Om}|[a,b,c]\pa_t[\rho,u,\ta]|^2|[a,b,c]|^2dx+\int_{\Om}\left[S_1^1(R)+S_2^1(R)+S_3^1(R)\right]^2|[a,b,c]|^2dx\\&\quad
\lesssim\|[a,b,c]\|^2_6\|S_1^1(R)\|^2_3
+\|[a,b,c]\|_\infty\left\|\frac{P_1^MR}{\sqrt{M_-}}\right\|^2_2
+\|[a,b,c]\|^2_2\left\|\frac{\pa_tR_0}{\sqrt{M_-}}\right\|^2_{\infty}
\\&\qquad+\|\pa_t[\rho,u,\ta]\|^2_{\infty}\|[a,b,c]\|^2_\infty\|[a,b,c]\|^2_2,
\end{split}
\end{equation*}
in view of $\dis{\int_{\R^3}}\frac{\langle v-u\rangle^\ell M_{[\rho,u,\ta]}^2}{M_-}dv<\infty$ for $\ell\geq0$ and $\ta/2<1.$
Thus the proof of Lemma \ref{nopp1} is completed.

\end{proof}

\begin{lemma}\label{co.est.}\cite[Lemma 4.2, pp.348]{Liu-Yang-Yu-Zhao-2006}
 If $\frac \theta 2<\widetilde{\theta}$, then there exist two positive constants
$\de_0=\de_0(u,\theta;\widetilde{u},\widetilde{\theta})$
and $\eta_0=\eta_0(u,\theta;\widetilde{u},\widetilde{\theta})$ such that
if $|u-\widetilde{u}|+|\theta-\widetilde{\theta}|<\eta_0$, we have for
$h(v)\in {\mathcal{N}}^\bot$,
\begin{eqnarray*}%\label{co.est.ine.}
-\int_{{\R}^3}\frac{ h L_{M} h}{\widetilde{M}}dv \geq
\de_0\int_{{\R}^3}\frac{\nu_{M}(v)h^2}{\widetilde{M}}dv.
\end{eqnarray*}
Here $M\equiv M_{[\rho,u,\theta]}(v)$, $ \widetilde{M}= \widetilde{M}_{[\widetilde{\rho},\widetilde{u},\widetilde{\theta}]}(v)$
and
$$
{\mathcal{N}}^\bot=\left\{ f(v):\ \ \int_{{\R}^3}\frac{\chi_i(v)f(v)}{M}dv=0,\ i=0,1,2,3,4\right\}.
$$
\end{lemma}

\begin{corollary}\label{inv.L.} Under the assumptions in Lemma \ref{co.est.}, we have
for  $h(v)\in {\mathcal{N}}^\bot$,
\begin{eqnarray*}%\label{inv.L.ine.}
{\displaystyle\int_{{\R}^3}}\frac{\nu_{M}(v)}{\widetilde{M}}\left|L^{-1}_{M}h\right|^2dv\leq
\de_0^{-2}{\displaystyle\int_{{\R}^3}}\frac{\nu_{M}^{-1}(v)h^2(v)}{\widetilde{M}}dv.
\end{eqnarray*}
\end{corollary}

\begin{lemma}\label{K}
Let $\rho, u$ and $\ta$ satisfy \eqref{apes}, then it holds that
\begin{equation}\label{k.es}
|K_{M,w}h|\leq\frac{C}{1+|v|}\|h\|_{\infty}.
\end{equation}
\end{lemma}

\begin{proof}
Recall
\begin{equation*}%\label{Kesp1}
\begin{split}
K_{M,w}h=&\frac{w_\ell}{\sqrt{M_-}}K_{M}\left(\frac{\sqrt{M_-}h}{w_\ell}\right)
=\frac{w_\ell\sqrt{M}}{\sqrt{M_-}}k_{M}\left(\frac{\sqrt{M_-}h}{\sqrt{M}w_\ell}\right),
\end{split}
\end{equation*}
where
\begin{equation*}
|k_M(v,v^{\prime })|\leq C\{|v-v^{\prime }|+|v-v^{\prime }|^{-1}\}e^{-%
\frac{1}{8\ta}|v-v^{\prime }|^{2}-\frac{1}{8\ta}\frac{||v-u|^{2}-|v^{\prime}-u|^{2}|^{2}}{|v-v^{\prime }|^{2}}}.  %\label{grad}
\end{equation*}
Performing the similar calculations as \cite[Lemma 3, pp.727]{Guo-2010}, one can further show that
\begin{equation*}
\int_{\R^3} \{|v-v^{\prime }|+|v-v^{\prime }|^{-1}\}e^{-\frac{1-\varepsilon }{8\ta}%
|v-v^{\prime }|^{2}-\frac{1-\varepsilon }{8\ta}\frac{||v|^{2}-|v^{\prime
}|^{2}|^{2}}{|v-v^{\prime }|^{2}}}\frac{w_\ell(v)e^{\frac{\pm\varrho|v|^2}{4\ta}}}{w_\ell(v^{\prime })e^{\frac{\pm\varrho|v'|^2}{4\ta}}}dv^{\prime}\leq \frac{C}{1+|v|}.  %\label{wk}
\end{equation*}
with $\varepsilon\geq0$, $\varrho\geq0$ and small enough.
Then \eqref{k.es} follows from
$$
\frac{\sqrt{M_-}}{\sqrt{M}}=\frac{\rho}{(2\pi)^{3/2}}e^{\frac{-\ta+1}{4\ta}|v|^2-\frac{v\cdot u}{2\ta}+\frac{u^2}{4\ta}},
$$
and $|\ta-1|\lesssim \ka_0\eps<\varrho$. This finishes the proof of Lemma \ref{K}.
\end{proof}

%\begin{equation}\label{Kesp1}
%\begin{split}
%K_{M,w}h=&\frac{w_\ell}{\sqrt{M_-}}K_{M_-}\left(\frac{\sqrt{M_-}h}{w_\ell}\right)
%+\frac{w_\ell}{\sqrt{M_-}}(K_M-K_{M_-})\left(\frac{\sqrt{M_-}h}{w_\ell}\right)
%\\=&w_\ell k_{M_-}\left(\frac{h}{w_\ell}\right)
%+\frac{w_\ell Q_{\textrm{gain}}(M-M_-,\frac{\sqrt{M_-}h}{w_\ell})}{\sqrt{M_-}}-\frac{w_\ell Q_{\textrm{loss}}(\frac{\sqrt{M_-}h}{w_\ell},M-M_-)}{\sqrt{M_-}}
%\\&+\frac{w_\ell Q_{\textrm{gain}}(\frac{\sqrt{M_-}h}{w_\ell},M-M_-)}{\sqrt{M_-}}.
%\end{split}
%\end{equation}
%Next, by a similar argument as that of \cite{}, one has on the one hand
%$$
%|w_\ell k_{M_-}\left(\frac{h}{w_\ell}\right)|\leq\frac{C}{1+|v|}\|h\|_{\infty}.
%$$
%On the other hand, the last three terms in the right hand side of \eqref{Kesp1} can be dominated by
%$$C\ka_0\eps\langle v\rangle\|h\|_{\infty},$$
%due to
%Lemma \ref{est.nonop}, \eqref{apes} and Sobolev's inequality. This completes the proof of Lemma \ref{K}.
The following lemma is cited from \cite[Lemma 3.2, pp.37]{EGKM-15}.
\begin{lemma} \label{traceth}
Define
\begin{equation*}%\label{non_grazing}
\gamma_{\pm}^{\delta} : =
\{ (x,v) \in \gamma_{\pm} :  | n(x)\cdot v | > \delta, \ \ \delta\leq  |v| \leq 1/\de,\ \de>0  \}.
 \end{equation*}
For and $T>0,$ if
$f \in L^{1} ([0, T] \times \Omega \times  \mathbb{R}^{3})$, it holds that
 \begin{eqnarray*}
 \int^{T}_{0} \int_{\gamma_{+}^{\delta }} | f(t,x,v)|d \gamma d t
&\lesssim & \iint_{\Omega \times \mathbb{R}^{3}}  | f(0,x,v)  |  d v d x
+  \int^{T}_{0} \iint_{\Omega \times \mathbb{R}^{3}}   | f(t,x,v) | d v d x d t%\label{trace_d}
\\
& &+    \int^{T}_{0} \iint_{\Omega \times \mathbb{R}^{3}}
 \big| [  \partial_{t} f +  v\cdot \nabla_{x} f ](t,x,v) \big|
d v d x d t .   \nonumber
 \end{eqnarray*}
 \end{lemma}
The following Lemma is concerned with the classical Sobolev imbedding inequalities.
\begin{lemma}\label{sob.ine.lem}
If $f\in H^1(\Om)$, then it follows for $2\leq p\leq6$
$$\|f\|_{p}\leq C\|f\|_{H^1(\Om)}.$$
\end{lemma}
Next, we
borrow the following anisotropic Sobolev embedding and trace estimates from \cite[Proposition 2.2, pp.316]{Masmoudi-Rousset-2017}.
\begin{lemma}\label{sob.ine}
	Let  $m_1\geq 0,~m_2\geq 0$ be integers, $f\in H^{m_1}_{co}(\Om)$  and   $\nabla f\in H^{m_2}_{co}(\Om)$.\\
Then
{\small	\begin{equation}\label{sobinep1}
	\|f\|^2_{L^\infty}\leq C \|\nabla f\|_{H^{m_2}_{co}}\|f\|_{H^{m_1}_{co}},
	\end{equation}}
	provided $m_1+m_2\geq 3$, and
{\small	\begin{equation}\label{sobinep2}
	|f|^2_{H^{s}(\partial\Om)}\leq C\|\nabla f\|_{H^{m_2}_{co}}\|f\|_{H^{m_1}_{co}},
	\end{equation}}
	for $m_1+m_2\geq 2s\geq 0$.
\end{lemma}
We now give a Helmoholtz decomposition result cf. \cite[Theorem 2, pp.71]{SB}, which is essentially used in the study of the compressible Navier-Stokes system with Dirichlet boundary condition.
\begin{lemma}[Imposing a boundary condition for divergence free] \label{Hdec}
Let
$\Om\in\R^3$
be a simply connected bounded Lipschitz domain with a smooth boundary. For every vector
field $f\in L^2(\Om;\R^3)$, there exist $\varphi:\Om\rightarrow\R$ and $\psi:\Om\rightarrow\R^3$ such that
$$
f=\na_x\varphi+\na_x\times\psi,\ \varphi\in H^1(\Om;\R),\ \na_x\cdot\psi=0,\ n\cdot(\na_x\times\psi)|_{\pa\Om}=0,
$$
moreover, there exists a constant $C_H > 0$ such that
\begin{align*}
\|\varphi\|_{H^1}+\|\psi\|_{H^1}\leq C_H\|f\|_{2}.
\end{align*}
\end{lemma}
Based on the proof of Lemma \ref{Hdec}, we can show the following $H^{k}_{co}$ boundedness of the divergence free projection operator $P_2$, cf. \cite{BJP}[Lemmas 1, 2, pp.343].
\begin{lemma}\label{P1lem}
Let $k>0$ be any integer,
for every vector
field $f\in H^{k+1}(\Om;\R^3)$, then it holds that
\begin{align}\label{P1bd}
\sum\limits_{|\al|\leq k}\| Z^\al P_2f\|_{H^1}\lesssim
\sum\limits_{|\al|\leq k}\|Z^{\al} f\|_{H^1},
\end{align}
where $P_2$ is the divergence free projection operator defined as in Section \ref{sec2}.
\end{lemma}
\begin{proof}
By Lemma \ref{Hdec}, we have $P_1f=\na_x\varphi$, where $\varphi$ satisfies the Neumann problem
\begin{eqnarray}\label{vp.eq}
\left\{
\begin{array}{rll}
\Delta_x\varphi=&\na_x\cdot f, \ x\in\Om,\\[2mm]
\frac{\pa \varphi}{\pa n}=&f\cdot n, \ x\in\pa\Om.
\end{array}\right.
\end{eqnarray}
Acting $Z^\al$ with $|\al|=1,2,\cdots$ to \eqref{vp.eq}, one has
\begin{eqnarray*}\label{Zvp.eq}
\left\{
\begin{array}{rll}
\Delta_xZ^\al\varphi=&\na_x\cdot Z^\al f+[Z^\al,\na_x\cdot]f-[Z^\al,\Delta_x]\varphi, \ x\in\Om,\\[2mm]
\frac{\pa Z^\al\varphi}{\pa n}=&Z^\al(f\cdot n)-\sum\limits_{|\al'|\leq |\al|-1}C_{\al}^{\al'}Z^{\al'}\na_x\varphi\cdot Z^{\al-\al'}n, \ x\in\pa\Om.
\end{array}\right.
\end{eqnarray*}
Then it follows from \cite{BJP}[Lemmas 1, pp.343] and the trace estimates \eqref{sobinep2} that
\begin{align}\label{Hdecbd}
\|\na_xZ^\al\varphi\|_{H^1}\leq& C(\|\na_x\cdot Z^\al f\|_2+\|[Z^\al,\na_x\cdot]f\|_2+\|[Z^\al,\Delta_x]\varphi\|_2)
\notag\\&+C(|Z^\al f\cdot n|_{H^{\frac{1}{2}}}+{\bf 1}_{\al>0}|Z^{\al'}\na_x\varphi\cdot Z^{\al-\al'}n|_{H^{\frac{1}{2}}})\notag\\
\leq& C\sum\limits_{\al'\leq \al}\|Z^{\al'}f\|_{H^1}+C{\bf 1}_{\al>0}\sum\limits_{\al'<\al}\|\na_xZ^\al\varphi\|_{H^1}.
\end{align}
Next, taking a linear combination of \eqref{Hdecbd} for $|\al|=1,2,\cdots, k$, we obtain
\begin{align*}
\sum\limits_{|\al|\leq k}\|\na_x Z^\al \varphi\|_{H^1}\lesssim
\sum\limits_{|\al|\leq k}\|Z^{\al} f\|_{H^1},
\end{align*}
on the other hand, since
\begin{align*}
\sum\limits_{|\al|\leq k}\| Z^\al\na_x\varphi\|_{H^1}\lesssim& \sum\limits_{|\al|\leq k}\|\na_x Z^\al \varphi\|_{H^1}+\sum\limits_{|\al|\leq k}\| [Z^\al,\na_x]\varphi\|_{H^1}\notag\\
\lesssim& \sum\limits_{|\al|\leq k}\|\na_x Z^\al \varphi\|_{H^1}+\sum\limits_{|\al'|\leq k-1}\| Z^{\al'}\na_x\varphi\|_{H^1}
\lesssim \sum\limits_{|\al|\leq k}\|Z^{\al} f\|_{H^1},
\end{align*}
this together with $Z^\al f=Z^\al \na_x\varphi+Z^\al\na_x\times\psi$ implies \eqref{P1bd}, thus the proof of Lemma \ref{P1lem} is finished.
\end{proof}

%\noindent\textbf{Proof}. The proof is just a using of  the covering $\Om\subset \Om_0\cup_{k=1}^n\Om_k$ and Proposition 2.2 in \cite{Masmoudi-R-1}, the details are   omitted here.  $\hfill\Box$

%The following anisotropic imbedding results  can be found in the appendix of \cite{Gie-Kelliher}.
%\begin{proposition}\label{prop2.3-11}
%	Let $\Omega$ be a $C^{m+2}$ domain with  $m\geq 3$. Suppose that $f$ and $\nabla f$ lie in space $H^m_{co}$, then it holds that
%	\begin{equation}
%	\|f\|_{L^\infty(\Gamma_a)}\leq C(m,a)\|f\|^{\f12-\f1{2m}}_{L^2}\cdot\|f\|^{\f1{2m}}_{H^m_{co}}\cdot[\|f\|_{L^2}+
%	\|\nabla f\|_{H^m_{co}}]^{\f12}
%	\end{equation}
%	and
%	\begin{equation}
%	\|f\|_{L^\infty(\Omega-\Gamma_a)}\leq C(m,a)\|f\|^{1-\f3{2m}}_{L^2}\cdot\|f\|^{\f3{2m}}_{H^m_{co}},
%	\end{equation}
%	where $\Gamma_a=\{x\in\Omega~|~\mbox{dist}(x,\partial\Omega)\leq a\}$ with  $a>0$ being a small fixed constant.
%\end{proposition}

\medskip

\noindent {\bf Acknowledgements:}
RJD was supported by the General Research Fund (Project No.~14301719) from RGC of Hong Kong and a direct grant from CUHK. SQL was
supported by grants from the National Natural Science Foundation of China (contracts: 11971201 and 11731008). SQL would thank Prof. Yan Guo for the discussion on the topic. The authors also thank the anonymous referee for improving the presentation of the manuscript, in particular on the proof related to the Helmholtz decomposition.

\end{document}